\documentclass[%
 reprint,
 amsmath,amssymb,
 aps,
]{revtex4-2}

\usepackage{graphicx}
\usepackage{dcolumn}
\usepackage{bm}

\usepackage{subcaption}
\usepackage{tikz}
\usetikzlibrary{positioning}
\usetikzlibrary{svg.path}
\usepackage{amsthm}
\newtheorem{theorem}{Theorem}
\newtheorem{lemma}{Lemma}
\newtheorem{conjecture}{Conjecture}
\newtheorem{assumption}{Assumption}

\newcommand{\df}[2]{\frac{\text{d} \, #1}{\text{d} #2}}

\newcommand{\E}{\text{E}}

\begin{document}

\preprint{APS/123-QED}

\title{Transition from time-variant to static networks: \\timescale separation in the N-Intertwined Mean-Field Approximation of Susceptible-Infectious-Susceptible epidemics}
\author{Robin Persoons}
\altaffiliation{Faculty of Electrical Engineering, Mathematics and Computer Science, Delft University of Technology}
\altaffiliation{Corresponding Author: r.d.l.persoons@tudelft.nl}
\author{Mattia Sensi}%
\altaffiliation{MathNeuro Team, Inria at Universit\'e C\^ote d'Azur}
\altaffiliation{Department of Mathematical Sciences ``G. L. Lagrange'', Politecnico di Torino}
\author{Bastian Prasse}
\altaffiliation{European Centre for Disease Prevention and Control (ECDC)}
\author{Piet Van Mieghem}
\altaffiliation{Faculty of Electrical Engineering, Mathematics and Computer Science, Delft University of Technology}



\date{\today}

\begin{abstract}
We extend the N-Intertwined Mean-Field Approximation (NIMFA) for the Susceptible-Infectious-Susceptible (SIS) epidemiological process to time-varying networks. Processes on time-varying networks are often analysed under the assumption that the process and network evolution happen on different timescales. This approximation is called timescale separation. We investigate timescale separation between disease spreading and topology updates of the network. We introduce the transition times $\mathrm{\underline{T}}(r)$ and $\mathrm{\overline{T}}(r)$ as the boundaries between the intermediate regime and the annealed (fast changing network) and quenched (static network) regimes, respectively, for a fixed accuracy tolerance $r$. By analysing the convergence of static NIMFA processes, we analytically derive upper and lower bounds for $\mathrm{\overline{T}}(r)$. Our results provide insights/bounds on the time of convergence to the steady state of the static NIMFA SIS process. We show that, under our assumptions, the upper-transition time $\mathrm{\overline{T}}(r)$ is almost entirely determined by the basic reproduction number $R_0$ of the network. The value of the upper-transition time $\mathrm{\overline{T}}(r)$ around the epidemic threshold is large, which agrees with the current understanding that some real-world epidemics cannot be approximated with the aforementioned timescale separation.
\end{abstract}

\maketitle


\section{Introduction}

The modelling of infectious disease spreading has been central in mathematical biology for almost a century \cite{anderson1991infectious}. Proper models of disease spreading are fundamental in describing, forecasting and controlling the evolution of epidemics. One of the most common assumptions is \emph{homogeneous mixing}, which implies that any individual in the population is equally likely to encounter (and infect or get infected by) any other individual. Although homogeneous mixing greatly simplifies the analysis of the models, it is rather unrealistic, because individuals in real populations have more contacts with, for example, their family, friends and colleagues, which suggests that contacts are heterogeneous. Such heterogeneous scenarios can be modelled as networks, in which nodes represent individuals and links represent contacts.

A vast majority of papers on epidemics in networks focuses on \emph{static} networks \cite{pastor2015epidemicreview,nowzari2016analysis,mei2017dynamics}. However, real-world contact networks are evolving, as individuals move frequently and encounter different groups (family, colleagues, commuters on public transport, etc.). In particular, contact networks evolve during an epidemic as well and changes in the contact network play a crucial role on the spread of epidemics \cite{holme2012temporal,holme2015modern}. 

The analysis of epidemics on time-variant networks is often under the assumption of timescale separation: either the network changes significantly faster than the spread of the disease or the epidemic evolves significantly faster than the topology updates of the network. Timescale separation of the network and the epidemic results in three regimes (Fig. \ref{fig:tikzschematischeoverview}). The three regimes correspond to the network topology updates being much faster (annealed regime), comparable (intermediate regime) or much slower (quenched regime) than the spread of the disease. 

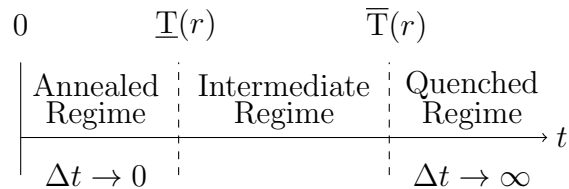
\begin{figure}[ht]
    \centering
\begin{tikzpicture}[path image/.style={
    path picture={
    \node at (path picture bounding box.center) {
    \includegraphics[height=3cm]{#1}
    };}}]

    \draw[->] (0,0) -- (7,0);
    \draw[-] (0,-0.5) -- (0,1);  
    \draw[dashed] (2.1,-0.5) -- (2.1,1);
    \draw[dashed] (4.9,-0.5) -- (4.9,1);
    \node at (6,-0.5) {\large $\Delta{t} \rightarrow \infty$};
    \node at (1,-0.5) {\large $\Delta{t} \rightarrow 0$};
    \node at (1,0.7){\large Annealed};
    \node at (1,0.3) {\large Regime};
    \node at (3.5,0.7){\large Intermediate};
    \node at (3.5,0.3) {\large Regime};
    \node at (6,0.7){\large Quenched};
    \node at (6,0.3) {\large Regime};
    \node at (5,1.5) {\large $\mathrm{\overline{T}}(r)$};
    \node at (2.2,1.5) {\large $\mathrm{\underline{T}}(r)$};
    \node at (0,1.5) {\large 0};
    \node at (7.2,0) {\large $t$};
\end{tikzpicture}  
    \caption{Illustration of the timescale separation. The time between two changes in the graph $\Delta t$ is increasing on the $t$-axis. The transition times $\mathrm{\underline{T}}(r)$ and $\mathrm{\overline{T}}(r)$ create the borders of the annealed and quenched regimes by including processes that are approximately annealed or quenched to the respective regimes, based on the accuracy tolerance $r$. The annealed regime is bounded from below by $\Delta{t} = 0$. The intermediate regime lies between the transition times $\mathrm{\underline{T}}(r)$ and $\mathrm{\overline{T}}(r)$. The quenched regime extends from $\Delta{t} = \mathrm{\overline{T}}(r)$ to infinity.}
    \label{fig:tikzschematischeoverview}
\end{figure}

In the \emph{quenched} regime, the network changes very slowly compared to the evolution of the epidemic. Quenched processes are well approximated with processes on static networks and therefore well studied over the last two decades. Indeed, as we will show in Section \ref{sec:Problemstatement}, the quenched regime assumes that, before each topology update, the epidemic has almost reached its equilibrium.

In the \emph{annealed} regime, the epidemic evolves very slowly compared to the network. The epidemic spreads as on an ``average'' network. General results for the annealed regime show that the annealed process shares attributes with the static process on the edge-average graph \cite{kohar2013emergence,valdano2018epidemic,zhang2017spectral}. Additional results in the annealed regime have been derived under the degree-based mean-field theory by Pastor-Satorras and Vespignani \cite{pastor2001epidemic,pastor2001epidemic2,schwarzkopf2010epidemic,li2012susceptible,devriendt2017unified}.  

The \emph{intermediate} regime, that resides between the quenched and annealed regimes, is difficult to analyse. However, it is considered to be the most important in real world scenarios \cite{holme2015modern,holme2014birth,leitch2019toward,perra2012activity,stehle2011simulation}. Diseases spread at a timescale comparable to the timescale of the movement of individuals.

In this work, we introduce the lower and upper-transition times $\mathrm{\underline{T}}(r)$ and $\mathrm{\overline{T}}(r)$ as the boundaries between the three timescale regimes in Fig. \ref{fig:tikzschematischeoverview}. In particular, we provide analytical bounds on the upper-transition time $\mathrm{\overline{T}}(r)$, which indicates whether the contact network can be considered approximately static. As a function of the accuracy tolerance $r$ these transition times indicate to which regime a process with a specific time between topology updates (inter-update time) $\Delta{t}$ belongs. The dependence on an accuracy tolerance $r$ is required, because the annealed regime for small inter-update times $\Delta{t}$ and the quenched regime for large inter-update times $\Delta{t}$ only describe the exact behaviour for $\Delta{t} \rightarrow 0$ and $\Delta{t} \rightarrow \infty$, respectively. Indeed, only approximately quenched and approximately annealed processes exist outside of these limits. The accuracy tolerance $r$ allows us to extend the quenched and annealed regimes to include, per definition, these approximately quenched and annealed processes. When the error due to the quenched or annealed approximation does not exceed the accuracy tolerance $r$, the process is considered quenched or annealed. These transition times are the boundaries of the intermediate regime (Fig. \ref{fig:tikzschematischeoverview}), when an error due to timescale separation of at most $r$ is allowed.  

Epidemics on time-varying networks are studied analytically in \cite{schwarzkopf2010epidemic,leitch2019toward,perra2012activity,pare2015stability,ogura2016stability,pare2017multi,pare2017epidemic} and results based on simulations and analysis of real-world data are found in \cite{stehle2011simulation,lieberman2005evolutionary,prakash2010virus,kotnis2013stochastic,ren2014epidemic,vestergaard2015temporal,nadini2018epidemic,guo2021impact,han2023impact}. However, to the best of the authors' knowledge, our quantification of the boundaries in Fig. \ref{fig:tikzschematischeoverview} is novel.

The paper is structured as follows. First, we briefly recall the N-Intertwined Mean-Field Approximation (NIMFA) SIS process \cite{VanMieghem2008Virusspreadinnetworks} in Section \ref{sec:MarkovianandNIMFASIS}. In Section \ref{sec:NIMFAontimevariantnetworks}, we present our extension of NIMFA to time-varying networks. In Section \ref{sec:Problemstatement}, we discuss the timescale regimes and timescale separation in more depth. In Section \ref{sec:timeregimesoftemporalnetworkepidemics}, we formally introduce the upper-transition time $\mathrm{\overline{T}}(r)$ and show numerical results on the upper-transition time $\mathrm{\overline{T}}(r)$. In Section \ref{sec:bounds}, we derive upper and lower bounds on the upper-transition time $\mathrm{\overline{T}}(r)$ and numerically compare them with results from Section \ref{sec:timeregimesoftemporalnetworkepidemics}. We conclude in Section \ref{sec:summary}.

\subsection{The Markovian and NIMFA SIS processes} \label{sec:MarkovianandNIMFASIS}
We consider the homogeneous continuous-time Markovian SIS process on a static network, represented by a simple undirected contact graph $G(N,L)$, with corresponding $N\times N$ symmetric adjacency matrix $A$, where $a_{ij} = 1$ if nodes $i$ and $j$ are connected and $a_{ij} = 0$ otherwise \cite{VanMieghem2010graphspectra}. A simple graph has no self-loops and thus $a_{ii} = 0$ for all nodes $i$. These requirements on the adjacency matrix hold for all graphs considered in this paper.

The Markovian SIS process has states specified by Bernoulli random variables $X_{i}\in \{0,1\}$ for each node $i$. If $X_{i}=1$ the node $i$ is \textit{infected}, else the node $i$ is \textit{healthy} but \textit{susceptible}. At a time $t$, the node $i$ is infected with probability $v_{i}(t) = \Pr[X_{i}(t)=1]$ and healthy with probability $1-v_{i}(t)$. We assume that the infection attempts from an infected node $i$ to a healthy node $j$ are Poisson processes with infection rates $\beta$. Each node also has a Poisson curing (or recovery) process with curing (or recovery) rate $\delta$. The effective infection rate is defined as $\tau = \frac{\beta}{\delta}$. The vector of all $v_i(t)$ is denoted as $V(t) = [v_1(t) \; v_2(t) \dots v_N(t)]^{T}$. We define the \textit{prevalence} $y(t)$, which is the average fraction of infected nodes:
\begin{equation}\label{eq:prevalance}
    y(t) = \frac{1}{N}\sum_{i=1}^{N}v_{i}(t).
\end{equation}
The SIS model exhibits a \textit{phase transition} at the epidemic threshold $\tau_c$. If $\tau < \tau_{c}$, the epidemic dies out exponentially fast \cite{ganesh2005effect}. If $\tau > \tau_c$, the epidemic lasts very long \footnote{Technically, there is a non-zero probability that the epidemic dies out fast for $\tau > \tau_c$ in the Markovian SIS model, because the probability of curing before infecting anyone is non-zero for any $\tau$.} \cite{VanMieghem2013decaytowardstheoverallhealthystate}. For any effective infection rate $\tau$ the epidemic will eventually die out \cite{vanmieghem2020explosive}, because the overall healthy state $V(t) = 0$ is the only steady state of the Markovian SIS process. The fact that the Markov state $X_i(t)$ is a Bernoulli random variable, for which $\E[X_i] = \Pr[X_i = 1]$ holds, leads to a differential equation for the infection probability of node $i$, first proposed in \cite{cator2012second}:
\begin{eqnarray}\label{eq:markovianSIS}
    \df{\E[X_{i}]}{t} = \E\bigg[-\delta X_i +\beta(1-X_i)\sum_{j=1}^{N}a_{ij}X_j\bigg] \nonumber\\
    = -\delta\E[X_i] + \beta \sum_{j=1}^{N}a_{ij}\E[X_j] - \beta \sum_{j=1}^{N}a_{ij}\E[X_i X_j].
\end{eqnarray}
The joint probability $\E[X_iX_j] = \Pr[X_i = 1, X_j = 1]$ is remarkably complicated \cite{cator2012second,van2014performance}. Instead, one can consider the N-Intertwined Mean-Field Approximation (NIMFA) \cite{VanMieghem2008Virusspreadinnetworks,VanMieghem2011Nintertwined} for the SIS model. NIMFA replaces in the Markovian SIS process the random variable $X_i$ with the expectation $\E[X_i]$ and $\E[X_iX_j]$ reduces to $\E[X_i]\E[X_j]$. Then, the governing NIMFA equations for a static network are given by the differential equations \cite{lajmanovich1976deterministic}
\begin{equation}\label{NIMFASIS}
    \df{v_{i}(t)}{t} = - \delta v_{i}(t) + \beta ( 1 - v_{i}(t))\sum_{j=1}^{N}a_{ij}v_{j}(t),  \quad i=1,2,\dots,N.
\end{equation}
The phase transition occurs in NIMFA SIS at the first-order NIMFA epidemic threshold $\tau_c^{(1)}= \frac{1}{\lambda_{1}(A)}$, where $\lambda_{1}(A)$ is the largest eigenvalue of the adjacency matrix $A$. Additionally, the NIMFA epidemic threshold lower bounds the Markovian SIS epidemic threshold $\tau_c^{(1)} < \tau_c$ for all networks \footnote{NIMFA upper bounds the infection probability in the Markovian SIS process \cite{VanMieghem2008Virusspreadinnetworks}. Specifically, when $\tau_c > \tau > \tau_{c}^{(1)}$ NIMFA will not die out, but a Markovian SIS process will die out fast.}. The NIMFA steady-state prevalence is denoted as $y_\infty$ and the infection probability of node $i$ in the steady state is denoted as $v_{\infty,i}$. NIMFA has a non-trivial steady state ($y_{\infty} \neq 0$) for $\tau > \tau_c^{(1)}$, which corresponds to the metastable (or quasi-stationary) state in the Markovian SIS process. Therefore, an analysis of the NIMFA steady state allows insights into the metastable state of the Markovian SIS process. When $V(0) \neq 0$ and $V(0) \neq V_{\infty} = [v_{\infty,1}(t) \; v_{\infty,2}(t) \dots v_{\infty,N}(t)]^{T}$, the NIMFA process will tend \cite{lajmanovich1976deterministic,khanafer2016stability},  for $t\rightarrow\infty$, to either the steady state $V_{\infty} = 0$ (if $\tau \leq \tau_c^{(1)}$) or to an upper bound for the Markovian SIS metastable state $V_{\infty} > 0$ (if $\tau > \tau_c^{(1)}$). Strictly speaking, the probability vector $V(t)=0$ is also a steady state for $\tau > \tau_c^{(1)}$. However, if $\tau > \tau_c^{(1)}$, then we only denote the non-negative vector $V_{\infty}>0$ as the steady state.

\subsection{NIMFA on time-variant networks}\label{sec:NIMFAontimevariantnetworks}
The topology of a time-variant network is represented at time $t$ by a simple undirected contact graph $G(t) = G(N,L(t))$. At all times $t$, the network has the same $N$ nodes, but the amount of links $L(t)$ may vary. The contact graph $G(t)$ is represented by its $N\times N$ symmetric adjacency matrix $A(t)$. 

We denote with $t_m$ the occurrence time of the $m$-th topology change. Within the interval $[t_{m-1}, t_m)$, with length $T_m = t_{m}-t_{m-1}$, the network remains static. 
We denote the graph $G(t)$ during the interval $t_{m-1}\leq t < t_{m}$ by $G_{m}$ and similarly the adjacency matrix $A(t) = A_m$ for $t_{m-1}\leq t < t_{m}$, with elements $a_{ij}(t) = (a_m)_{ij}$. Denoting the total amount of topologies by $M$, the complete interval $[t_{0},t_{M}]$ has exactly $M-1$ topology updates and the sequences of graphs and adjacency matrices are denoted by:
\begin{equation*}
    \mathcal{G} := \{G_{1},\dots,G_{M}\}, \quad
    \mathcal{A} := \{A_1,\dots,A_M\}.
\end{equation*}
We will call the times $\{t_0,\dots,t_{M}\}$ the \textit{update times} and the lengths of the intervals $[t_{m-1},t_{m})$, namely the set $\{T_1,\dots,T_M\}$, the \textit{inter-update times}. 
In a time-variant network, the NIMFA epidemic threshold equals $\tau_c^{(1)}(G_m) = \frac{1}{\lambda_1(A_{m})}$, which varies with the topology $m$. For each graph $G_m$ or adjacency matrix $A_{m}$ and for a fixed effective infection rate $\tau$, we define the basic reproduction number 
\begin{equation}\label{eq:basicrepnum}
(R_0)_m = R_0 (G_m,\tau)  := \frac{\tau}{\tau_c^{(1)}(G_m)} = \frac{\beta}{\delta}\lambda_1(A_{m}).    
\end{equation}
The phase transition coincides with $R_0 = 1$, which follows from substituting $\tau = \tau_c^{(1)}(G_m)$ in \eqref{eq:basicrepnum}. We say that $R_0(t) = (R_0)_m$ if $t_{m-1} \leq t < t_{m}$. The governing NIMFA equations for the entire time period $[t_0,t_{M}]$ for a node $i$ are given by:
\begin{widetext}
\begin{equation}\label{eq:nimfatemporal}
\arraycolsep=1.4pt\def\arraystretch{2.2}
\left\lbrace\begin{array}{ll}
\dfrac{\text{d}v_i(t)}{\text{d}t}=-\delta v_i(t) + \beta(1-v_i(t))\sum_{j=1}^{N} (a_1)_{ij} v_j(t), &\qquad t \in [t_0,t_1), \\
\qquad\vdots & \\
\dfrac{\text{d}v_i(t)}{\text{d}t}=-\delta v_i(t) + \beta(1-v_i(t))\sum_{j=1}^{N} (a_M)_{ij} v_j(t), &\qquad t \in [t_{M-1},t_M).
\end{array}\right.
\end{equation}
\end{widetext}
In terms of the probability vector $V(t) = [v_1(t) \; v_2(t) \dots v_N(t)]^T$, the matrix representation of (\ref{eq:nimfatemporal}) on each interval $[t_{m-1},t_m)$ is
\begin{equation}\label{sys_matrix}
    \dfrac{\text{d}V(t)}{\text{d}t} = (\beta A_m-\delta I)V(t) -\beta \text{diag}(V(t))A_mV(t),
\end{equation}
where $I$ denotes the $N\times N$ identity matrix and diag($x$) is the diagonal matrix with the elements $x_i$ of the $N\times1$ vector $x$ on the diagonal. Since the derivative of $v_i(t)$ exists on the entire interval $[t_0,t_M)$, we know that $v_{i}(t)$ is continuous on $[t_0,t_M)$. The starting conditions at each update time $v_i(t_m)$ are the limits $v_i(t_m - \varepsilon)$ as $\varepsilon \to 0^+$. It follows from (\ref{eq:nimfatemporal}) that the derivative is, in general, discontinuous at $t_m$: $$
\df{v_{i}}{t}\Big|_{t=t_m - \varepsilon} \neq \df{v_{i}}{t}\Big|_{t=t_m},
$$ 
in the limit $\varepsilon \to 0^+$ if $A_{m} \neq A_{m+1}$.

We will, without loss of generality, take $\delta = 1$, because for any $\delta > 0$, we can rescale the time variable in units of the average curing time $\frac{1}{\delta}$. Indeed, we replace the time $t$ with $\theta = \delta t$ in (\ref{NIMFASIS}), while keeping the rates the same in the new time units:
\begin{equation}\label{NIMFASISdelta1}
    \df{v_i}{\theta} =  -v_i(\theta)+ \tau(1-v_{i}(\theta))\sum_{j=1}^{N}a_{ij}v_j(\theta),
\end{equation}
which are the ``rescaled'' NIMFA governing equations. The same method can be applied to the system (\ref{eq:nimfatemporal}), where an equivalent system with $\delta =1$ is found. The intervals $[t_{m-1},t_m)$ of the ``rescaled'' system are measured in units of the average curing time $\frac{1}{\delta}$. In the following, we will use \eqref{NIMFASISdelta1} instead of \eqref{NIMFASIS} and we write the dimensionless time $t$ instead of $\theta$ when using \eqref{NIMFASISdelta1} for clarity. In the rescaled NIMFA governing equations \eqref{NIMFASISdelta1}, the effective infection rate $\tau$ equals the infection rate $\beta$ because $\delta = 1$.

\subsection{Timescale separation and the transition times}\label{sec:Problemstatement} 
In this section, we explore the interplay between the timescale of the epidemic process and the timescale of the topology updating process. We assume at first that the inter-update times $T_m$ are constant and equal to $\Delta{t}$. The timescales of the epidemic process are characterized by the average infection time $\frac{1}{\beta}$ between infection attempts on links and the average curing time $\frac{1}{\delta}$. Fig. \ref{fig:timeregimesplot} shows the prevalence $y(t)$ of three temporal NIMFA processes, that correspond to the three regimes in Fig. \ref{fig:tikzschematischeoverview}. Three processes with the same infection rate $\beta$, the same curing rate $\delta$ and random Erd\H{o}s-R\'enyi \footnote{An Erd\H{o}s-R\'enyi random graph (ER graph) $G_p(N)$ is characterized by the link between each pair of the $N$ nodes existing with probability $p$, independent of any other link (see, e.g. \cite{van2014performance}).} contact graphs $G_m$ with the same distribution, but with different inter-update times $\Delta{t}$ are illustrated. The solid red line in Fig. \ref{fig:timeregimesplot} shows the averaging behavior of the annealed regime. The dotted blue line illustrates the convergence to equilibrium on each network topology of the quenched regime. The dashed black line shows the irregular process of the intermediate regime.

\begin{figure*}[ht]
    \centering
    \includegraphics[width=\textwidth]{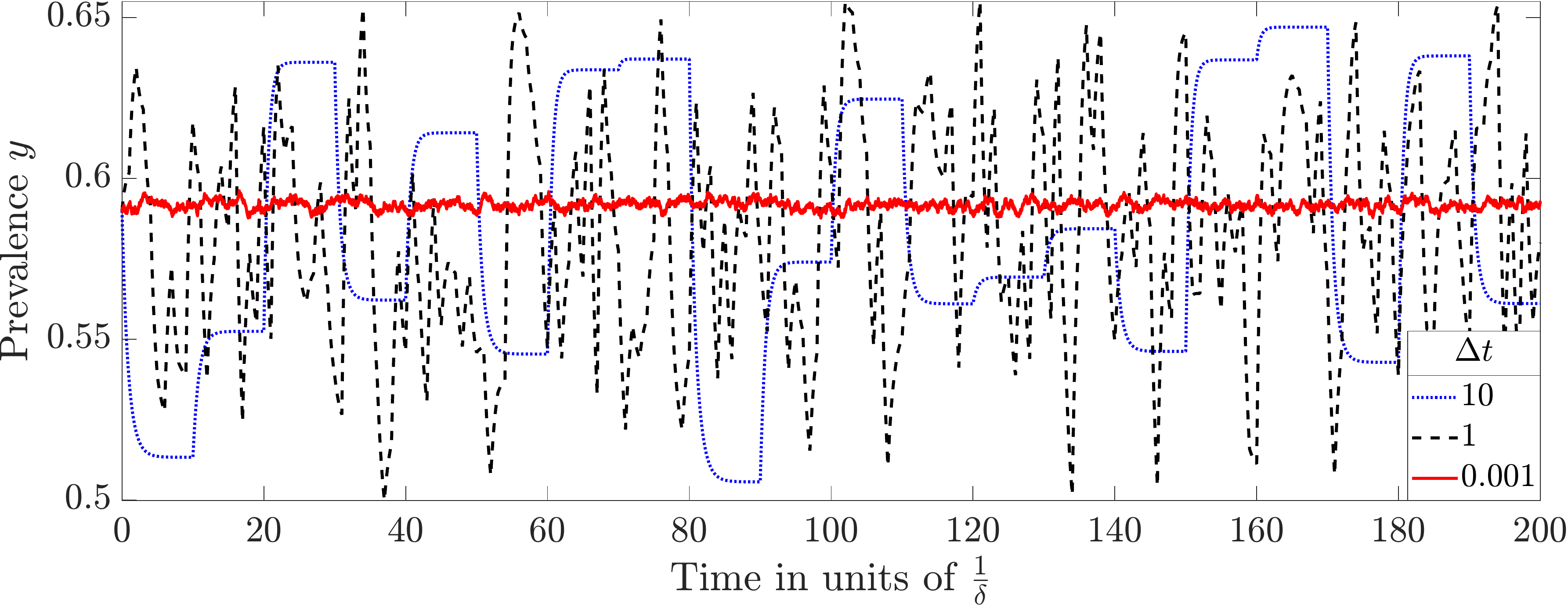}
    \caption{NIMFA SIS on a time-varying network for different inter-update times $\Delta{t}$. Parameters: graph size $N =50$; infection rate $\beta =0.1$; curing rate $\delta = 1$. The time-varying network is given by a sequence of Erd\H{o}s-R\'enyi graphs, where the link density $p$ is chosen uniformly at random from the interval $[0.4,0.6]$. The solid red line (small inter-update time $\Delta{t}=0.001$) shows the averaging behavior when nearing the annealed regime. The dashed black line (medium inter-update time $\Delta{t}=1$) shows the irregularity of the intermediate regime. The dotted blue line (large inter-update time $\Delta{t}=10$) shows the convergence on each topology from the quenched regime. }
    \label{fig:timeregimesplot}
\end{figure*}

We introduce the transition times $\mathrm{\underline{T}}(r)$ and $\mathrm{\overline{T}}(r)$ as functions of a graph $G$, the infection rate $\beta$, the curing rate $\delta$ and the accuracy tolerance $r$. The transition times or the boundaries of the regimes $\mathrm{\underline{T}}(r)$ and $\mathrm{\overline{T}}(r)$ are only meaningful as a function of the accuracy tolerance $r$. The accuracy tolerance $r$ identifies temporal processes that can be approximated well with quenched or annealed processes. Here, the accuracy tolerance $r$ can be lowered to increase the accuracy of the approximation. Specifically, we classify any process which has $T_m \leq \mathrm{\underline{T}}(r;G_m)$ for all $m=1,\dots,M$ as annealed and any process which has $T_m \geq \mathrm{\overline{T}}(r;G_m)$ for all $m=1,\dots,M$ as quenched. In this work, we restrict ourselves to an analysis of the upper-transition time $\mathrm{\overline{T}}(r)$. We will also restrict ourselves to NIMFA SIS. We argue that our results using NIMFA are a good approximation of the physical Markovian SIS model. We compare the NIMFA prevalence to the expected/average prevalence of a Markovian SIS process \footnote{The Markovian SIS process average is conditioned on non-extinction} with the same parameters on the same sequence of graphs at different inter-update times $\Delta t$, which is shown in Fig. \ref{fig:markovnimfacomparison}. Since the processes are similar for all inter-update times, we claim that our mean-field results are representative of Markovian SIS epidemics on temporal networks.

\begin{figure*}[ht]
    \centering
    \begin{subfigure}[]{0.32\textwidth}
        \centering
        \includegraphics[width=\linewidth]{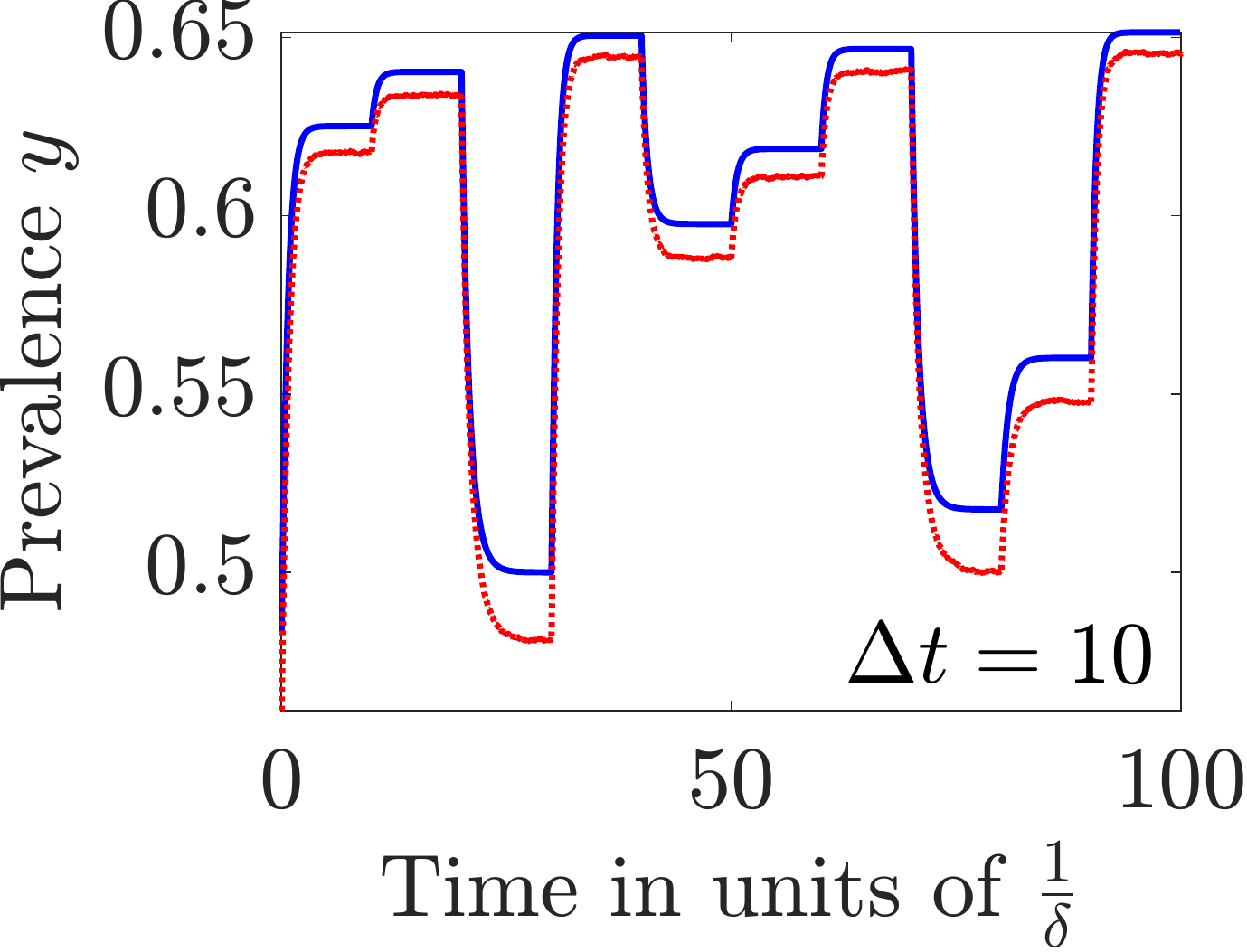} 
    \end{subfigure}
    \begin{subfigure}[]{0.32\textwidth}
        \centering
        \includegraphics[width=\linewidth]{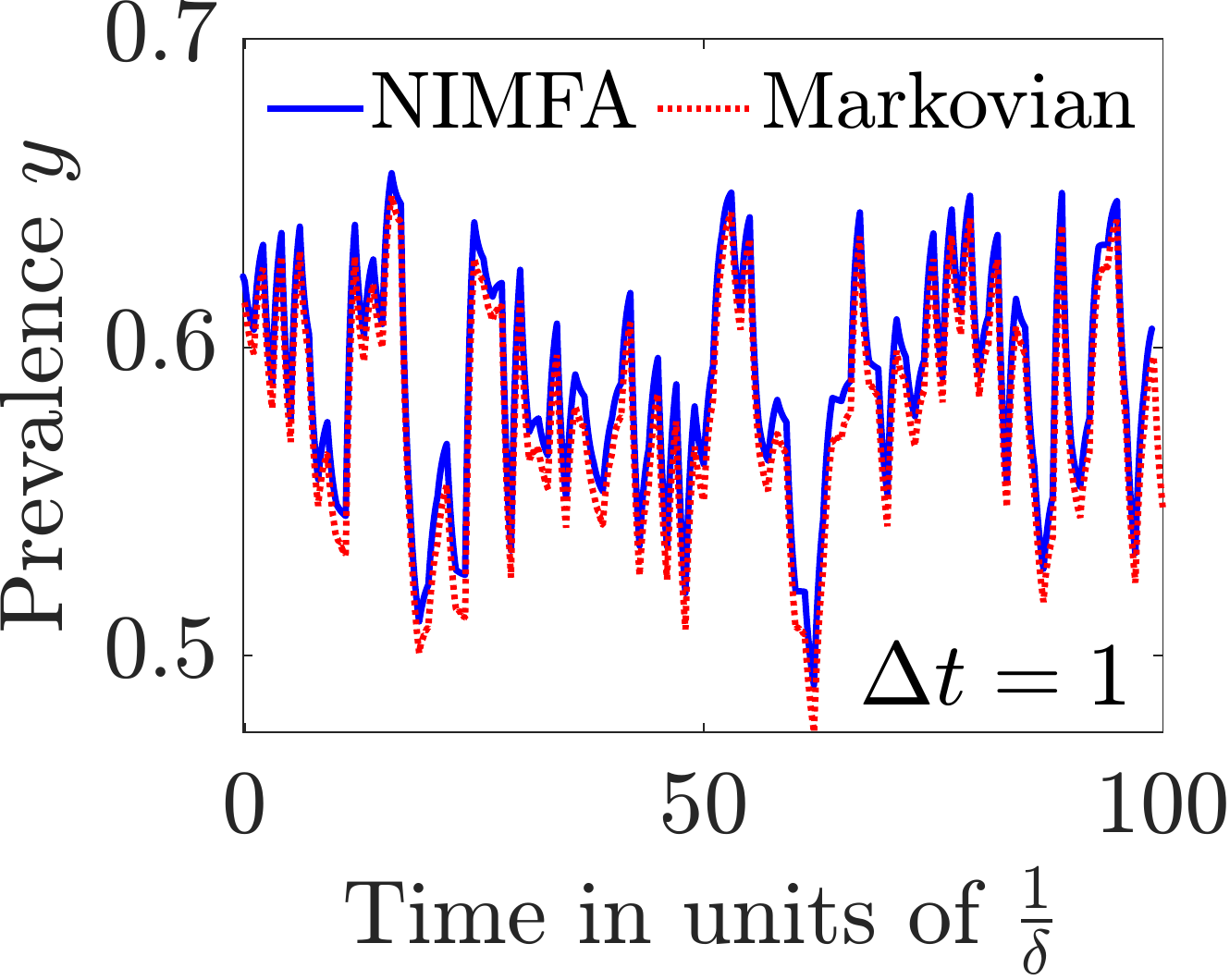} 
    \end{subfigure}
    \begin{subfigure}[]{0.32\textwidth}
        \centering
        \includegraphics[width=\linewidth]{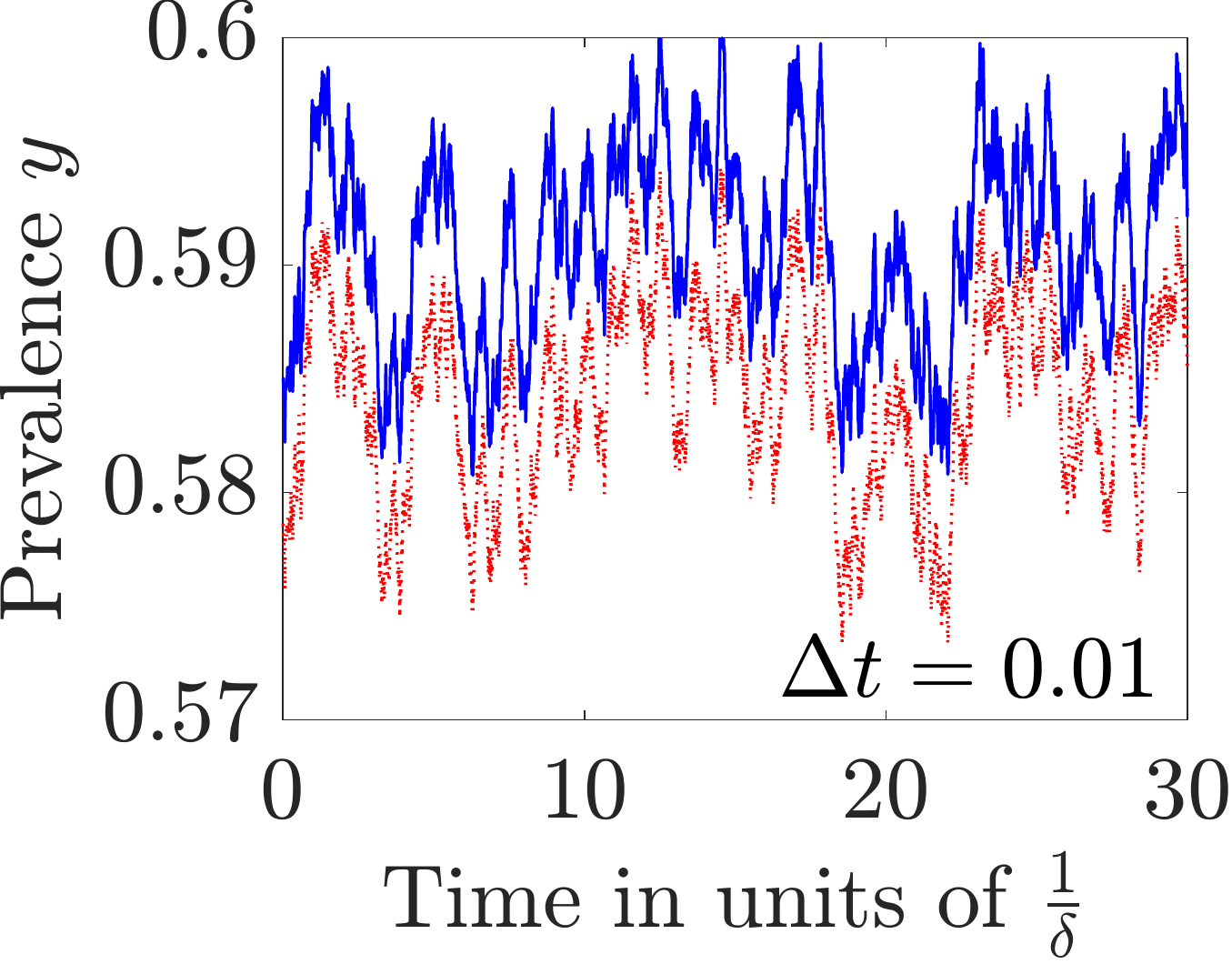} 
    \end{subfigure}
    \caption{Comparison of NIMFA SIS and Markovian SIS on time-varying networks for different inter-update times $\Delta{t}$. Left: $\Delta{t}=10$, Middle: $\Delta{t}=1$, Right: $\Delta{t}=0.01$. The solid blue line corresponds to the NIMFA process and the dashed red line to the Markovian process.  Parameters are the same for NIMFA and Markovian. The infection rate $\beta =0.1$ and curing rate $\delta = 1$. The time-varying network is given by a sequence of Erd\H{o}s-R\'enyi graphs, with $N=50$ and link density $p$ picked uniformly at random from the interval $[0.4,0.6]$. The graph sequence for NIMFA and Markovian SIS are the same in each sub-figure, but the sequences are different between sub-figures.}
    \label{fig:markovnimfacomparison}
\end{figure*}

In order to define a well-posed problem, we will make the following simplifying assumptions:

\begin{enumerate}
    \item The number of nodes $N$ and the number of topologies $M$ are fixed and finite.
    \item We consider the mean-field NIMFA SIS process instead of the Markovian SIS process. 
    \item We confine ourselves to homogeneous SIS with link infection rates $\beta_{ij} = \beta$ for all links $i\sim j$ and nodal curing rates $\delta_i = \delta$ for all nodes $i$.
    \item In our numerical analysis all graphs $G_m$ are Erd\H{o}s-R\'enyi graphs $G(N,p)$, where the link density $p$ is a uniformly distributed random variable. We do not require that the ER-graphs are connected, because real-world time-varying contact networks can be disconnected. Appendix \ref{app:assumptionofERgraphs} discusses whether ER-graphs are representative of general graphs. In more realistic temporal networks, the graphs would plausibly change less at each update time, which would result in shorter transition times. Investigating specific temporal processes is beyond the scope of this paper.
    \item We have simulated a range of values of the number of nodes $N$ and infection rate $\beta$. With the exception of some disconnected graphs, the behaviour of the transition time was not influenced significantly by either $N$ or $\beta$ in these simulations. Graphs with similar basic reproduction number $R_0$ will have similar upper-transition times, even if $N$ and $\beta$ differ. However, for large values of $N$, ER-graphs become increasingly more regular. Since changing $N$ and $\beta$ does not influence the results for a given $R_0$, all simulations shown in this work have $N=50$ nodes and infection rate $\beta=0.1$; where $\beta$ is chosen to have a good range of possible $R_0$ values. The analytical bounds hold for general $N$ and $\beta$. 

\end{enumerate}

\section{The upper-transition time} \label{sec:timeregimesoftemporalnetworkepidemics}

The temporal NIMFA SIS process gains a Markov-like property when the inter-update times $T_m$ tend to infinity. When the inter-update times $T_m$ increase, the graphs $\{G_k\}_{k<m}$ have a decreasing influence on the state vector at the next update time $V(t_m)$. When the influence of the graphs $\{G_k\}_{k<m}$ becomes negligible, the process becomes approximately memoryless to the network updates: the current graph $G_m$ determines the probability vector at the next network update $V(t_m)$ almost completely. The Markov-like memorylessness arises from the fact that, in the quenched limit of $T_m \to \infty$, the static NIMFA SIS process will tend to its steady-state vector $V_{\infty}$ from any starting point $V(0)\neq 0$. Therefore, when $T_m$ increases, the difference between the state vector $V(t_{m})$ at the end of the interval $[t_{m-1},t_{m})$ and the steady-state vector $V_{\infty}$ decreases for all $V(t_{m-1})$, when $V(t_{m-1}) \neq 0$. For large inter-update times $T_m$, the state vector at the next update time $V(t_m) \approx V_{\infty}(G_m)$ (which implies $y(t_m) \approx y_{\infty}(G_m)$), independent of all previous graphs $\{G_k\}_{k<m}$. In Fig. \ref{fig:TemporalMarkovProperty}, it is shown that the prevalence at the end of the second interval $y(t_2)$ is approximately independent of the first graph $G_1$ and depends (almost) solely on $G_2$. \\

Since the state vector $V(t)$ for $t \in [t_m,t_{m+1})$ is completely determined by the starting state $V(t_{m})$ and the graph $G_{m+1}$, the memorylessness of the state vector $V(t_m)$ extends to the entire interval $[t_m,t_{m+1})$. 
When $T_{m-1}$ tends to infinity, $V(t_{m-1})$ is determined only by the graph $G_{m-1}$. Therefore, the state vector $V(t)$ and prevalence $y(t)$ on the interval $[t_{m-1}, t_m)$ are fully determined by the graphs $G_{m-1}$ and $G_m$. Similar to the starting state $V(t_m)$, the state vector $V(t)$ on the interval $[t_{m-1}, t_m)$ is approximately memoryless for large inter-update times $T_m$. Fig. \ref{fig:TemporalMarkovProperty} shows the Markov-like memorylessness of the prevalence $y(t)$ on the interval $[t_2, t_3)$, which is (almost) independent of the first graph $G_1$.\\

\begin{figure}[ht]
    \centering
    \includegraphics[width=0.45\textwidth]{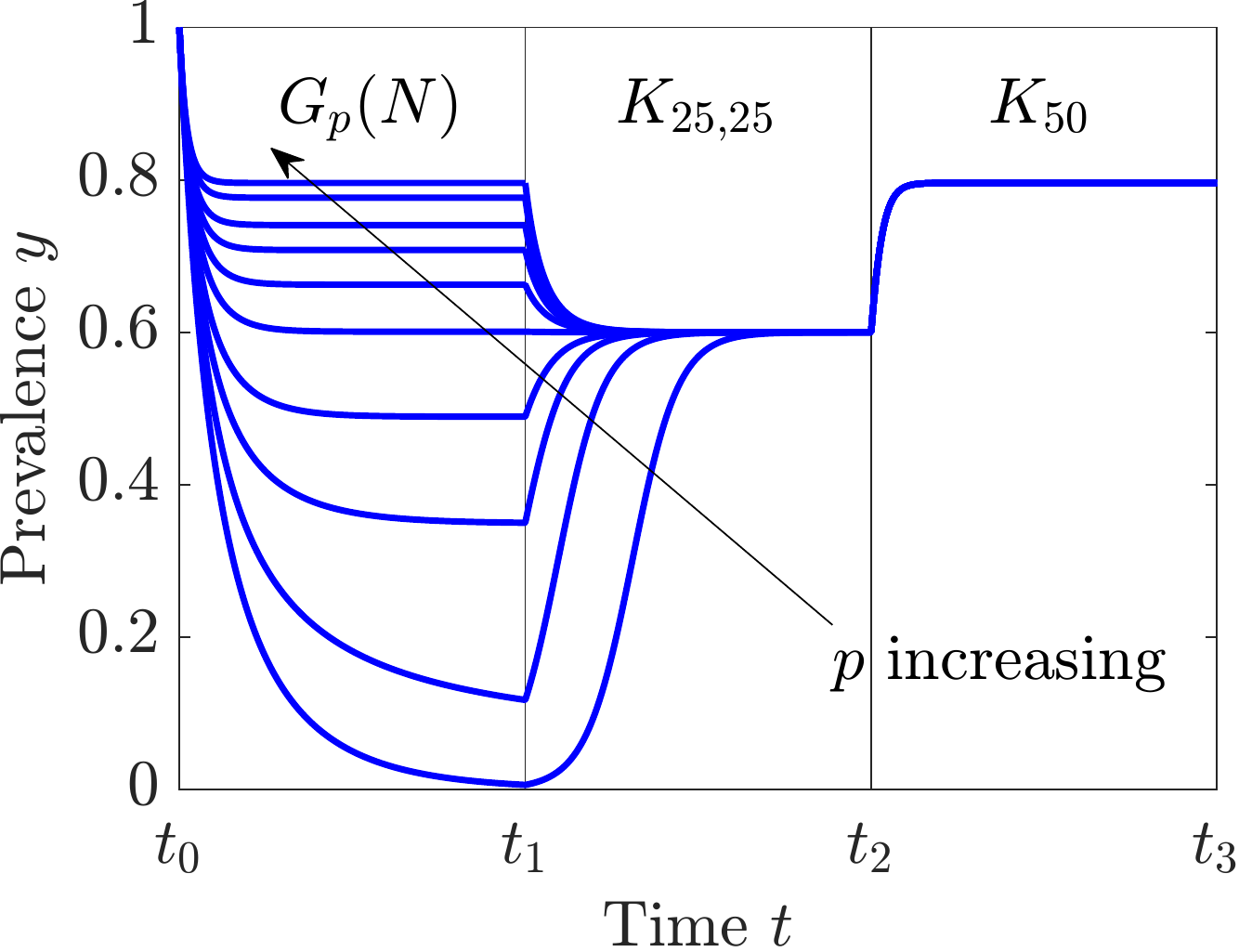}
    \caption{Ten different NIMFA SIS processes on time-varying networks with long inter-update times (i.e. $t_k - t_{k-1}$ is large). The first graph $G_1$ on the interval $[t_0, t_1)$, is different for each process: the different graphs are ER graphs with $N=50$ and link density $p=\{0.1,0.2,0.3,0.4,0.5,0.6,0.7,0.8.0.9,1\}$. The second graph  $G_2$ on the interval $[t_1,t_2)$ is the same for each process and is the complete bipartite graph $K_{25,25}$. The final graph $G_3$ on the interval $[t_2,t_3)$ is also the same for each process and is the complete graph on $50$ nodes $K_{50}$. The graph $G_1$ has visibly negligible influence on the interval $[t_2, t_3)$ when $G_2$ and $G_3$ are fixed due to the Markov-like memorylessness property with respect to previous graphs that defines the quenched regime.}
    \label{fig:TemporalMarkovProperty}
\end{figure}

When $(R_0)_m > 1$ in the quenched regime, the steady-state vector $V_{\infty}(G_m)$ approximates the infection probability vector $V(t)$ at the end of the $m$-th interval $V(t_m)$. The case of $(R_0)_m \leq 1$ should be treated carefully, because the steady state $V_{\infty} = 0$ describes an epidemic which is extinct. Hence, when $(R_0)_m<1$ in the quenched regime, we should approximate the entries of the infection probability vector at the topology update as $V_i(t_{m}) \geq \epsilon$, with a small $0<\varepsilon \ll 1$, because in NIMFA $V(t) = 0$ is only actually reached in the limit as $t\to \infty$ and would mean that the epidemic process stops.

The Markov-like attributes that appear at large inter-update times $T_m$ are what quantify a process as ``quenched''. We investigate whether we can specify the inter-update time $T_m = \mathrm{\overline{T}}(r)$, as a function of the effective infection rate $\tau$, the graph $G_m$ and an accuracy tolerance $r$, such that for $T_m \geq \mathrm{\overline{T}}(r,G_m,\tau)$, the NIMFA SIS process will reach the steady-state prevalence $y_{\infty}$ on the interval $[t_{m-1},t_m)$ with an error $|y_{\infty}- y(t_m)| \leq r$, i.e. at most the accuracy tolerance $r$. Then, at the update time $t_m$, the process has converged and has lost its dependence on the initial condition with respect to the accuracy tolerance $r$. The time $\mathrm{\overline{T}}(r,G_m,\tau)$ is called the \emph{upper-transition time}. In short, we write $\mathrm{\overline{T}}(r) = \mathrm{\overline{T}}(r,G_m,\tau)$.  \\
An intuitive first definition of the upper-transition time $\mathrm{\overline{T}}(r)$ is the smallest time $t$ such that the error $|y(t) - y_{\infty}|\leq r$ for all starting values $v_i(0) \in (0,1]$ for all nodes $i$. This definition implies that the steady state is reached up to the accuracy tolerance $r$. However, for small starting infection probability vectors $||V(0)||_1 = \varepsilon \ll 1$ and graphs $G$ with basic reproduction number $R_0(G,\tau) > 1$, the time of convergence to reach $|y(t) - y_{\infty}| \leq r$ can be made arbitrarily long by selecting a sufficiently small $\varepsilon$, due to Lemma \ref{lemma:arbitrarilyslowfromzero}, which is proven in Appendix \ref{app:lemmaproof}. 
\begin{lemma}\label{lemma:arbitrarilyslowfromzero}
    Given a graph $G$, an effective infection rate $\tau > \tau_c^{(1)}$, an arbitrary large time $\mathcal{T}$ and an accuracy tolerance $r<y_{\infty}$, there exists a starting infection probability vector $V(0) = \varepsilon(r) u > 0$, where $u$ is the all-one vector, such that $|y(t)-y_{\infty}| > r$ 
 for all $t\leq \mathcal{T}$. 
\end{lemma}

 To ensure that the upper-transition time $\mathrm{\overline{T}}(r)$ is finite, we bound the infection probabilities $v_i(0) \geq r$ with the accuracy tolerance $r$ and define the upper-transition time $\mathrm{\overline{T}}(r)$ for a fixed effective infection rate $\tau$ and a fixed graph $G$ as 
\begin{equation}\label{eq:deftbar}
\mathrm{\overline{T}}(r) = \text{argmin}_{t \geq 0} \left\{|y(t) - y_{\infty}| \leq r \; \forall V(0) \geq ru \right\},
\end{equation} 
where $y(t)$ indicates the NIMFA SIS prevalence on the graph $G$ and we write $\forall V(0) \geq ru$ as shorthand for ``for all $V(0)$ such that $v_{i}(0)\in [r,1]$ for all nodes $i$ in $G$''. The definition of the upper-transition time $\mathrm{\overline{T}}(r)$ in \eqref{eq:deftbar} ensures that the steady state $y_{\infty}$ is reached by the time $t=\mathrm{\overline{T}}(r)$, provided that the initial state $V(0)$ is not too small and satisfies $V(0) \geq ru$. To stress the dependence on the underlying, fixed graph $G_m$, we also denote the upper-transition time as $\mathrm{\overline{T}}(r, G_m)$. \\
\\

For NIMFA on temporal networks, we emphasise that $T_m \geq \mathrm{\overline{T}}(r, G_m)$ for all graphs $G_m \in \mathcal{G}$ does \emph{not} imply that $|y(t_m) - y_{\infty}| < r$ for every graph $G_m$, even if the initial state satisfies $V(0) \geq r u$. The underlying reason is that $V(0) \geq ru$ does not imply $V(t_m) \geq ru$ for all graphs $G_m$. For instance, the graph $G_1$ may correspond to a basic reproduction number $(R_0)_1 \leq 1$, due to which the viral state $V(t)$ converges to $0$ as $t \rightarrow \infty$. Hence, when the graph $G_1$ changes to $G_2$ at time $t_1$, it is possible that the viral state vector obeys $V(t_1) < r u$. Then the prevalence $y(t)$ of the temporal process decreases below $r$ and the prevalence at each following graph update $y(t_m)\Big|_{t_m > t}$ is no longer guaranteed to be within the accuracy tolerance $r$ around the steady-state prevalence $y_{\infty}$ even though $T_m \geq \mathrm{\overline{T}}(r,G_m)$ for all graphs $G_m \in \mathcal{G}$ (i.e. even though we are in the approximately quenched regime). While it would be preferable that $T_m \geq \mathrm{\overline{T}}(r,G_m)$ for all graphs $G_m \in \mathcal{G}$ implies $|y(t_{m}) - y_{\infty}| \leq r$ for all topologies $m$ and all starting infection probability vectors $V(0)$, Lemma \ref{lemma:arbitrarilyslowfromzero} shows that a lower bound on the starting infection probabilities $v_i(0)$ is necessary for the upper-transition time $\mathrm{\overline{T}}(r)$ to be finite. We argue that not considering processes where the prevalence $y(t)$ has dropped below $y(t)=r$ is a reasonable choice when interpreting $y(t) \leq r$ as a die-out. NIMFA has no actual die-out in finite time, a characteristic that can be considered nonphysical or unrealistic. Since NIMFA SIS is an approximation of Markovian SIS conditioned on non-extinction in the first place, we make the assumption that no ``NIMFA die-outs'' occur.\\

In the following, we will assume that $T_1 = T_2 = \dots = T_M = \Delta t$. Only in this specific case, the requirement $T_m \geq \mathrm{\overline{T}}(r,G_m)$ for all graphs $G_m \in \mathcal{G}$ becomes $\Delta t \geq \max_{G_m \in \mathcal{G}}(\mathrm{\overline{T}}(r,G_m))$.  The upper-transition time $\mathrm{\overline{T}}(r,\mathcal{G})$ of the graph sequence $\mathcal{G}$ can now be defined as $\max_{G_m \in \mathcal{G}}(\mathrm{\overline{T}}(r,G_m))$.

\subsection{Numerical results}
The upper-transition time  $\mathrm{\overline{T}}(r)$ can only be determined approximately, because there is no simple closed formula for the steady-state prevalence $y_{\infty}$ for general graphs. Therefore, we numerically solve the NIMFA equations \eqref{NIMFASISdelta1} to calculate the steady state prevalence $y_{\infty}$. We investigate an heuristic for the upper-transition time $\mathrm{\overline{T}}(r)$ that does not require $y_{\infty}$ in Appendix \ref{app:derivativeconvegencetime}.

Fig. \ref{fig:tbarzoomed} shows the numerical approximation of the upper-transition time $\mathrm{\overline{T}}(r)$, with accuracy tolerance $r = 10^{-4}$, for different values of the basic reproduction number $R_0$ and the starting prevalence $y(0)$. Specifically, each node has the same starting infection probability $v_i(0) = y(0)$, for different values of $y(0)$. The sharp lines in Fig. \ref{fig:tbarzoomed} show that the upper-transition time $\mathrm{\overline{T}}(r)$ is almost fully determined for each graph by the basic reproduction number $R_0$. Additionally, Fig. \ref{fig:tbarzoomed} illustrates the asymptotic behavior around the epidemic threshold at $R_0 = 1$. The asymptote splits the figure in two regimes: $R_0 \leq 1$ and $R_0 > 1$. Fig. \ref{fig:tbarzoomed} shows that, if $R_0 \leq 1$, then the curves (from top to bottom) are in decreasing order of $y(0)$. Since, with a higher initial prevalence $y(0)$, the process must decay more to reach the threshold of $|y(t) - y_{\infty}| = y(t) <r$. 
Therefore, Fig. \ref{fig:tbarzoomed} shows that the upper-transition time $\mathrm{\overline{T}}(r)$ is zero for $y(0) = 10^{-4}$ (blue line) below the epidemic threshold. For $R_0 > 1$, the curves (from top to bottom) are in increasing order of $y(0)$, if we ignore the dips. The inset in Fig. \ref{fig:tbarzoomed} shows the dips below the curves in more detail. The dips occur when the starting prevalence $y(0)$ is close to the steady-state prevalence $y_{\infty}$ for the corresponding values of the basic reproduction number $R_0$. The upper-transition time $\mathrm{\overline{T}}(r)$ is small, because the initial difference $|y(0)-y_{\infty}|$ is small.

\begin{figure*}[ht]
    \centering
    \includegraphics[width=\textwidth]{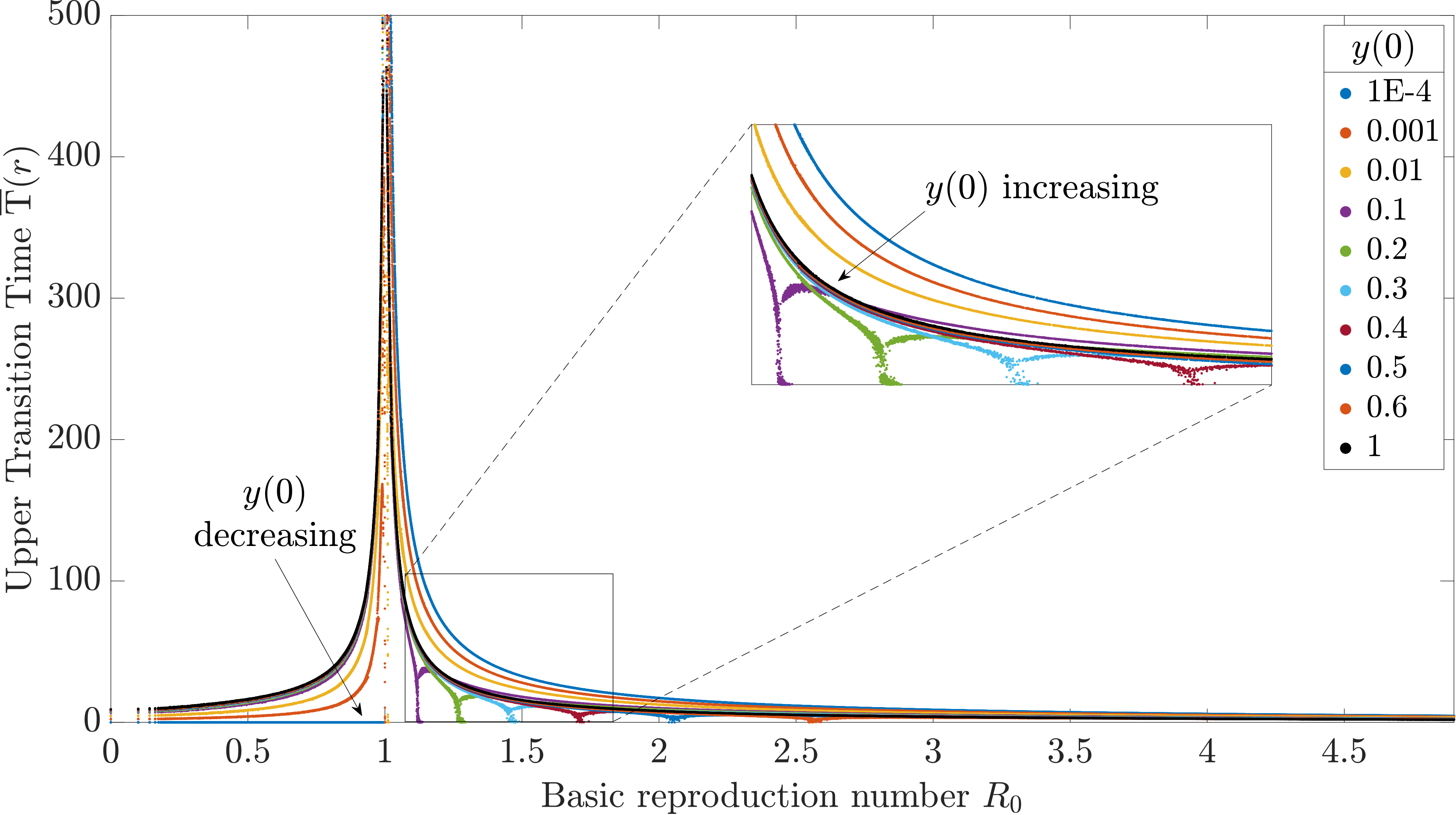}
    \caption{The upper-transition time $\mathrm{\overline{T}}(r)$ versus the basic reproduction number $R_0$ for different values of the starting prevalence $y(0)$. The chosen parameters are a graph size $N=50$, accuracy tolerance $r=10^{-4}$, infection rate $\beta = 0.1$. Each graph is an Erd\H{o}s-R\'enyi graph $G_{p}(N)$ with link density $p$ chosen uniformly at random: $p\sim \text{Unif}(0,1)$. For $R_0 > 1$, the small values of $y(0)$ correspond to the top curves, i.e. to a higher upper-transition time $\mathrm{\overline{T}}(r)$. For $R_0<1$, the small values of $y(0)$ correspond to the bottom curves, i.e. to a lower upper-transition time $\mathrm{\overline{T}}(r)$. For $y(0) = 10^{-4}$ the upper-transition time $\mathrm{\overline{T}}(r)$ is zero per definition. To determine the upper-transition time $\mathrm{\overline{T}}(r)$ we approximate the steady-state prevalence $y_{\infty}$ with $y_{\infty} \approx y(t_{\text{max}}) = y(10^4)$. The inset shows the dips below the curves in more detail.}
    \label{fig:tbarzoomed}
\end{figure*}

\subsection{Applying the Quenched approximation}\label{sec:results}
Given that the inter-updates times $T_1,T_2,\dots, T_M = \Delta t$ are all equal, we have extended the definition of the upper-transition time $\mathrm{\overline{T}}(r) = \mathrm{\overline{T}}(r,\mathcal{G})$ to sequences of graphs. In the following, we consider the problem of predicting the viral state vector $V(t)$ when only the sequence of graphs $G_1, ..., G_m$ is known. Predicting an epidemic in this setting is relevant when it is possible to obtain (an estimate of) the underlying, time-varying network structure of a population, but it is difficult to observe the viral state of the individuals. For instance, in some settings, estimates of contact networks could be obtained from human mobility data and social contact patterns \cite{mossong2008social,verelst2021socrates}, whereas accurately estimating the viral state might require expensive surveillance systems. Here, we focus on predicting the viral state $V(t)$ on the graph $G_m$, during the corresponding time interval $[t_{m-1}, t_m)$, in the quenched regime. We require only the graph $G_m$, the previous graph $G_{m-1}$ and the inter-update time $\Delta t > \mathrm{\overline{T}}(r,\mathcal{G})$ to be known \footnote{As an extension of our prediction method in this application setting, one could consider not only the most recent graph $G_{m-1}$ to be known, but the whole graph sequence $G_1, ..., G_{m-1}$.}. Particularly, our prediction method does not consider any observations of the viral state $V(t)$, neither at times $t < t_{m-1}$ nor at times $t \in [t_{m-1},t_m)$. We assume that there are no 'die-outs'.
The prediction method is applying the quenched approximation: the state vector $V(t_{m-1} + t)$ is predicted to equal the state vector $V(t)$ of a NIMFA SIS process on a static graph $G_m$ starting in $V(0) = V_{\infty}(G_{m-1})$. \\
In this section, we show how the quality of the prediction improves with $\Delta t$. Similar to Fig. \ref{fig:timeregimesplot} and Fig. \ref{fig:markovnimfacomparison}, the graph sequence $\mathcal{G}$ exists of i.i.d. ER-graphs with link density $p$ uniformly distributed on the interval $[0.3,0.8]$. For each interval $[t_{m-1},t_m)$ we make a ``quenched prediction'' based on $G_{m-1}$ and $G_m$. We calculate $V_{\infty}(G_{m-1})$ and then predict the epidemic on the interval to be equal to the NIMFA process on the static graph $G_m$ starting in $V(0) = V_{\infty}(G_{m-1})$. Fig. \ref{fig:predictions} shows these predictions for a sequence of ER-graphs with link density $p \in [0.3,0.8]$. As expected, Fig. \ref{fig:predictions} shows that the predictions improve with increasing inter-update time $\Delta t$. Additionally, when $y_{\infty}(G_{m-1})$ and $y_{\infty}(G_m)$ are close to each other, the prediction is accurate even at $\Delta t = 1$. The fact that the upper-transition time $\mathrm{\overline{T}}(r)$ is small when  $y_{\infty}(G_{m-1}) \approx y_{\infty}(G_m)$, corresponds to the dips in Fig. \ref{fig:tbarzoomed}, which appear when $y(0) \approx y_{\infty}$. Interestingly, the largest errors in Fig. \ref{fig:predictions} occur after intervals with a (large) decrease in the prevalence $y(t)$. Additional simulations with the same parameters also showed the largest errors after intervals with decreasing prevalence $y(t)$.
The reason that the estimate $V(t_m) = V_{\infty}(G_{m})$ is worse when $y(t_{m-1})$ is (much) larger than $y(t_m)$ is that the graph $G_{m}$ has a lower basic reproduction number $R_0$. Above the epidemic threshold $R_0 > 1$, graphs with lower basic reproduction numbers $R_0$ generally converge slower (largely independent of $V(0)$), as shown in Fig. \ref{fig:tbarzoomed}. The graphs with low basic reproduction number $R_0$ correspond to the intervals where the prevalence strongly decreases in Fig. \ref{fig:predictions}, because the steady-state prevalence $y_{\infty}$ increases with the basic reproduction number $R_0$.

\begin{figure}
  \centering
  \begin{subfigure}[]{0.47\textwidth}
  \centering
      \includegraphics[width=\textwidth]{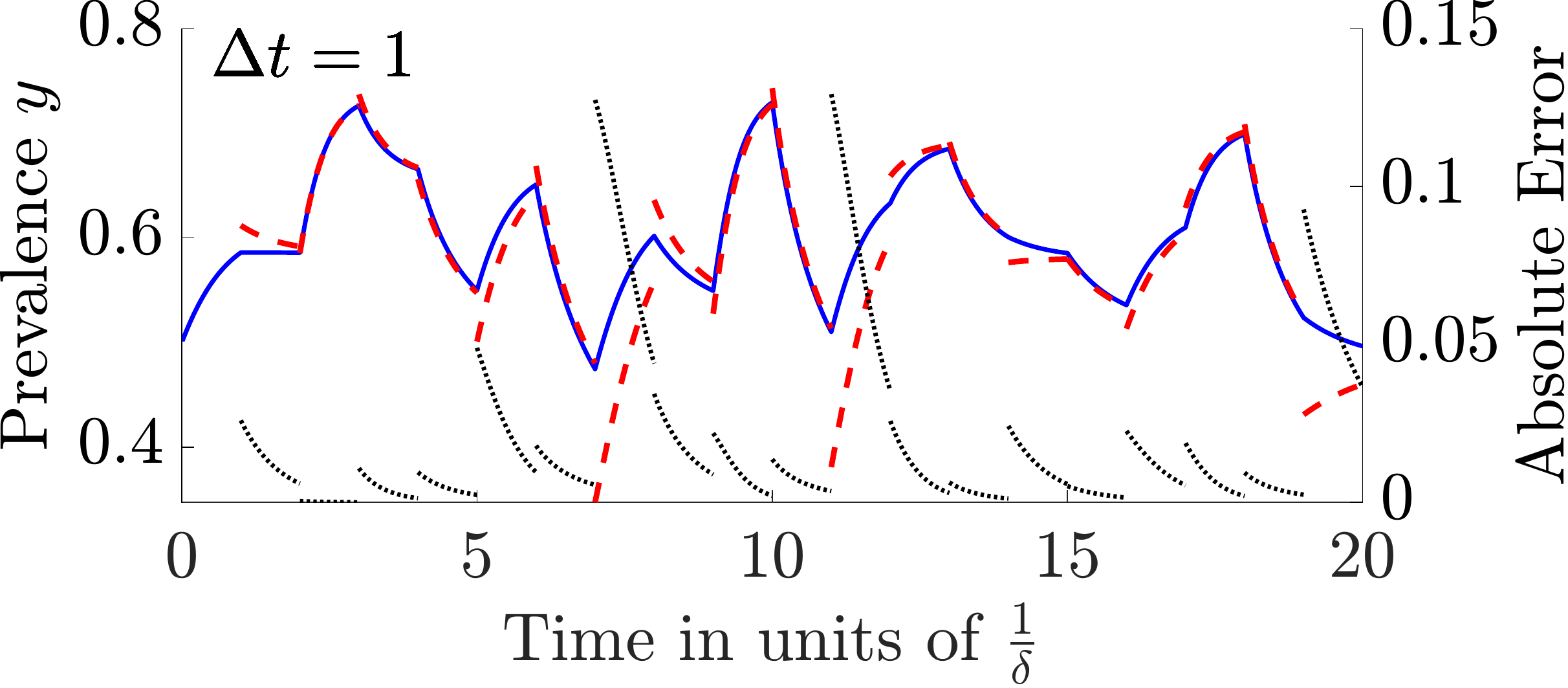}
  \end{subfigure}
  \begin{subfigure}[]{0.47\textwidth}
  \centering
      \includegraphics[width=\textwidth]{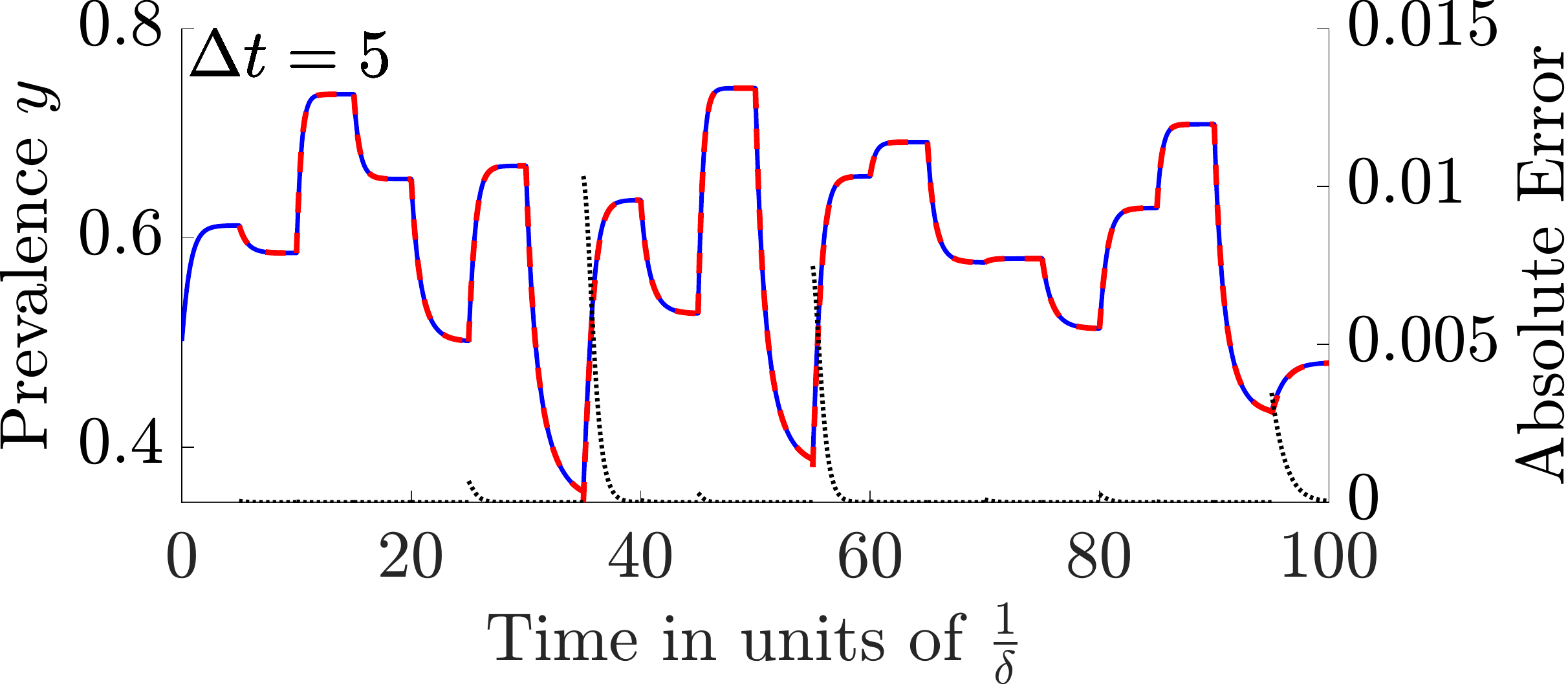}
  \end{subfigure}
  \begin{subfigure}[]{0.47\textwidth}
  \centering
      \includegraphics[width=\textwidth]{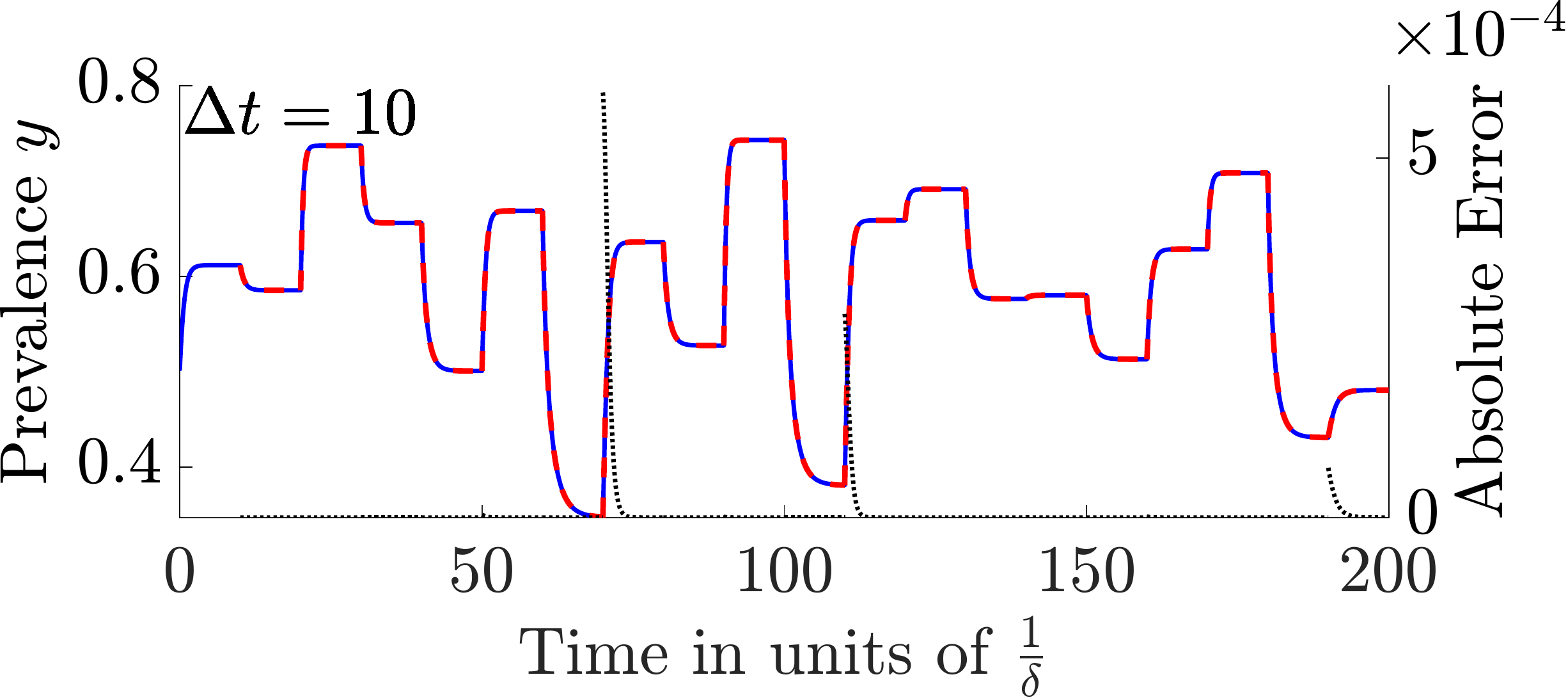}
  \end{subfigure}
  \caption{On the left $y$-axis the prevalence of a NIMFA process on a time-varying network (solid blue) is compared with the prevalence of the ``quenched prediction'' of each interval $[t_{m-1},t_m)$ starting from $y_{\infty}(G_{m-1})$ (dashed red) for different inter-update times $\Delta t$. The top figure has $\Delta t = 1$, the middle has $\Delta t = 5$ and the bottom one has $\Delta t = 10$. On the right $y$-axis the absolute error (dotted black) between the prediction and the process is shown. Each graph $G_m$ is an ER-graph with link density $p\in[0.3,0.8]$ and the graphs $G_m$ are the same between subfigures.}
\label{fig:predictions}
\end{figure}
In addition, Fig. \ref{fig:dieoutsareanissue} illustrates the quenched prediction when the process ``dies out''. We will estimate the starting condition of the interval with $V(t_m) = r u$ if $y_{\infty}(G_{m-1}) < r$. Fig. \ref{fig:dieoutsareanissue} shows the problem illustrated in Lemma \ref{lemma:arbitrarilyslowfromzero}. Since the inter-update time is well above the transition time for many of the graphs in the sequence, the prevalence drops below $r$ and the predictions fail, which is a major weakness of our method.\\
Firstly, Lemma \ref{lemma:arbitrarilyslowfromzero} shows that no matter how small the accuracy tolerance $r$ is chosen, ``die-outs'' will continue to cause bad predictions. Secondly, precisely when the approximation should be accurate (namely when $\Delta t$ is large), the process will stay on a graph with $R_0 < 1$ long enough to reach a prevalence $y(t) < r$ more often. Luckily, most real-world epidemics are no longer interesting after they die-out. Hence, as mentioned before, the no ``NIMFA die-outs" assumption is reasonable in these cases. However, if, for example, multiple COVID-19 waves are modelled (which is possible with NIMFA because there is no real die-out between waves), our quenched approximation will likely fail to predict a new wave accurately after a ``die-out''.

\begin{figure}
    \centering
    \includegraphics[width=0.47\textwidth]{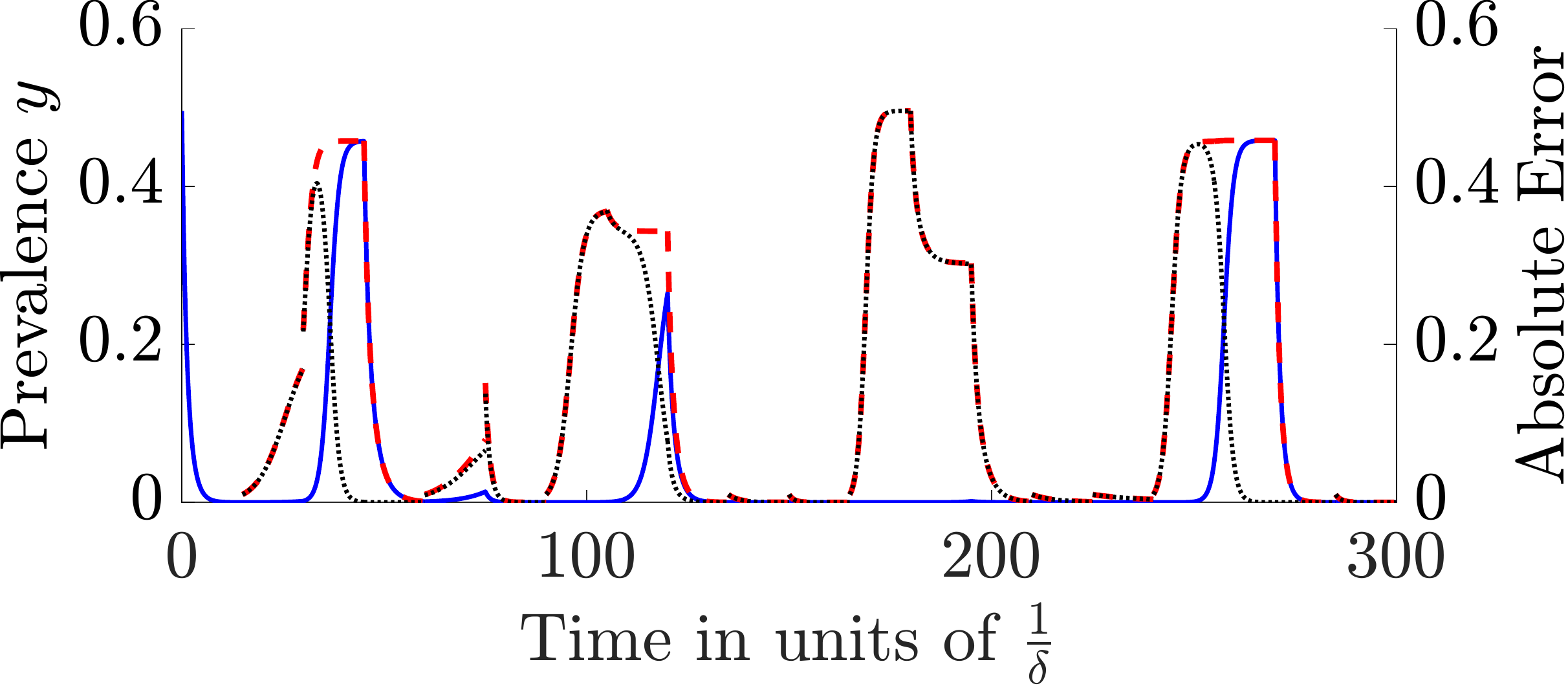}
    \caption{On the left $y$-axis the prevalence of a NIMFA process on a time-varying network (solid blue) is compared with the prevalence of the ``quenched prediction'' of each interval $[t_{m-1},t_m)$ starting from $y_{\infty}(G_{m-1})$ (dashed red) with inter-update times $\Delta t = 10$. When $y_{\infty}(G_{m-1}) = 0$, we set $y(t_m) = r u$ instead. On the right $y$-axis the absolute error (dotted black) between the prediction and the process is shown. Graphs are specifically chosen for the epidemic to be both above and below the epidemic threshold on different intervals. The issue illustrated in Lemma \ref{lemma:arbitrarilyslowfromzero} is visible: if the prevalence decreases much below $y = r u$, the error in the prediction is large.}
    \label{fig:dieoutsareanissue}
\end{figure}

\section{Bounds for the upper-transition time}\label{sec:bounds}
We bound the convergence towards the steady state $V_{\infty}$ from above and below separately.
We write $V(t) \geq V_{\infty}$ if $v_i(t) \geq v_{\infty,i}$ for all nodes $i$ and similarly, $V(t) \leq V_{\infty}$ if $v_i(t) \leq v_{\infty,i}$ for all nodes $i$. We call the convergence to the steady state from an infection probability vector $V(t) \geq V_{\infty}$ the \textit{decay} of $V(t)$ and convergence to the steady state from an infection probability vector $V(t) \leq V_{\infty}$ the \textit{growth} of $V(t)$. We upper bound the upper-transition time $\mathrm{\overline{T}}(r)$ by deriving bounds for decay (Section \ref{sec:upperbounds1}) and for growth (Section \ref{sec:upperbounds2}) and taking the maximum to combine them in Section \ref{sec:thegeneralupperbound}.
The following Theorem \ref{thm:prassedevriendt} (Theorem 5 in \cite{prasse2021clustering}) tells us that the slowest allowed growth starts in $V(0) = ru$ and the slowest decay starts in $V(0) = u$ for all contact graphs $G$, curing rates $\delta$ and infection rates $\beta$.
\begin{theorem}[Theorem 5 in \cite{prasse2021clustering}]{\label{thm:prassedevriendt}}
    Consider two NIMFA systems with respective positive curing rates $\delta_i$ and $\Tilde{\delta}_{i}$, non-negative infection rates $\beta_{ij}$ and $\Tilde{\beta}_{ij}$ and viral states $v_i(t)$ and $\Tilde{v_i}(t)$. Suppose that the initial viral states $v_i (0)$ and $\Tilde{v_i}(0)$ are in $[0,1]$ for all nodes $i$ and that matrices $B$ and $\Tilde{B}$, with elements $\beta_{ij}$ and $\Tilde{\beta}_{ij}$, respectively, are irreducible. Then, if  $\Tilde{\delta}_{i}\leq \delta_i$ and $\Tilde{\beta}_{ij}\geq \beta_{ij}$ for all nodes $i,j$, $\Tilde{V}(0) \geq V(0)$ implies that $\Tilde{V}(t) \geq V(t)$ at every time $t$.
\end{theorem}

In the following, we derive bounds for the convergence of NIMFA by considering two cases. First, we consider a \emph{growing} epidemic, where $V(0) < V_{\infty}$, for all nodes $i$. Second, we consider a \emph{decaying} epidemic, where $V(0) > V_{\infty}$, for all nodes $i$. We bound both cases with their extremal cases, which are $V(0) = ru$ and $V(0) = u$ respectively. While the proofs for the bounds in this section rely on Assumption \ref{assumption:vzero} below for generality, numerical simulations indicate that the bounds are applicable to general conditions on the initial viral state $V(0)$. Specifically, the simulations suggest that mixed cases, for which the bounds do not hold, are rare, if not non-existent. The underlying reason is that the mixed cases converge faster than at least one of the extremal cases. This property is assumed to be true in Assumption \ref{assumption:vzero}. Formally extending our results to mixed cases is subject for future work. \\
\begin{assumption}\label{assumption:vzero}
    A NIMFA SIS process with effective infection rate $\tau$, on the graph $G$, with a mixed initial condition $V(0)$, where $ r<= v_i(0) < v_{\infty, i}$ for some nodes $i$ and $v_j(0) > \text{ max }\{v_{\infty, j},r\}$, for some nodes $j$, converges slower to the steady state prevalence $y_{\infty}$ than the same process with the initial condition either $V(0) = u$ or $V(0) = ru$
\end{assumption}
To compare different processes, we introduce the notation $y(t;G,\tau,V(0))$ for the prevalence of the static NIMFA SIS process at time $t$, on the graph $G$ with effective infection rate $\tau$ and starting infection probability vector $V(0)$.

\subsection{ Conjectures for upper bounds for the decay to the all healthy state and endemic steady state}\label{sec:upperbounds1}
In this section, we first investigate the decay from the all infected state $V(0) = u$ on regular graphs and we state two conjectures that upper bound the upper-transition time. Conjecture \ref{thm:decaybound1} gives an upper bound for the case $R_0 \leq 1$ and Conjecture \ref{conjecture:taugroterdantauc} gives an upper bound for the case $R_0 > 1$. Afterwards, we derive a second bound for decay below the epidemic threshold in Lemma \ref{lemma:secondupperbound}.\\

 By Theorem \ref{thm:prassedevriendt}, the slowest decay towards the all-healthy state $V_{\infty} = 0$, for fixed effective infection rates $\tau < \tau_c^{(1)}(K_N) =$$\frac{1}{N-1}$, occurs in the complete graph $K_N$. The decay is slowest on the complete graph, because each node $i$ neighbours all other nodes $j$. Therefore, each node will receive the maximal possible infection attempts for each state vector $V(t)$. 

Based on extensive numerical simulations, we conjecture here that the prevalence $y(t;K_N,\tau^{(1)}_c(K_N),u)$ on the complete graph $K_N$ at the epidemic threshold $\tau_c^{(1)}(K_N)$, starting in the all-infected state $V(0) = u$, decays slower than \emph{any} prevalence $y(t;G,\tau,u)$, on any graph $G$, when $\tau \leq \tau_c^{(1)}(G)$, even if $\tau_c^{(1)}(K_N) \leq \tau \leq \tau_c^{(1)}(G)$. 
\begin{conjecture} \label{thm:decaybound1}
Given a graph $G$ and $\tau \leq \tau^{(1)}_c(G)$, the prevalence $y(t;G,\tau,V(0))$ of the epidemic process on $G$ with effective infection rate $\tau$ and starting infection probability vector $V(0)$ satisfies
\begin{eqnarray}\label{eq:thmdecaybound}
y(t;G,\tau,V(0)) \leq y(t;G,\tau^{(1)}_c(G),u)\nonumber\\ \leq y(t;K_N,\tau^{(1)}_c(K_N),u) = \frac{1}{1+t}.    
\end{eqnarray}
\end{conjecture}
Where $y(t;K_N,\tau^{(1)}_c(K_N),u) = \frac{1}{1+t}$ is proven in Lemma \ref{lemma:frac11plust} in Appendix \ref{app:lemmafrac11plust}. Analytical substantiation for Conjecture \ref{thm:decaybound1} is provided in Appendix \ref{app:substantationconjecture}. Additionally, we conjecture that the requirement $\tau \leq \tau_c^{(1)}(G)$ is not necessary and that the prevalence $y(t;K_N,\tau^{(1)}_c(K_N),u)$ also converges to the steady state $V(0) = 0$ slower than any prevalence $y(t;G,\tau,u)$ converges to its endemic steady state prevalence $y_{\infty}(G,\tau)$ when $\tau > \tau_c^{(1)}(G)$. 

\begin{conjecture}\label{conjecture:taugroterdantauc}
For any $\tau > \tau_c$ the prevalence $y(t;G,\tau,V(0))$ of the epidemic process on $G$ with effective infection rate $\tau$ starting in $V(0)\geq V_{\infty}$  decays faster to the steady state prevalence $y_{\infty}(G;\tau)$ than $y(t;K_N,\tau^{(1)}_c(K_N),u)$ to $y_\infty =0$. In other words, it holds that:
\begin{eqnarray}
|y(t;G,\tau,V(0)) - y_{\infty}(G;\tau)| \leq |y(t;G,\tau,u) - y_{\infty}(G;\tau)|\nonumber\\ 
\leq y(t;K_N,\tau^{(1)}_c(K_N),u) = \frac{1}{1+t}.\nonumber\\
\end{eqnarray}
\end{conjecture}

Simulations in Fig. \ref{fig:conjecturefigure} support Conjectures \ref{thm:decaybound1} and \ref{conjecture:taugroterdantauc} and suggest that the prevalence $y(t;K_N,\tau^{(1)}_c(K_N),u)$ upper bounds not only decay processes with effective infection rate $\tau \leq \tau_c^{(1)}$, but converges slower to the steady-state prevalence $y_{\infty}$ than any decay process, for any effective infection rate $\tau$. In Fig. \ref{fig:conjecturefigure}, the dashed red line represents $y(t;K_N,\tau^{(1)}_c(K_N),u) = \frac{1}{1+t}$. Each of the 100 blue lines per sub-figure corresponds to a decay process starting in $y(0) = u$ on an ER-graph. In each figure, $\tau$ changes as a multiple of $\tau^{(1)}_{c}(K_N)$ and is the same for each graph. In addition to the ER-graphs in Fig. \ref{fig:conjecturefigure}, the decay on several graphs with a fixed structure (including, for example, the path or star graph on $N$ nodes) was simulated for varied effective infection rates $\tau$. On the graphs with fixed structure, the decay from $y(0) =1$ to the steady-state prevalence $y_{\infty}$ is also faster than on $K_N$ with $\tau = \tau_c^{(1)}(K_N)$.
\begin{figure}[ht]
    \centering
    \begin{subfigure}[]{0.2\textwidth}
        \centering
        \includegraphics[width=\linewidth]{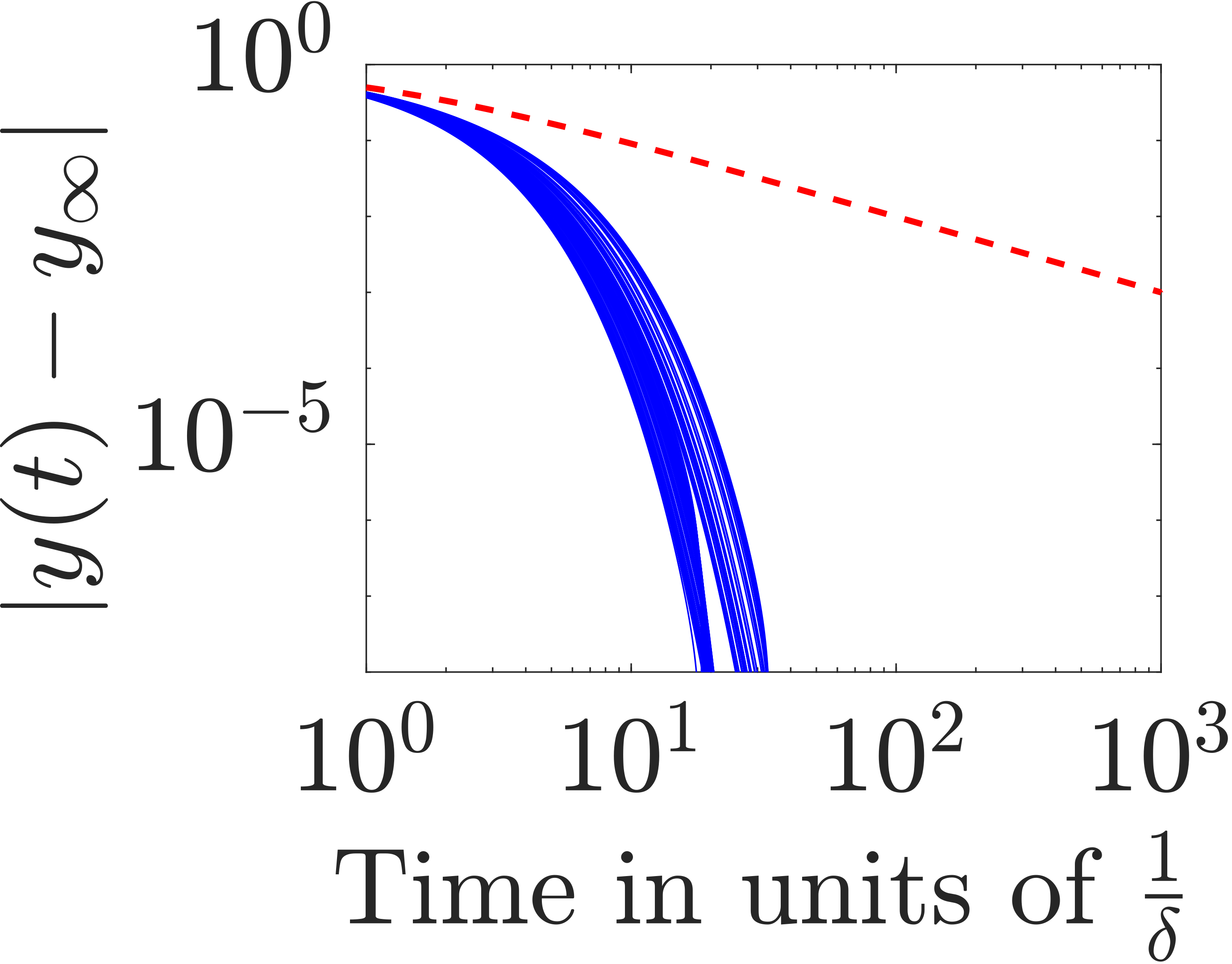} 
        \caption{$\tau = \frac{1}{2}\tau_c^{(1)}(K_N)$} \label{fig:conj1}
    \end{subfigure}
    \begin{subfigure}[]{0.2\textwidth}
        \centering
        \includegraphics[width=\linewidth]{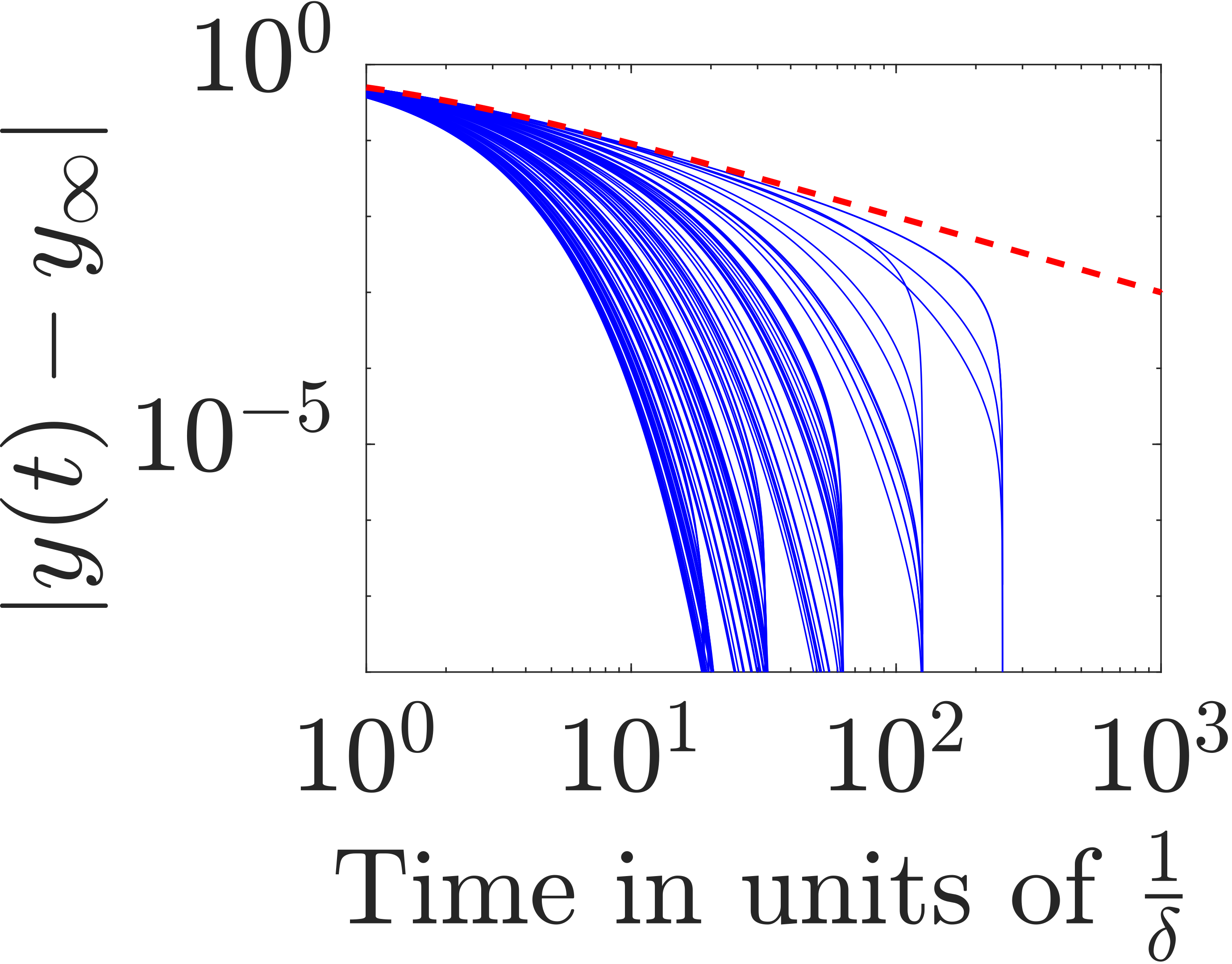} 
        \caption{$\tau = \tau_c^{(1)}(K_N)$} \label{fig:conj2}
    \end{subfigure}
    
    \begin{subfigure}[]{0.2\textwidth}
    \centering
        \includegraphics[width=\linewidth]{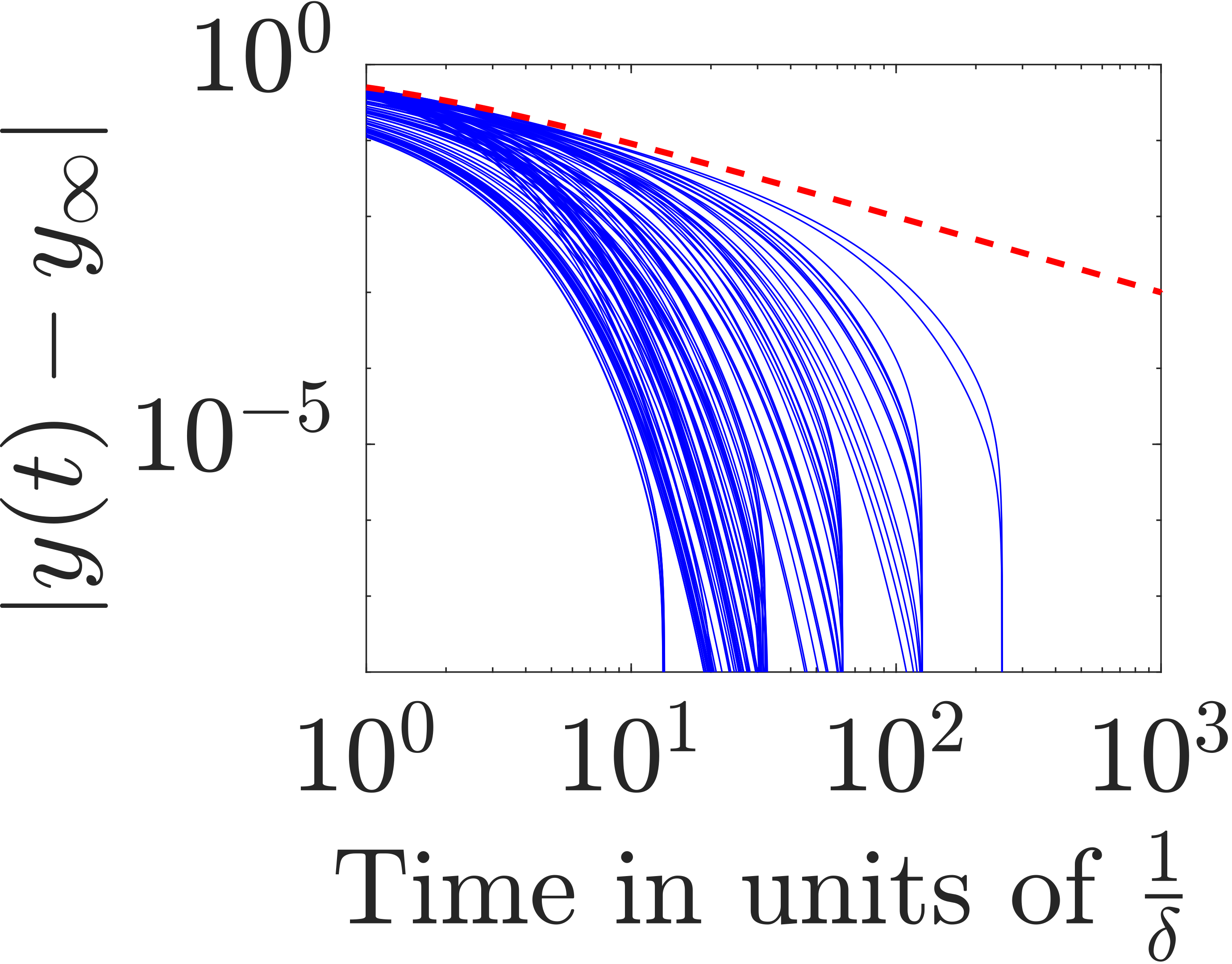} 
        \caption{$\tau = 2\tau_c^{(1)}(K_N)$} \label{fig:conj3}
    \end{subfigure}
    \begin{subfigure}[]{0.2\textwidth}
    \centering
        \includegraphics[width=\linewidth]{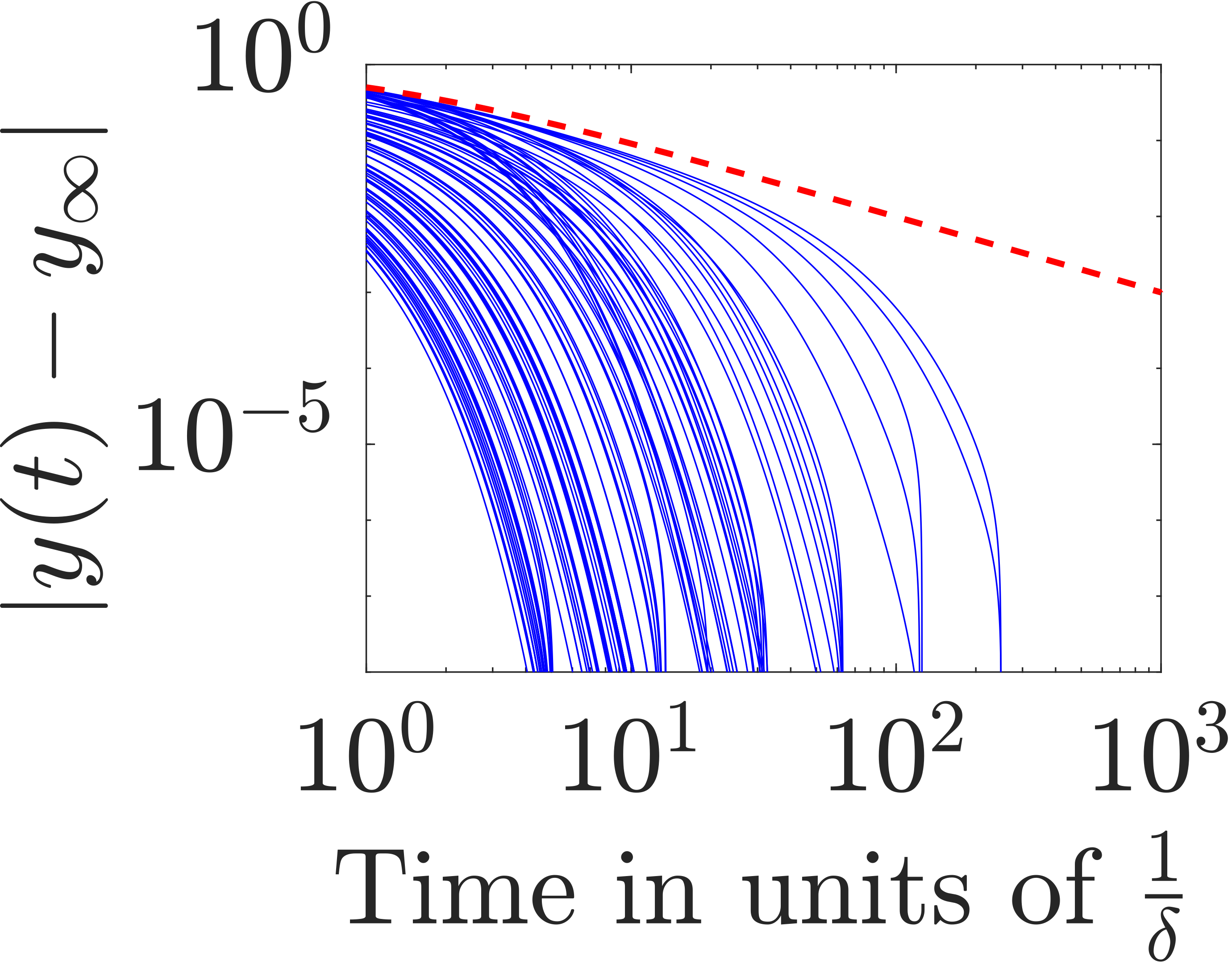} 
        \caption{$\tau = 5\tau_c^{(1)}(K_N)$} \label{fig:conj4}
    \end{subfigure}
    \caption{Error between the prevalence $y(t)$ and the steady-state prevalence $y_{\infty}$ (approximated by $y(t_{\text{max}}$)) versus time. The dashed red line is the function $y(t) = \frac{1}{1+t}$. Each of the blue curves corresponds to a decay process starting in $y(0)=u$ on} different $ER(N,p)$ graphs where $p \sim \text{Unif}(0,1)$. The effective infection rate $\tau$ is varied for the different figures. \label{fig:conjecturefigure}
\end{figure}

If Conjecture \ref{thm:decaybound1} is true, then the prevalence $y(t) = \frac{1}{1+t}$ upper bounds the prevalence of all epidemic processes with $R_0 \leq 1$. In particular, the bound holds for disconnected graphs with $R_0 \leq 1$, because each subgraph corresponding to a connected component will have $R_0 \leq 1$ and the prevalence of the disconnected graph is upper bounded by the largest of the subgraph prevalences. 
If Conjecture \ref{conjecture:taugroterdantauc} is true, then it follows that $y(t) = \frac{1}{1+t}$ upper bounds all decaying epidemic processes with $R_0 > 1$ as well. Assuming both conjectures are true, we can combine Conjecture \ref{thm:decaybound1} with Conjecture \ref{conjecture:taugroterdantauc} and invert the upper bound to find that every static NIMFA process with $V(0) \geq V_{\infty}$ will satisfy $|y(t) - y_{\infty}| \leq r$ at time $\frac{1-r}{r}$ such that
\begin{equation}\label{eq:bartr0leq1}
    \mathrm{\overline{T}}(r) \leq \frac{1-r}{r}.
\end{equation}

For effective infection rates $\tau \leq \tau_c^{(1)}$, we derive a second upper bound on $\mathrm{\overline{T}}(r)$ in Lemma \ref{lemma:secondupperbound}.
\begin{lemma}\label{lemma:secondupperbound}
Given a graph $G$ and $\tau < \tau^{(1)}_c(G)$, the prevalence $y(t;G,\tau,V(0))$ of the epidemic process on $G$ with effective infection rate $\tau$ and starting infection probability vector $V(0)$ satisfies
\begin{equation}\label{eq:thmdecaybound2}
y(t;G,\tau,V(0)) \leq y(t;G,\tau,u) \leq e^{(R_0(G,\tau) - 1)t},
\end{equation}
which leads to an upper bound on the upper-transition time $\mathrm{\overline{T}}(r)$:
\begin{equation}\label{eq:decayupperbound2}
    \mathrm{\overline{T}}(r) \leq \frac{1}{1- R_0(G,\tau)}\log\left(\frac{1}{r}\right).
\end{equation}
\end{lemma}
\begin{proof}
    We start with the matrix NIMFA equation \eqref{sys_matrix} and upper bound the derivative of the infection probability vector $V(t)$ by disregarding the non-linear term. After rescaling time such that $\delta = 1$ we find
 \begin{equation}\label{eq:tussenstapproofupperbound2}
    \df{V(t)}{t} \leq (\tau A - I)V(t),
     \end{equation}
     whose solution is given in \cite{Van2013non} as:
    \begin{equation}
        V(t) \leq e^{(\tau A - I)t}V(0).
    \end{equation}
    After using the norm $||\cdot||_2$ we obtain: 
    \begin{eqnarray}
    \label{eq:1normV0}
       ||V(t)||_2 \leq ||V(0)||_2 \; e^{\lambda_1(\tau A - I)t} \nonumber \\ = ||V(0)||_2 \; e^{(\tau\lambda_1(A) - 1)t} = ||V(0)||_2 \; e^{(R_0(G,\tau) - 1)t},
    \end{eqnarray}
where the first inequality follows from:
$$
||e^Mx||_2 \leq ||e^M||_2||x||_2 \leq e^{||M||_2}||x||_2 = e^{\lambda_1(M)}||x||_2, 
$$
 where $x$ is a $N\times 1$ vector, $M$ is a positive and symmetric $N \times N$ matrix and $\lambda_1(M)$ is the spectral radius of $M$. 
Since $||V(t)||_1 = N y(t) \leq \sqrt{N}||V(t)||_2$, we obtain from \eqref{eq:1normV0}:
   \begin{equation}\label{eq:prevalencebound}
        y(t) \leq \frac{\sqrt{N}}{N}||V(t)||_2 \leq \frac{1}{\sqrt{N}}||V(0)||_2 \; e^{(R_0(G,\tau) -1)t}.
    \end{equation}
    Substituting $||V(0)||_2 = ||u||_2 = \sqrt{N}$ in \eqref{eq:prevalencebound} gives the right-hand side inequality in \eqref{eq:thmdecaybound2}. Evaluating \eqref{eq:thmdecaybound2} at the time $\mathcal{T} = \frac{1}{1- R_0(G,\tau)}\log\left(\frac{1}{r}\right)$ yields:
\begin{eqnarray*}
    y(\mathcal{T};G,\tau,V(0)) \leq e^{(R_0(G,\tau) -1)\frac{1}{1- R_0(G,\tau) }\log\left(\frac{1}{r}\right)} \nonumber\\ = e^{-\log\left(\frac{1}{r}\right)} = r,
\end{eqnarray*}
from which the upper bound \eqref{eq:decayupperbound2} follows.
\end{proof}

The upper bounds \eqref{eq:bartr0leq1} and \eqref{eq:decayupperbound2} complement each other, because the bound \eqref{eq:bartr0leq1} is accurate around $R_0 = 1$, where the bound \eqref{eq:decayupperbound2} diverges and the bound \eqref{eq:decayupperbound2} is tighter than \eqref{eq:bartr0leq1} for small $R_0$. We derive the value $\mathcal{R}$ of the basic reproduction number $R_0$ where the bounds intersect by solving $\frac{1}{1- R_0}\log\left(\frac{1}{r}\right) = \frac{1-r}{r}$ for $R_0$ and obtain 
\begin{equation}\label{eq:crossingpointbounds}
\mathcal{R} = 1- \frac{r}{1-r}\log\left(\frac{1}{r}\right).
\end{equation}
Since 
$$ \lim_{r \to 0} \frac{r}{1-r}\log\left(\frac{1}{r}\right) = 0,$$
the intersection $\mathcal{R}$ is closer to $R_0 = 1$ for stricter (smaller) accuracy tolerances $r$. For the accuracy tolerances $r$ used in our simulations, the interval where \eqref{eq:bartr0leq1} is tighter than \eqref{eq:decayupperbound2} is already negligibly small.

\subsection{The upper bound for growth to the endemic steady state}\label{sec:upperbounds2}
For fixed $\tau$ and graphs $G$ with $R_0(G,\tau) > 1$, we briefly recall theory deduced in \cite{prasse2020time}. 
\begin{assumption}[Assumption 1 in \cite{prasse2020time}]\label{assumption:prasse1}
Assume that $\delta_i > 0$ for all nodes $i$ and $\beta_{ij} \geq 0$ for all nodes $i,j$. Additionally, we assume in the limit of $R_0 \downarrow 1$ that it holds that $\delta_i \nrightarrow 0$ and $\delta_i \nrightarrow \infty$ for all nodes $i$.
\end{assumption}
Assumption \ref{assumption:prasse1} is trivially satisfied in our case, because we assume all $\delta_i=\delta=1$, and all $\beta_{ij}=\beta>0$. We will use results in the limit $R_0 \downarrow 1$ as an approximation for the case when $R_0 > 1$. Since the upper-transition time is largest around the epidemic threshold, this approximation should be good for these $R_0$ values of interest. The precise mathematical description of the limit $R_0 \downarrow 1$ can be found in \cite{prasse2020time}.

\begin{assumption}[Assumption 3 in \cite{prasse2020time}]\label{assumption:prasse2}
For every basic reproduction number $R_0 > 1$, the infection matrix $B$ is symmetric and irreducible. Furthermore, in the limit $R_0 \downarrow 1$, the infection rate matrix $B$ converges to a symmetric and irreducible matrix. This holds if and only if $B$ (and its limit) corresponds to a connected undirected graph.
\end{assumption}
Assumption \ref{assumption:prasse2} only considers connected graphs \cite{van2014performance} directly. If a graph has disconnected components, the analysis can be applied to the connected components separately, because the connected components are independent. Hence, it is natural to assume, without loss of generality, that the graph is connected.

\begin{lemma}[Corollary 1 in \cite{prasse2020time}] \label{lemma:prasse1} 
Suppose that Assumptions 1 and 2 hold and that the initial viral state $V(0)$ equals $V(0) = \xi_0V_{\infty}$ for some scalar $\xi_0 \in (0,1)$. Then, for any scalar $\xi_1 \in [\xi_0, 1)$ the largest time $\hat{t}$ at which the viral state satisfies $V(\hat{t}) \leq \xi_1v_{\infty,i}$ for every node $i$ converges to
\begin{equation}\label{eq:t01}
\hat{t} = \frac{\tau_c}{(\tau - \tau_c)\delta} \log \left(\frac{\xi_1}{\xi_0}\frac{1-\xi_0}{1-\xi_1}\right),  
\end{equation}
when the basic reproduction number $R_0$ approaches $1$ from above.
\end{lemma}
\begin{lemma}[Equation 24 in \cite{prasse2020time}] \label{lemma:prasse2}
Given Assumptions 1 and 2 and that the initial viral state $V(0)$ is small or parallel to the steady-state vector $V_{\infty}$ we have that
$$
\frac{v_i(t)}{v_j(t)} \rightarrow \frac{(x_1)_i}{(x_1)_j},
$$
at every time $t$ when $R_0 \downarrow 1$, where $x_1$ is the non-negative eigenvector corresponding to the largest eigenvalue of the adjacency matrix $A$.
\end{lemma}

We will use Lemma \ref{lemma:prasse1} and Lemma \ref{lemma:prasse2} to upper bound the growth towards the steady state $V_{\infty}(G)$ in Lemma \ref{lemma:upperboundbart1}. Lemma \ref{lemma:prasse1} gives an expression for the convergence time to the steady state $\hat{t}$ up to a proportionality tolerance $\xi_1$ in the limit $R_0 \downarrow 1$. Lemma \ref{lemma:prasse2} ensures that we can pick the proportionality tolerance $\xi_1$ in such a way that the convergence time $\hat{t}$ implies $y(\hat{t}) \geq y_{\infty} -r$ in the limit $R_0 \downarrow 1$. The convergence time to the steady state $\hat{t}$ is then an upper bound for the upper-transition time $\mathrm{\overline{T}}(r)$ in the limit $R_0 \downarrow 1$. We derive this bound, because the upper-transition time $\mathrm{\overline{T}}(r)$ is largest for graphs around $R_0 = 1$.  However, as shown in Section \ref{sec:combination}, the bound holds for all $R_0 > 1$ in our numerical simulations. The unexpected accuracy of the bound (far) above the epidemic threshold is either because the error in $\hat{t}$ is negative for large $R_0$ (making \eqref{eq:t01} an upper bound for all $R_0 > 1$) or because the additional bounding steps in the derivation are large enough for the bound to hold for large $R_0$. In the following, we assume, without loss of generality, that the nodes $i$ are labeled such that $v_{\infty,1} \geq v_{\infty,2} \geq \dots \geq v_{\infty,N}$ and write $d_{\textnormal{max}}$ for the highest degree in the graph $G$.

\begin{lemma}[Upper bound on $\mathrm{\overline{T}}(r)$ for growth]\label{lemma:upperboundbart1}
 For a connected graph $G$, it holds that for all $V(0)$ with $v_i(0) \in [r,v_{\infty,i}]$, assuming $|y(0) - y_{\infty}| > r$, the upper-transition time $\mathrm{\overline{T}}(r)$; the first time such that $|y(\mathrm{\overline{T}}(r)) - y_{\infty}| \leq r$, is bounded by
\begin{eqnarray}\label{eq:lemmagrowthequation}
 \mathrm{\overline{T}}(r) \leq
    \frac{2}{R_0(G,\tau) - 1} \log \left(\frac{v_{\infty_1}-r}{r}\right) \nonumber\\ 
    \leq \frac{2}{R_0(G,\tau) - 1} \log \left(\frac{\tau d_{\textnormal{max}}}{r\left(\tau d_{\textnormal{max}}+1\right)}-1\right),  
\end{eqnarray}
when the basic reproduction number $R_0 $ approaches $1$ from above. When $|y(0) - y_{\infty}| \leq r$, the upper-transition time $\mathrm{\overline{T}}(r) = 0$.
\end{lemma}

\begin{proof}
By Theorem \ref{thm:prassedevriendt}, we only need to bound the case when $V(0) = ru$. We consider Lemma \ref{lemma:prasse1} and take $\xi_0 = \frac{r}{v_{\infty,1}}$ and $\xi_1 = (1- \frac{r}{v_{\infty,1}})$. For all nodes $i$, it holds that $\frac{v_{\infty,i}}{v_{\infty,1}} \leq 1$. We obtain:
$$
v_i(0) = \xi_0 v_{\infty,i} = r\frac{v_{\infty,i}}{v_{\infty,1}} \leq r,
$$ 
$$
\xi_1 v_{\infty,i} = \left(1-\frac{r}{v_{\infty,1}}\right)v_{\infty,i} = v_{\infty,i} - r \frac{v_{\infty,i}}{v_{\infty,1}} \geq v_{\infty,i} - r.
$$ 
With our chosen $\xi_0$ and $\xi_1$, the requirements $\xi_0 \in (0,1)$ and $\xi_1 \in [\xi_0, 1)$ in Lemma \ref{lemma:prasse1} become $\frac{r}{v_{\infty,1}} < 1$ and $\frac{r}{v_{\infty,1}} \leq 1 - \frac{r}{v_{\infty,1}}$, because $\frac{r}{v_{\infty,1}} > 0$ always holds. 

When either inequality does not hold, we argue that $|y(0)-y_{\infty}| < r$  and thus $\mathrm{\overline{T}}(r) = 0$. 

If $\frac{r}{v_{\infty,1}} \geq 1$, then we have $v_{\infty,1} \leq r$ and thus $|y(0)-y_{\infty}| \leq |r - v_{\infty,1}| < r$. Similarly, if $\frac{r}{v_{\infty,1}} > 1 - \frac{r}{v_{\infty,1}}$, then $\frac{r}{v_{\infty,1}} >\frac{v_{\infty,1} - r}{v_{\infty,1}}$, implying that $r > v_{\infty,1} - r$ and thus $|y(0)-y_{\infty}| \leq |r - v_{\infty,1}| < r$. 

In both cases, the upper-transition time $\mathrm{\overline{T}}(r)$ is zero, because the starting prevalence $y(0)$ is already closer to the steady-state prevalence $y_{\infty}$ than the accuracy tolerance $r$.

With the choice of $\xi_0 = \frac{r}{v_{\infty,1}}$ and $\xi_1 = (1- \frac{r}{v_{\infty,1}})$ the time $\hat{t}$ in (\ref{eq:t01}) is an upper bound for $\mathrm{\overline{T}}(r)$ in case of growth. Firstly, we have $v_i(0) \leq r$ for all $i$, which lower bounds the allowed starting conditions $v_i(0) \geq r$ from the definition (\ref{eq:deftbar}). Then, Theorem \ref{thm:prassedevriendt} states that the growth from $\xi_0V_{\infty}$ is slower than the slowest growth considered for the upper-transition time $\mathrm{\overline{T}}(r)$, namely the growth from $V(0) = ru$. Secondly, in the limit $R_0 \downarrow 1$, at time $\hat{t}$, we have $|v_i(\hat{t}) - v_{\infty,i}| \leq r$ for all nodes $i$. Lemma \ref{lemma:prasse1} states that $|v_i(\hat{t}) - v_{\infty,i}| \leq r$ holds for at least one node $i$, but Lemma \ref{lemma:prasse2} guarantees that, for all nodes $i$, the value of $\xi_1 v_{\infty,i}$ is reached at the same time $\hat{t}$, because the initial viral state vector $V(0)$ is parallel to the steady-state vector $V_{\infty}$. After substituting $\xi_0$ and $\xi_1$ into Lemma \ref{lemma:prasse1}, we find for $V(0) < V_{\infty}$ on connected graphs that
\begin{eqnarray}\label{eq:thateersteversie}
    \mathrm{\overline{T}}(r) \leq \hat{t} = \frac{1}{R_0(G,\tau) - 1} \log \left(\frac{(1-\frac{r}{v_{\infty,1}})^2}{(\frac{r}{v_{\infty,1}})^2}\right) \nonumber\\  = \frac{2}{R_0(G,\tau) - 1} \log \left(\frac{v_{\infty,1}-r}{r}\right) ,
\end{eqnarray}
where we have replaced $\frac{\tau_c}{(\tau - \tau_c)\delta}$ in \eqref{eq:t01} with $\frac{1}{R_0(G,\tau)-1}$, because we take $\delta = 1$ and 
\begin{eqnarray*}   
\frac{\tau_c^{(1)}(G)}{\tau - \tau_c^{(1)}(G)} = \frac{\frac{\tau}{R_0(G,\tau)}}{\tau - \frac{\tau}{R_0(G,\tau)}} = \nonumber\\ \frac{\tau}{R_0(G,\tau)\tau - \tau} = \frac{1}{R_0(G,\tau) -1}.
\end{eqnarray*}

Finally, we can upper bound $v_{\infty,1}$ invoking equation (15) from \cite{VanMieghem2008Virusspreadinnetworks}:
\begin{equation}\label{eq:vanmieghemvinfibound}
    0 \leq v_{\infty,i} \leq 1 - \frac{1}{1 + \tau d_i},
\end{equation}
where $d_i$ is the degree of node $i$. Equation \eqref{eq:vanmieghemvinfibound} holds for all nodes $i$ and thus also for the node with maximum degree such that $v_{\infty,1} \leq 1 - \frac{1}{1 + \tau d_{\text{max}}}$. Substituting this upper bound into (\ref{eq:thateersteversie}) proves the second inequality in (\ref{eq:lemmagrowthequation}) after simplification.
\end{proof}

\subsection{Summary: the general upper bound for time-variant networks}\label{sec:thegeneralupperbound}
We have conjectured and proved upper bounds for individual graphs on $\mathrm{\overline{T}}(r)$ for decay processes in (\ref{eq:bartr0leq1}) and \eqref{eq:decayupperbound2} and proved an upper bound for growth on connected graphs in the limit $R_0 \downarrow 1$ in Lemma \ref{lemma:upperboundbart1}. We combine the results and define an upper bound $\hat{T}_c(r,G)$ on $\mathrm{\overline{T}}(r)$ for a connected graph, specified by the subscript `$c$': 
\begin{eqnarray}\label{eq:combinedtbar}
    \hat{T}_c(r,G) = \frac{1}{1- R_0(G,\tau)}\log\left(\frac{1}{r}\right)\text{ if }R_0(G,\tau) \leq \mathcal{R},\nonumber\\
    \frac{1-r}{r} \text{ if }\mathcal{R} \leq R_0(G,\tau) \leq 1 \nonumber\\
    \max\left(\frac{1-r}{r},\frac{2}{R_0(G,\tau) - 1} \log \left(\frac{\tau d_{\textnormal{max}}}{r\left(\tau d_{\textnormal{max}}+1\right)}-1\right)\right) \nonumber\\
    \text{if }R_0(G,\tau) > 1 \nonumber\\
\end{eqnarray}
For disconnected simple graphs, the process on each of the components is completely independent. We can upper bound the convergence of the components separately using (\ref{eq:combinedtbar}). Consider a general graph $G$ with connected component subgraphs $C_1,\dots, C_K$, we define the upper bound $\hat{T}(r,G)$: 
\begin{equation}\label{eq:defupperbounddisconnected}
\hat{T}(r,G) = \max_{C_k} \left\{ \hat{T}_c(r,C_k) \right\}.
\end{equation}
Since definition (\ref{eq:defupperbounddisconnected}) also holds for connected graphs, we write our general upper bound for $\mathrm{\overline{T}}(r)$ for a set or sequence of graphs $\mathcal{G} = \{G_1,G_2,\dots,G_M\}$ as
\begin{equation}\label{eq:generalupperboundtbar}
    \mathrm{\overline{T}(r,\mathcal{G})} \leq \max_{G\in \mathcal{G}} \left\{ \hat{T}(r,G) \right\}= \max_{G\in \mathcal{G}} \left\{ \max_{C_k} \left\{ \hat{T}_c(r,C_k) \right\} \right\}.
\end{equation}

\subsection{Lower bounds for the upper-transition time}\label{sec:lowerbounds} In this section, we derive lower bounds on the slowest growth $V(0) = ru$ and decay $V(0) = u$ processes to compare with the upper bounds.
Inequality \eqref{eq:prevalencebound} allows us to deduce a non-trivial lower bound for $\mathrm{\overline{T}}(r)$ in growth processes:
\begin{lemma}
    Given a graph $G$, an effective infection rate $\tau > \tau_c^{(1)}(G)$ and an accuracy tolerance $r$, the upper-transition time $\mathrm{\overline{T}}(r)$ has a lower bound for growth given by
    \begin{equation}\label{eq:lowerboundgrowth}
        \mathrm{\overline{T}}(r) \geq \frac{1}{(R_0(G,\tau) -1)}\log\left(\frac{y_{\infty}(G) - r}{r}\right).
    \end{equation}
\end{lemma}
\begin{proof}
    Theorem \ref{thm:prassedevriendt} shows that the slowest growth starts in the lowest allowed starting probability vector $V(0) = r u$.  We have $||V(0)||_2 = ||ru||_2 = \sqrt{N} r$. Substitution in \eqref{eq:prevalencebound} gives
    $$
    y(t) \leq r e^{(R_0(G,\tau) -1)t},
    $$
    where $R_0(G,\tau) - 1 > 0$ is now positive, contrary to Lemma \ref{lemma:secondupperbound}. The time $\mathcal{T}$ such that 
    \begin{equation}\label{eq:tussenstaplowerbound}
     r e^{(R_0(G,\tau) -1)\mathcal{T}} = y_{\infty}(G) - r   
    \end{equation}
    bounds $\mathrm{\overline{T}}(r)$ from below, because $y(\mathcal{T}) \leq r e^{(R_0(G,\tau) -1)\mathcal{T}} = y_{\infty}(G) -r$.
    We solve \eqref{eq:tussenstaplowerbound} for $\mathcal{T}$ and find \eqref{eq:lowerboundgrowth}.
    \end{proof}
We also determine the following lower bound for decay processes.
\begin{lemma}
    Given a graph $G$, an effective infection rate $\tau$ and an accuracy tolerance $r$, the upper-transition time  $\mathrm{\overline{T}}(r)$ has a lower bound for the decay given by
    \begin{equation}\label{eq:lowerbounddecay}
        \mathrm{\overline{T}}(r) \geq \log\left(\frac{1}{y_{\infty}(G)+r}\right).
    \end{equation}
\end{lemma}
\begin{proof}
    We consider a decay process and lower bound its prevalence by ignoring the infection process. We have at all times $t$ that $\df{y}{t} \geq - y$, by ignoring the non-negative second term in \eqref{NIMFASISdelta1} and summing over the infection probabilities $v_i$. Using Gr\"onwalls' Lemma \cite{gronwall1919note} we obtain $y(t) \geq y(0)e^{-t}$. The highest lower bound occurs when $y(0) = 1$ and we obtain $|y(t) - y_{\infty}(G)| \geq e^{-t} - y_{\infty}(G)$. We find that $t = \mathcal{T} = \log\left(\frac{1}{y_{\infty}(G)+r}\right)$ is the time such that $e^{-t} = y_{\infty}(G) +r$ and therefore a lower bound on $\mathrm{\overline{T}}(r)$ in a decay process, because $|y(\mathcal{T}) - y_{\infty}| \geq e^{-\mathcal{T}} -y_{\infty}(G) = r$.    
\end{proof}

\subsection{Evaluation of the bounds}\label{sec:combination} 
We simulate the upper-transition time $\mathrm{\overline{T}}(r)$, with $r=10^{-4}$ for the extreme values of $V(0)$, namely, $V(0) = u$ and $V(0) = ru$. Fig. \ref{fig:tbarcompare} shows the upper-transition time $\mathrm{\overline{T}}(r)$ together with the bounds from Section \ref{sec:bounds}. The lower bounds are indicated with the symbol $L_x$ and the upper bounds with $U_x$. Here $x$ is $D$ for decay or $G$ for growth. The upper transition times are indicated with $\mathrm{\overline{T}}_D$ and $\mathrm{\overline{T}}_G$. The upper bound $U_D$, given by \eqref{eq:decayupperbound2} and \eqref{eq:bartr0leq1}, is remarkably accurate for $R_0 \leq 1$, but is less sharp for $R_0 > 1$. This is expected, because the upper bound (\ref{eq:bartr0leq1}) is an upper bound of the asymptote at $R_0 =1$. The upper bound $U_G$, given by \eqref{eq:lemmagrowthequation}, is a solid bound for the upper-transition time $\mathrm{\overline{T}}_G$. Additionally, for large $R_0$, the upper bound $U_G$ approximates the upper-transition time well. The lower bound $L_G$, given by \eqref{eq:lowerboundgrowth}, unexpectedly fits the upper-transition time $\mathrm{\overline{T}}_D$ very well for $R_0 > 1$. The lower bound $L_D$, given by \eqref{eq:lowerbounddecay}, is weak. All bounds hold for all $R_0 \geq 0$. The upper bound $U_D$ and the lower bound $L_G$ are both derived from the same approximation \eqref{eq:prevalencebound}, but $L_G$ seems to perform worse than $U_D$. 

\begin{figure}[ht]
    \centering
    \includegraphics[width=0.5\textwidth]{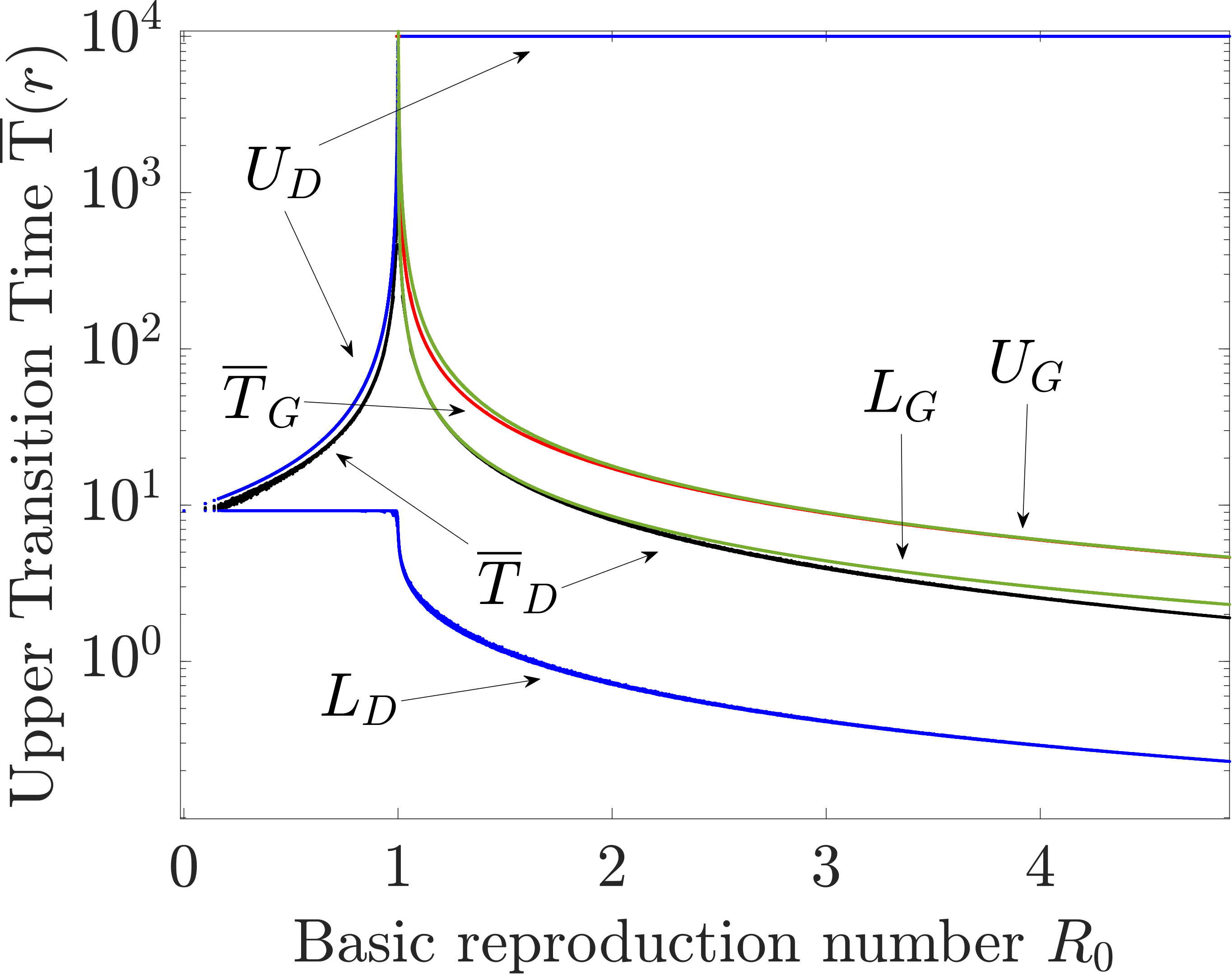}
    \caption{Comparison of the analytically derived bounds and numerically determined upper-transition time $\mathrm{\overline{T}}(r)$, with $r=10^{-4}$. The upper-transition time $\mathrm{\overline{T}}(r)$ is shown for $y(0) = 1$  in black with symbol $\overline{T}_D$ for decay and for $y(0) = r$ in red with symbol $\overline{T}_G$ for growth. The blue lines are the bounds on decay (black) and the green lines are the bounds on growth (red). The lower bounds are indicated with the symbol $L_x$ and the upper bounds are indicated with the symbol $U_x$, where $x$ is $D$ for decay or $G$ for growth. The upper bound $U_D$ is given by equation \eqref{eq:decayupperbound2} for $R_0 < \mathcal{R}$ and \eqref{eq:bartr0leq1} for $R_0 \geq \mathcal{R}$. Here, the intersection point $\mathcal{R}$ is given by \eqref{eq:crossingpointbounds}. The lower bound $L_D$ is given by \eqref{eq:lowerbounddecay} and the lower bound $L_G$ is given by \eqref{eq:lowerboundgrowth}. The upper bound $U_G$ is given by \eqref{eq:lemmagrowthequation}.}
    \label{fig:tbarcompare}
\end{figure}

Lastly, we return to the temporal process and repeat the simulations from Fig. \ref{fig:predictions}. However, we choose the inter-update time $\Delta t$ to be equal to the upper bound for $\mathrm{\overline{T}}(r)$ from \eqref{eq:combinedtbar} with $r=10^{-4}$. We do not consider the upper bound \eqref{eq:bartr0leq1} for $R_0 > 1$, because Fig. \ref{fig:tbarcompare} suggests that \eqref{eq:lemmagrowthequation} is a general upper bound. For the lower bound, we can only use the loose bound \eqref{eq:lowerbounddecay}, because \eqref{eq:lowerbounddecay} is smaller than \eqref{eq:lowerboundgrowth} everywhere. The upper figure in Fig. \ref{fig:predictionsbounds} shows that, at the upper bound, the absolute error is many times smaller than $r$. The lower figure in Fig. \ref{fig:predictionsbounds} shows that, at the lower bound, the absolute error is many times larger than $r$. At it's lowest point the error is $2r$.
\begin{figure}
  \centering
  \begin{subfigure}[]{0.47\textwidth}
  \centering
      \includegraphics[width=\textwidth]{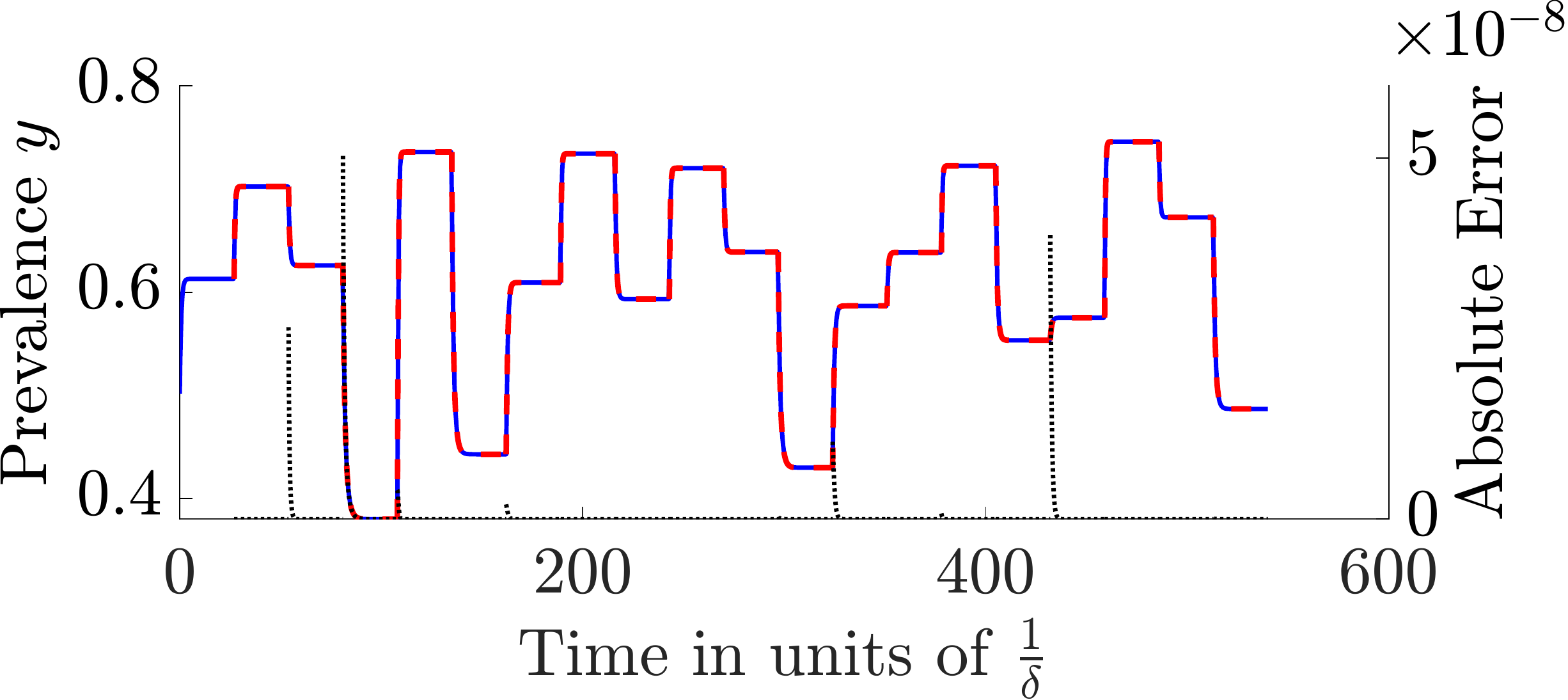}
  \end{subfigure}
  \begin{subfigure}[]{0.47\textwidth}
  \centering
      \includegraphics[width=\textwidth]{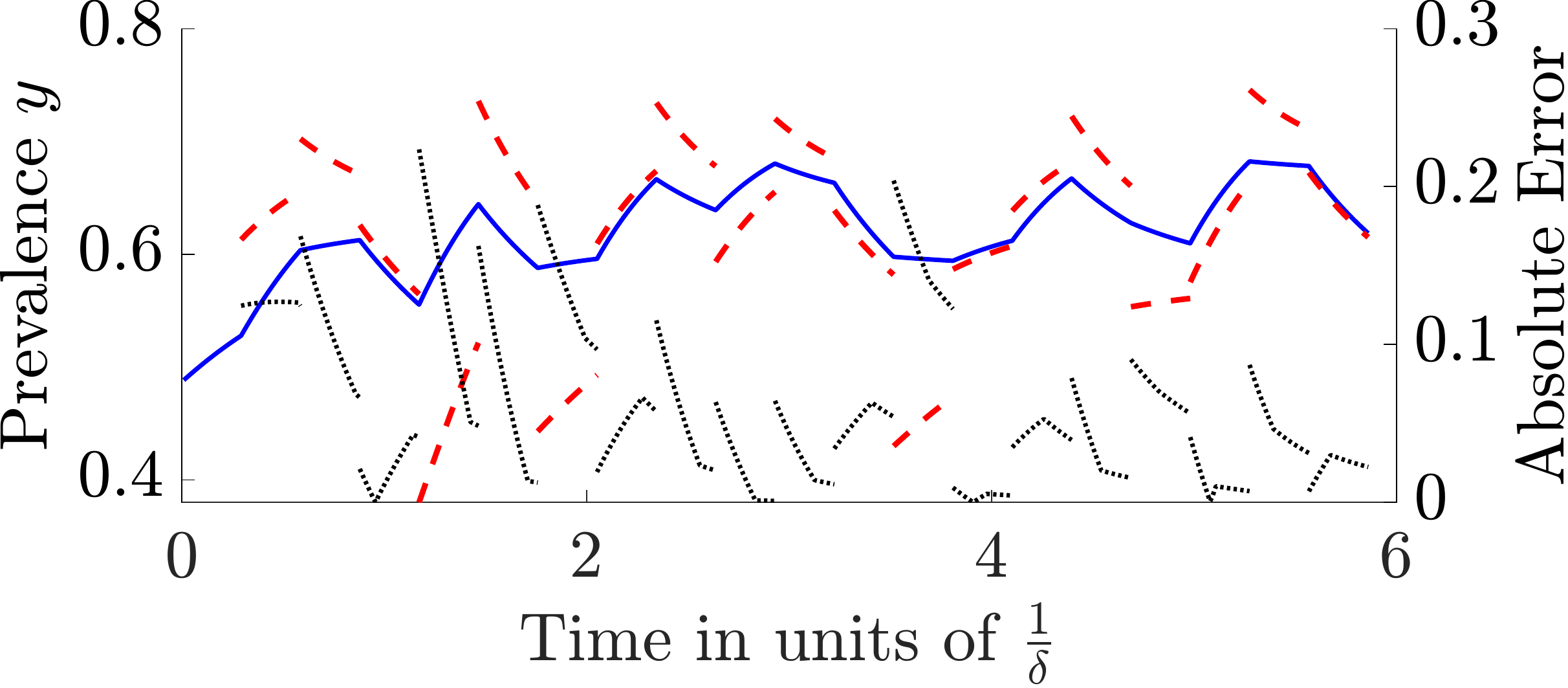}
  \end{subfigure}
  \caption{On the left $y$-axis the prevalence of a NIMFA process on a time-varying network (solid blue) is compared with the prevalence of the ``quenched prediction'' of each interval $[t_{m-1},t_m)$ starting from $y_{\infty}(G_{m-1})$ (dashed red). The inter-update times $\Delta t$ equal the (modified) upper bound \eqref{eq:combinedtbar} in the top figure and the lower bound \eqref{eq:lowerbounddecay} in the bottom figure. On the right $y$-axis the absolute error (dotted black) between the prediction and the process is shown. Each graph $G_m$ is an ER-graph with link density $p\in[0.3,0.8]$ and the graphs $G_m$ are the same between subfigures.}
\label{fig:predictionsbounds}
\end{figure}

\section{Conclusions}\label{sec:summary}
The interplay of the topological process and the epidemic process is a substantial challenge in the analysis of general epidemics on time-variant networks. Timescale separation often restricts studies to simpler scenarios, where the network is assumed static or approximated by the average of the time-variant networks over a finite interval. However, in real-world epidemics, the network changes at a speed comparable to the epidemic process. Hence, a thorough understanding of the intermediate regime, where the timescales cannot be separated, is vital for a realistic and reliable modelling approach. 

In this work, a first step towards the analysis of the intermediate regime is presented. We define the upper-transition time $\mathrm{\overline{T}}(r)$, a threshold quantity which characterizes the border of the intermediate regime and the \emph{quenched} regime, in which the network is approximately static. In an analysis of an SIS epidemic, this threshold quantity $\mathrm{\overline{T}}(r)$ can determine whether a network can be assumed to be static. Indeed, when the inter-update time $\Delta t$ is larger than $\mathrm{\overline{T}}(r)$, the epidemic process can, in most situations, be accurately predicted using the quenched approximation. We show that for fixed infection rate $\beta$, curing rate $\delta$ and initial state vector $V(0)$, but with different graphs, the basic reproduction number $R_0$ determines the upper-transition time $\mathrm{\overline{T}}(r)$. We derive upper and lower bounds for the upper-transition time $\mathrm{\overline{T}}(r)$ in \eqref{eq:generalupperboundtbar}, \eqref{eq:lowerboundgrowth} and \eqref{eq:lowerbounddecay}, and compare them in Fig. \ref{fig:tbarcompare} to numerical estimations of $\mathrm{\overline{T}}(r)$. We introduce the derivative convergence time $t^{\ast}(r^{\ast})$ in Appendix \ref{app:derivativeconvegencetime}, to upper bound the upper-transition time $\mathrm{\overline{T}}(r)$ and we argue that $t^{\ast}(r^{\ast})$ is easier to determine numerically, although the computation time saved varies depending on how one determines the steady-state prevalence $y_{\infty}$.
Additionally, we show that the upper-transition time $\mathrm{\overline{T}}(r)$ is large when networks around the epidemic threshold are present in the temporal process.

Some real-world epidemics, e.g. influenza, are characterized by $R_0$ slightly above the epidemic threshold \cite{biggerstaff2014estimates}. For similar diseases around the epidemic threshold, our work shows that the upper-transition time $\mathrm{\overline{T}}(r)$ is large in units of the average curing rate $\frac{1}{\delta}$. Therefore, these real-world epidemics near the threshold are in the intermediate regime; hence, we expect that the network topology updates play an active role in the disease spread. Additionally, the limits of the quenched approximation and predictions of the process, even if $\Delta{t} \geq \mathrm{\overline{T}}(r)$, as explained in Section \ref{sec:timeregimesoftemporalnetworkepidemics}, show that even in the approximate quenched regime, temporal effects are not negligible in general.

We leave various open paths for future work: 1) investigating heterogeneous NIMFA SIS; 2) considering inter-update times $T_m$ that are random variables with arbitrary distributions $F_{T_m}(x) = \Pr[T_m \leq x]$ instead of fixed $\Delta{t}$ as here; 3) considering a Markovian time-variant SIS process besides NIMFA; 4) since the lower transition time $\mathrm{\underline{T}}(r)$ is not as easily defined as the upper-transition time, an analysis of the lower boundary of the intermediate regime is also a promising topic for further analysis; 5) lastly, while our bounds \eqref{eq:generalupperboundtbar}, \eqref{eq:lowerboundgrowth} and \eqref{eq:lowerbounddecay} are general, our simulations of $\mathrm{\overline{T}}(r)$ were limited to ER-graphs only. We show in Appendix \ref{app:assumptionofERgraphs} that different networks do not behave significantly different for $N=50$. However, a thorough verification of our results on different classes of networks might be of interest. \\

\begin{acknowledgments}
This research has been funded by the European Research Council (ERC) under the European Union’s Horizon 2020 research and innovation programme (grant agreement
No 101019718). MS was supported by the Italian Ministry for University and Research (MUR) through the PRIN 2020 project ``Integrated Mathematical Approaches to Socio-Epidemiological Dynamics'' (No. 2020JLWP23). \\
{\bf Disclaimer:} The views and opinions expressed herein are the authors’ own and do not necessarily state or reflect those of ECDC. ECDC is not responsible for the data and information collation and analysis and cannot be held liable for conclusions or opinions drawn.
\end{acknowledgments}

\appendix
\section{Assumption of ER-graphs}\label{app:assumptionofERgraphs}
Fig. \ref{fig:compareERBAWS} illustrates that for $N=50$ the ER-graphs are representative of general graphs. We plot the upper-transition time $\mathrm{\overline{T}}(r)$ for each graph type (Barab\'{a}si-Albert (BA), Erd\H{o}s-R\'enyi (ER) and Watts-Strogatz (WS)). We draw 3000 realisations of each of the three random graph models, and we assume uniformly distributed parameters for each random graph model, except for the number of nodes $N$, which we consider fixed. \footnote{We denote Unif$(a,b)$ for the continuous uniform distribution on the interval $[a,b]$ and Unif$\{a,b\}$ for the discrete uniform distribution on the set $\{a, a+ 1,\dots, b-1, b\}$}. ER-graphs have one parameter: the link density $p\sim$ Unif$[0,1]$. BA-graphs have two parameters: the size of the starting clique $m_0\sim$ Unif$\{1, N\}$ and the degree of nodes when attached $m\sim$ Unif$\{1,m_0\}$. WS-graphs have two parameters: the average degree divided by two $K\sim$ Unif$\{1, \lfloor\frac{N-1}{2}\rfloor\}$ and the rewiring probability $\beta_{WS}\sim$ Unif$[0,1]$. While Fig. \ref{fig:compareERBAWS} shows that the three different graph types do result in a different transition times $\mathrm{\overline{T}}(r)$, the absolute difference of the transition time seems very small and almost negligible. It is plausible that these differences will become larger when $N$ increases. The inset of Fig. \ref{fig:compareERBAWS}, which is the same as the main figure but with logarithmic axes, illustrates that, while the absolute differences are negligible, there is already a significant relative difference between ER/WS and BA graphs.
\begin{figure*}[ht]
    \centering
    \includegraphics[width=\textwidth]{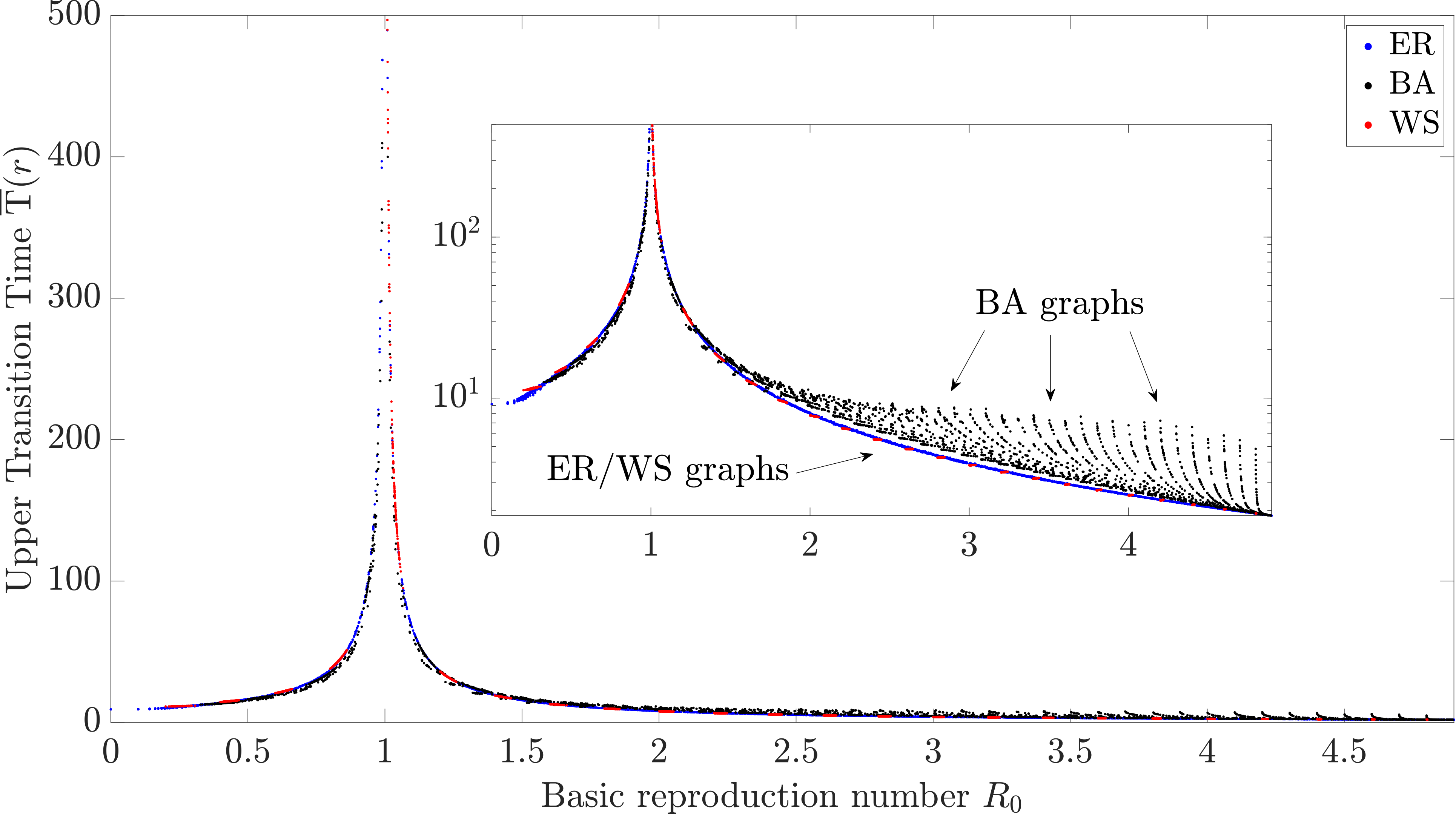}
    \caption{The upper-transition time $\mathrm{\overline{T}}(r)$ versus basic reproduction number $R_0$ for three different random graph types. Shown is the decay process starting in the all-infected initial state vector $V(0) = u$, for $N=50$ and $\beta = 0.1$. In blue, 3000 Erd\H{o}s-R\'enyi (ER) graphs with uniformly distributed parameters ($G_p(N)$ with $p\sim$ Unif$(0,1)$). In black, 3000 Barab\'{a}si-Albert (BA) graphs with uniformly distributed parameters ($m_0\sim$ Unif$\{1,N\}$ and $m\sim$ Unif$\{1,m_0\}$). In red, 3000 Watts-Strogatz (WS) graphs with uniformly distributed parameters ($K\sim$ Unif$\{1,\lfloor \frac{N-1}{2} \rfloor\}$ and $\beta_{WS}\sim$ Unif$(0,1)$). In the inset, the same figure with log-linear axes is shown. In log scale, for $R_0 > 2$, there is a significant difference from BA graphs to ER/WS graphs, which are on top of each other and below the BA graphs.}
    \label{fig:compareERBAWS}
\end{figure*}

\section{Proof of Lemma \ref{lemma:arbitrarilyslowfromzero}}\label{app:lemmaproof}
\begin{proof}
    We consider the starting infection probability vector $V(0) = \varepsilon(r) u$. We will upper bound the prevalence $y(t)$ and then substitute the starting condition and show that $|y(\mathcal{T})-y_{\infty}| > r$. The derivative $\df{y(t)}{t}$ of the prevalence $y(t)$ in \eqref{eq:prevalance} follows from \eqref{NIMFASISdelta1} as:
    $$
    \df{y(t)}{t} = - y(t) + \frac{\tau}{N} \sum_{i=1}^{N}(1-v_i(t))\sum^{N}_{j=1}a_{ij}v_j(t).
    $$
    Since $(1-v_i(t)) \leq 1$ and $a_{ij} \leq 1$, we find that
    $$
    \df{y(t)}{t} < - y(t) + \frac{\tau}{N} \sum_{i=1}^{N}\sum^{N}_{j=1}v_j(t) = -y(t) + \tau N y(t),
    $$
    where we have a strict inequality, because we do not consider self-loops $(a_{ii}=0)$, which is a reasonable assumption for individual-based contact graphs.
    Using Gr\"onwall's Lemma \cite{gronwall1919note} on the differential inequality $\df{y(t)}{t} < (\tau N -1) y(t)$ gives us an upper bound on the prevalence $y(t)$:
   \begin{equation}\label{eq:upperboundprevalence}
    y(t) < y(0)e^{(\tau N -1) t}.   
   \end{equation}
    The starting infection probability vector $V(0) = \varepsilon(r) u$ implies that the starting prevalence $y(0) = \varepsilon(r)$. We set $\varepsilon(r) = \frac{y_{\infty}-r}{e^{(\tau N -1)\mathcal{T}}}$, such that $y(0) = \varepsilon(r) < y_{\infty}$.  The difference $|y(t) - y_{\infty}|$ is: 
    $$
    |y(t) - y_{\infty}| = y_{\infty} - y(t) > y_{\infty} - (y_{\infty} - r)\frac{e^{(\tau N -1) t}}{e^{(\tau N -1) \mathcal{T}}}.
    $$
    Since $\frac{e^{(\tau N -1) t}}{e^{(\tau N -1) \mathcal{T}}} \leq 1$ for times $t \leq \mathcal{T}$, we find $|y(t) - y_{\infty}| > y_{\infty} - y_{\infty} + r = r$, which proves Lemma \ref{lemma:arbitrarilyslowfromzero}.
\end{proof}

\section{Derivative Convergence time}\label{app:derivativeconvegencetime}

As an heuristic for the upper-transition time $\mathrm{\overline{T}}(r)$, we consider the derivative convergence time $t^{\ast}$, which does not depend on $y_{\infty}$. The derivative convergence time is defined as the time $t\geq0$ when a NIMFA SIS process first obeys the inequality $|v_{i}(t+h) - v_{i}(t)| \leq hr^{\ast}$ for all nodes $i$ in the graph $G$:
\begin{equation}\label{eq:deftstar}
t^{\ast} = \min_{t \geq 0} \left\{|v_{i}(t+h) - v_{i}(t)| \leq hr^{\ast} \; ; \; \forall i=1,\dots,N \right\},
\end{equation} 
where $h$ is a step-size parameter and $r^{\ast}$ is an accuracy tolerance. We denote this accuracy tolerance as $r^{\ast}$ instead of $r$, because the derivative convergence time $t^{\ast}(r^{\ast})$ does not scale in the same way with the accuracy tolerance $r^{\ast}$ as the upper-transition time $\mathrm{\overline{T}}(r)$ with the accuracy tolerance $r$.

The derivative convergence time $t^{\ast}$ does not depend on the steady-state prevalence $y_{\infty}$ and therefore does not require information about future times. Therefore, given that $r^{\ast}$ is chosen in such a way that it corresponds to $r$, the derivative convergence time $t^{\ast}(r^{\ast})$ can be determined faster than the upper-transition time $\mathrm{\overline{T}}(r)$ as the calculation of $y_{\infty}$ can be skipped. Additionally, the convergence check is computationally cheaper.

Fig. \ref{fig:tstarzoomed} shows the values of the derivative convergence time $t^{\ast}$ for different values of the basic reproduction number $R_0$ and the starting prevalence $y(0)$. Specifically, we set the initial state vector to $V(0) = y(0)u$ and varied the initial prevalence $y(0)$. The inset in Fig. \ref{fig:tstarzoomed} shows sharp dips below the curves due to the starting value $y(0)$ being close to $y_\infty$ for that specific value of $R_0$, resulting in fast convergence. When the prevalence $y(0)$ is small, the initial changes in the prevalence are also small, resulting in the process instantaneously reaching the stopping criterion $|v_i(t+h) - v_i(t)| \leq hr^{\ast}$. Fig. \ref{fig:tstarzoomed} 
illustrates a trend similar to the one shown in Fig. \ref{fig:tbarzoomed}. Indeed, both the upper-transition time $\mathrm{\overline{T}}(r)$ and the derivative convergence time $t^{\ast}(r^{\ast})$ show the same qualitative asymptotic behaviour around $R_0 = 1$ and have the same dips when the starting prevalence $y(0)$ is close to the steady-state prevalence $y_{\infty}$. Similarly to the upper-transition time $\mathrm{\overline{T}}(r)$, the derivative convergence time $t^{\ast}$ is almost fully determined for each graph by the basic reproduction number $R_0$. The differences are clearer in Fig. \ref{fig:tbartstarcompare}, in which the log-scale shows that the derivative convergence time $t^{\ast}(r^{\ast})$ (thin red lines) has larger tails than the upper-transition time $\mathrm{\overline{T}}(r)$ (thick blue lines). 
\begin{figure*}[ht]
    \centering
    \includegraphics[width=\textwidth]{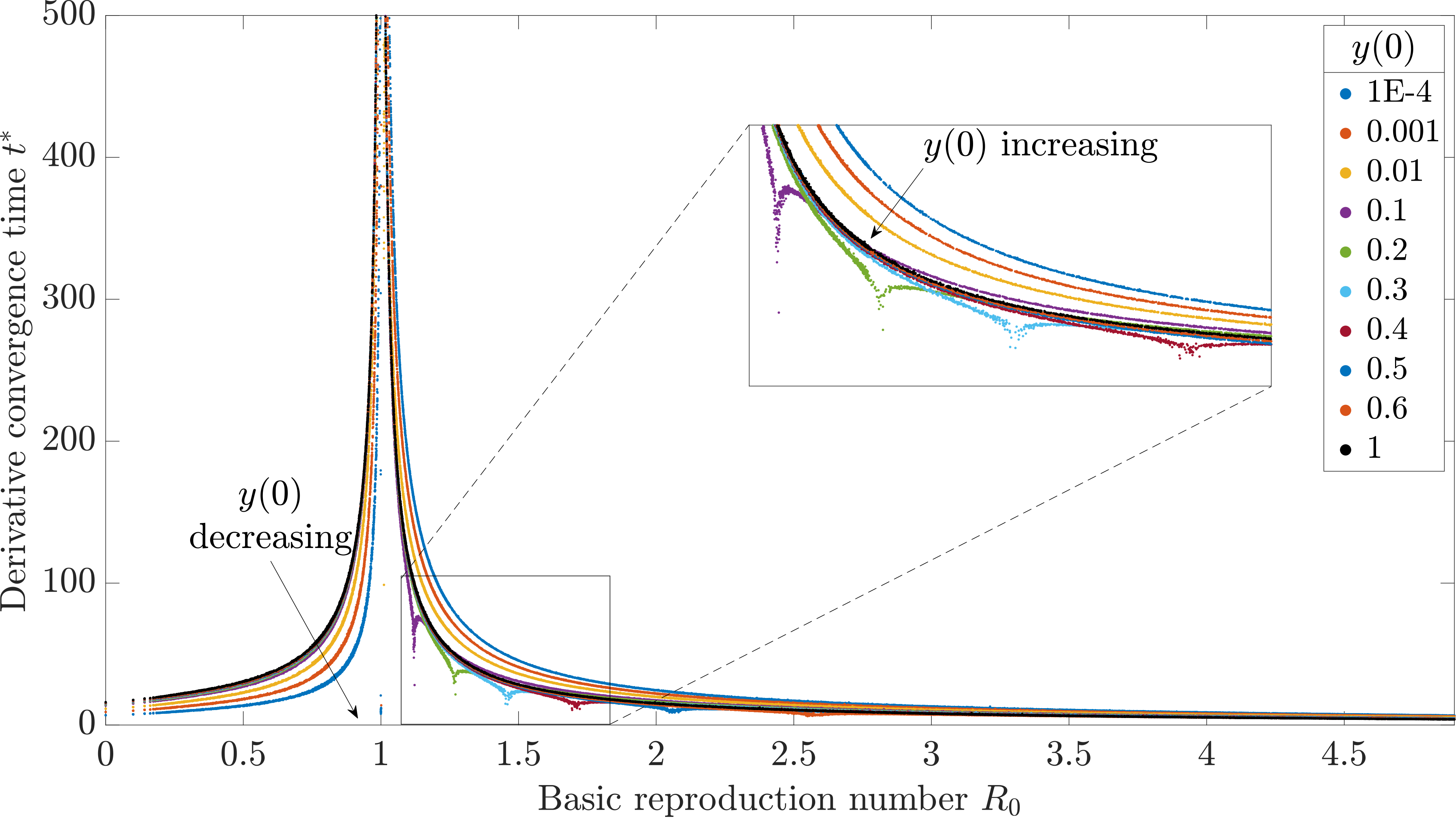}
    \caption{The derivative convergence time $t^{\ast}$ versus the basic reproduction number $R_0$ for different values of the starting prevalence $y(0)$. The chosen parameters are a graph size $N=50$, accuracy tolerance $r^{\ast}=10^{-7}$, infection rate $\beta = 0.1$ and step-size $h=0.01$. Each graph is a $ER(N,p)$ graph with $p\sim \text{Unif}(0,1)$. For $R_0 > 1$, the small values of $y(0)$ are on top. For $R_0<1$, the small values of $y(0)$ are at the bottom. The inset shows the dips below the curves in more detail.}
    \label{fig:tstarzoomed}
\end{figure*}

In order to upper bound the upper-transition time $\mathrm{\overline{T}}(r)$ with the derivative convergence time $t^{\ast}(r^{\ast})$, 
which could save computation time, it is of interest to know which value of $r^{\ast}$ is sufficient for some arbitrary value of $r$ and sequence of graphs $\mathcal{G}$. 

We compare the convergence criteria $|y(t+h) - y(t)| < hr^{\ast}$ and $|y(t)-y(t_{\text{max}})| < r$ for $t_{\text{max}} = 10^{4}$ and different values of $r$ and $r^{\ast}$. The numerical results in Fig. \ref{fig:tbartstarcompare} suggest that, when $r$ decreases by a factor of $10$, $r^{\ast}$ must decrease by a factor of $10^2$ to remain an upper bound for all values of $R_0$. In Fig. \ref{fig:tbartstarcompare}, the smallest $r^{\ast}$ for which the upper-transition time $\mathrm{\overline{T}}(r)$ is upper bounded by the derivative convergence time $t^{\ast}(r^{\ast})$ is $r^{\ast} = 10^{-2a}$ when $r=10^{-a}$, for some $a \in \mathbb{N}$. Indeed, we notice that the derivative convergence times $t^*(r^*)$ with $r^{\ast} = 10^{-2},10^{-4},10^{-6},10^{-8}$ bound the upper-transition times $\mathrm{\overline{T}}(r)$ with $r = 10^{-1},10^{-2},10^{-3},10^{-4}$ respectively.
\begin{figure*}
    \centering
    \includegraphics[width=\textwidth]{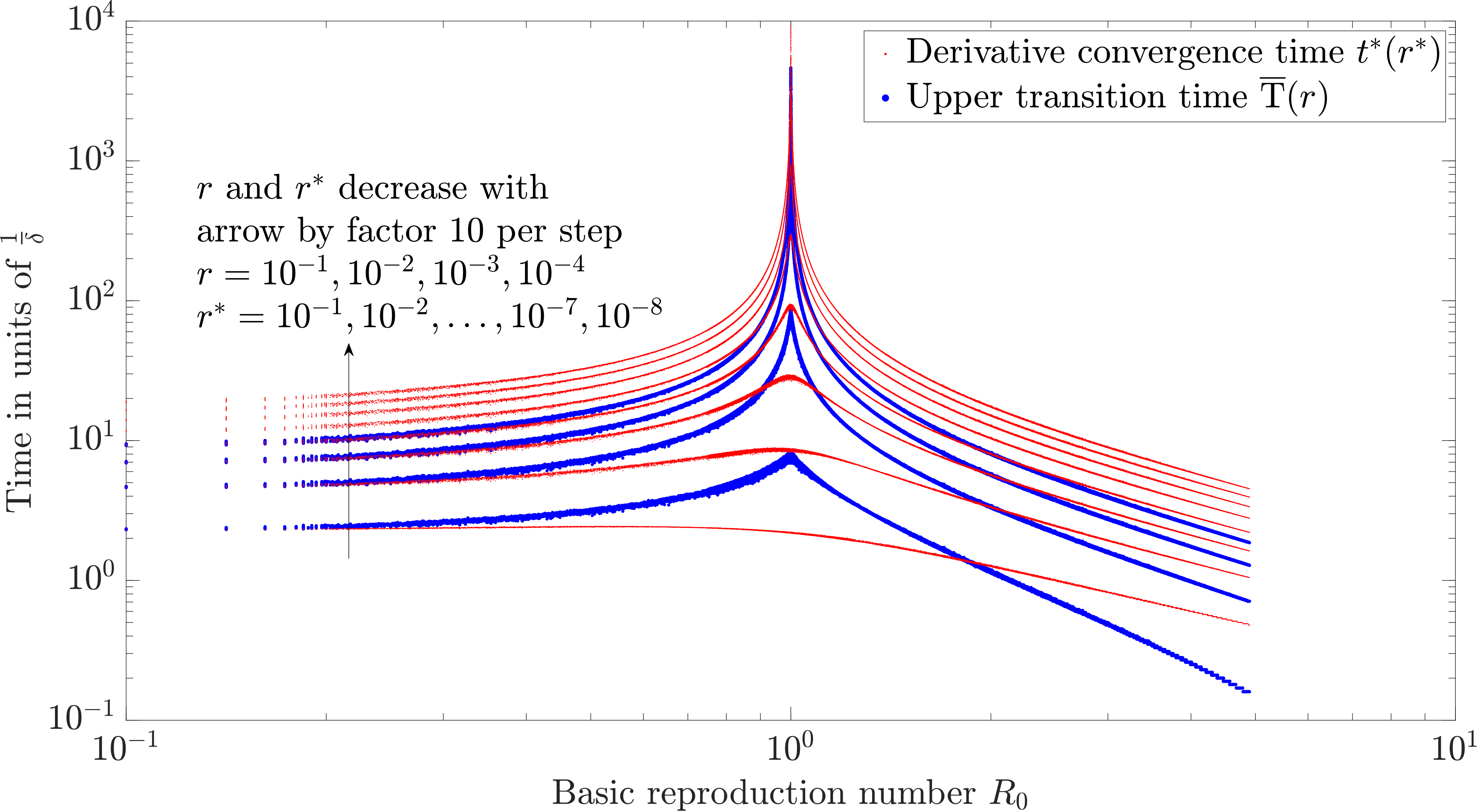}
    \caption{Comparison of the upper-transition time $\mathrm{\overline{T}}(r)$ (thick (dark) blue lines) and the derivative convergence time $t^{\ast}(r^{\ast})$ (thin (light) red lines) for different values of $r$ and $r^{\ast}$. The initial infection probability vector $V(0) = u$ and time-step parameter $h = 0.01$ are constant throughout all simulations. The values of the accuracy tolerance $r$ are $10^{-1},10^{-2},10^{-3},10^{-4}$ and the values of the accuracy tolerance $r^{\ast}$ are $10^{-1},10^{-2},\dots,10^{-7},10^{-8}$.}
    \label{fig:tbartstarcompare}
\end{figure*}

\section{Prevalence of the decay process from the all-infected state on regular graphs at the epidemic threshold.}\label{app:lemmafrac11plust}
In this section we state and proof the following lemma:
\begin{lemma}\label{lemma:frac11plust}
    For any $k$-regular graph $G$, the prevalence $y(t;G,\tau^{(1)}_c(G),u)$ of the NIMFA SIS process on $G$, with effective infection rate $\tau = \tau_c^{(1)}(G)$ with starting infection probability vector $V(0) = u$ satisfies 
    \begin{equation}\label{eq:frac11plust}
        y(t;G,\tau^{(1)}_c(G),u)=\frac{1}{1+t}.
    \end{equation}  
\end{lemma}
\begin{proof}
     The graph $G$ is $k$-regular and symmetry in (\ref{NIMFASISdelta1}) implies that $v_i(t) = v_j(t) = y(t)$ for all nodes $i$ and $j$ when the starting infection probabilities $v_i(0) = v_j(0)$ for all nodes $i$ and $j$. We substitute $v_i(t) = y(t)$ and the degree of each node $d_i = \sum_{j=1}^{N} a_{ij} = k$ in (\ref{NIMFASISdelta1}) to obtain the following differential equation with initial condition $y(0) = 1$:
\begin{equation*}
    \df{y}{t} = -y + \tau k(1-y)y.
\end{equation*}
For an effective infection rate $\tau =\tau_c^{(1)}(G) = \frac{1}{k}$, we find:
\begin{equation}\label{eq:eengedeelddooreenplustding}
    \df{y}{t} = -y + (1-y)y =-y^2.
\end{equation}
The differential equation \eqref{eq:eengedeelddooreenplustding} has solution $y(t) = \frac{1}{\frac{1}{y(0)}+t}$ and substituting $y(0) =1$ proves \eqref{eq:frac11plust}.
\end{proof}

\section{Substantiation for Conjecture 1}\label{app:substantationconjecture}
In this section, we present a proof outline for Conjecture \ref{thm:decaybound1}, highlighting which steps are missing. We restrict ourselves to the case of $\tau = \tau_c^{(1)}(G)$, because this case is an upper bound for other $\tau < \tau_c^{(1)}(G)$. We denote the projection of the viral state vector $V(t)$ on the principal eigenvector $x_1$ of the adjacency matrix $A$ as
\begin{equation}\label{eq:projection}
    c(t) = x_1^TV(t).
\end{equation}
Then, the viral state vector $V(t)$ can be written as a linear combination of two vectors
\begin{equation}\label{eq:decompv}
    V(t) = c(t)x_1(t) + \xi(0),
\end{equation}
where the $N \times 1$ vector
\begin{equation}
    \xi(t) = V(t) -c(t)x_1(t)
\end{equation}
is orthogonal to the principal eigenvector, i.e., $x_1^T\xi(t) = 0$.
At the initial time $t=0$, using $V(0) = u$ the decomposition of the viral state vector $V(0)$ in \eqref{eq:decompv} becomes
\begin{equation}\label{eq:projectiontiszero}
    V(0) = c(0)x_1 + \xi(0) = u.
\end{equation}
We rewrite the definition \eqref{eq:prevalance} of the prevalence $y(t)$ as
\begin{equation*}
    y(t) = \frac{1}{N}u^TV(t),
\end{equation*}
which yields with \eqref{eq:decompv} and \eqref{eq:projectiontiszero} that
\begin{equation}
    y(t) = \frac{1}{N}(c(0)x_1+\xi(0))^T(c(t)x_1 + \xi(t)).
\end{equation}
Since $x_1^T\xi(t) = 0$ at every time $t$ and $x_1^Tx_1 = 1$, we obtain that
\begin{equation}\label{eq:prevalenceprojected}
    Ny(t) = c(0)c(t)+\xi(0)^T\xi(t).
\end{equation}
To prove that $y(t) \leq \frac{1}{1+t}$ is an upper bound of the prevalence $y(t)$, we apply the triangle inequality to \eqref{eq:prevalenceprojected}, which implies that the prevalence $y(t)$ of any graph obeys
\begin{equation*}
    Ny(t) \leq c(0)c(t) + |\xi(0)^T\xi(t)|,
\end{equation*}
at every time $t$. Here we used  $|c(0)c(t)| = c(0)c(t)$, because \eqref{eq:projection} and $v_i(t) \geq 0$ for every node $i$ at every time $t$ implies that $c(t) \geq 0$ at every time $t$. Furthermore, we apply the Cauchy-Schwarz inequality to obtain that
\begin{equation*}
    Ny(t) \leq c(0)c(t) + ||\xi(0)||_2||\xi(t)||_2.
\end{equation*}
To prove Conjecture \ref{thm:decaybound1}, it remains to show that
\begin{equation}\label{eq:requirementforconjecture}
    c(0)c(t) + ||\xi(0)||_2||\xi(t)||_2 \leq \frac{N}{1+t},
\end{equation}
at every time $t$. Numerical simulations suggest that this inequality holds. To further specify a proof direction we consider a regular graph. For a regular graph it holds that $x_1 = \frac{1}{\sqrt{N}}u$ and $\xi(t) = 0$ for all times $t$. Solving the NIMFA equations \eqref{NIMFASISdelta1} we find that
\begin{equation*}
    c(t) = \frac{\sqrt{N}}{1+t}.
\end{equation*}
Hence, since $c(0) = \sqrt{N}$ for a regular graph we obtain
\begin{equation*}
    c(t) \leq \frac{c(0)}{1+t}.
\end{equation*}
Now, if a graph is not a regular graph, then the eigenvector $x_1$ is not a multiple of the all-one vector $u$, which implies that $||\xi(0)||_2 > 0$. Intuitively, the less regular a graph is, the smaller $c(t)$ and the larger $||\xi(t)||_2$ should be. It is therefore reasonable to assume one could prove that
\begin{equation}\label{eq:requirementonc}
    c(t) \leq c(0) \frac{1}{1+t}.
\end{equation}
We have not been able to prove \eqref{eq:requirementonc}. However, as for \eqref{eq:requirementforconjecture}, numerical simulations suggest that the inequality holds. In Theorem 2 from \cite{prasse2020time}, a similar decomposition is used in which the error term $||\xi(t)||_2$ was bounded by a function of the form
\begin{equation*}
    ||\xi(t)||_2 \leq ||\xi(0)||_2 e^{-\sigma t},
\end{equation*}
at all times $t$, for some constant $\sigma > 0$. Now, suppose that one can show that the same inequality holds in this case and that $\sigma = 1$. Then, after using $e^{-\sigma t} \leq e^{-t} \leq \frac{1}{1+t}$ we would find
\begin{equation}\label{eq:requirementonxi}
||\xi(t)||_2 \leq ||\xi(0)||_2 \frac{1}{1+t},
\end{equation}
which has also been numerically verified in numerous simulations. After substituting \eqref{eq:requirementonxi} and \eqref{eq:requirementonc}, the left side of \eqref{eq:requirementforconjecture} becomes
\begin{eqnarray}
    c(0)c(t) + ||\xi(0)||_2||\xi(t)||_2  \leq c^2(0)\frac{1}{1+t} \nonumber\\ +||\xi(0)||_2^2\frac{1}{1+t}  = ||V(0)||_2^2\frac{1}{1+t} = \frac{\sqrt{N}}{1+t},  
\end{eqnarray}
where the terms can be combined due to the orthogonality of $x_1$ and $\xi(0)$ and the last equality follows from $V(0) = u$. Note that this would prove Conjecture \ref{thm:decaybound1} by showing \eqref{eq:requirementforconjecture}. We emphasise that Equations \eqref{eq:requirementforconjecture}, \eqref{eq:requirementonc} and \eqref{eq:requirementonxi} have been numerically verified. Combined with the results shown in Fig. \ref{fig:conjecturefigure}, we believe that there is strong numerical evidence for Conjecture \ref{thm:decaybound1}.

\bibliography{references}

\begin{thebibliography}{53}%
\makeatletter
\providecommand \@ifxundefined [1]{%
 \@ifx{#1\undefined}
}%
\providecommand \@ifnum [1]{%
 \ifnum #1\expandafter \@firstoftwo
 \else \expandafter \@secondoftwo
 \fi
}%
\providecommand \@ifx [1]{%
 \ifx #1\expandafter \@firstoftwo
 \else \expandafter \@secondoftwo
 \fi
}%
\providecommand \natexlab [1]{#1}%
\providecommand \enquote  [1]{``#1''}%
\providecommand \bibnamefont  [1]{#1}%
\providecommand \bibfnamefont [1]{#1}%
\providecommand \citenamefont [1]{#1}%
\providecommand \href@noop [0]{\@secondoftwo}%
\providecommand \href [0]{\begingroup \@sanitize@url \@href}%
\providecommand \@href[1]{\@@startlink{#1}\@@href}%
\providecommand \@@href[1]{\endgroup#1\@@endlink}%
\providecommand \@sanitize@url [0]{\catcode `\\12\catcode `\$12\catcode
  `\&12\catcode `\#12\catcode `\^12\catcode `\_12\catcode `\%12\relax}%
\providecommand \@@startlink[1]{}%
\providecommand \@@endlink[0]{}%
\providecommand \url  [0]{\begingroup\@sanitize@url \@url }%
\providecommand \@url [1]{\endgroup\@href {#1}{\urlprefix }}%
\providecommand \urlprefix  [0]{URL }%
\providecommand \Eprint [0]{\href }%
\providecommand \doibase [0]{https://doi.org/}%
\providecommand \selectlanguage [0]{\@gobble}%
\providecommand \bibinfo  [0]{\@secondoftwo}%
\providecommand \bibfield  [0]{\@secondoftwo}%
\providecommand \translation [1]{[#1]}%
\providecommand \BibitemOpen [0]{}%
\providecommand \bibitemStop [0]{}%
\providecommand \bibitemNoStop [0]{.\EOS\space}%
\providecommand \EOS [0]{\spacefactor3000\relax}%
\providecommand \BibitemShut  [1]{\csname bibitem#1\endcsname}%
\let\auto@bib@innerbib\@empty
\bibitem [{\citenamefont {Anderson}\ and\ \citenamefont
  {May}(1991)}]{anderson1991infectious}%
  \BibitemOpen
  \bibfield  {author} {\bibinfo {author} {\bibfnamefont {R.~M.}\ \bibnamefont
  {Anderson}}\ and\ \bibinfo {author} {\bibfnamefont {R.~M.}\ \bibnamefont
  {May}},\ }\href@noop {} {\emph {\bibinfo {title} {Infectious diseases of
  humans: dynamics and control}}}\ (\bibinfo  {publisher} {Oxford university
  press},\ \bibinfo {year} {1991})\BibitemShut {NoStop}%
\bibitem [{\citenamefont {Pastor-Satorras}\ \emph {et~al.}(2015)\citenamefont
  {Pastor-Satorras}, \citenamefont {Castellano}, \citenamefont {Van~Mieghem},\
  and\ \citenamefont {Vespignani}}]{pastor2015epidemicreview}%
  \BibitemOpen
  \bibfield  {author} {\bibinfo {author} {\bibfnamefont {R.}~\bibnamefont
  {Pastor-Satorras}}, \bibinfo {author} {\bibfnamefont {C.}~\bibnamefont
  {Castellano}}, \bibinfo {author} {\bibfnamefont {P.}~\bibnamefont
  {Van~Mieghem}},\ and\ \bibinfo {author} {\bibfnamefont {A.}~\bibnamefont
  {Vespignani}},\ }\bibfield  {title} {\bibinfo {title} {Epidemic processes in
  complex networks},\ }\href@noop {} {\bibfield  {journal} {\bibinfo  {journal}
  {Reviews of modern physics}\ }\textbf {\bibinfo {volume} {87}},\ \bibinfo
  {pages} {925} (\bibinfo {year} {2015})}\BibitemShut {NoStop}%
\bibitem [{\citenamefont {Nowzari}\ \emph {et~al.}(2016)\citenamefont
  {Nowzari}, \citenamefont {Preciado},\ and\ \citenamefont
  {Pappas}}]{nowzari2016analysis}%
  \BibitemOpen
  \bibfield  {author} {\bibinfo {author} {\bibfnamefont {C.}~\bibnamefont
  {Nowzari}}, \bibinfo {author} {\bibfnamefont {V.~M.}\ \bibnamefont
  {Preciado}},\ and\ \bibinfo {author} {\bibfnamefont {G.~J.}\ \bibnamefont
  {Pappas}},\ }\bibfield  {title} {\bibinfo {title} {Analysis and control of
  epidemics: {A} survey of spreading processes on complex networks},\
  }\href@noop {} {\bibfield  {journal} {\bibinfo  {journal} {IEEE Control
  Systems Magazine}\ }\textbf {\bibinfo {volume} {36}},\ \bibinfo {pages} {26}
  (\bibinfo {year} {2016})}\BibitemShut {NoStop}%
\bibitem [{\citenamefont {Mei}\ \emph {et~al.}(2017)\citenamefont {Mei},
  \citenamefont {Mohagheghi}, \citenamefont {Zampieri},\ and\ \citenamefont
  {Bullo}}]{mei2017dynamics}%
  \BibitemOpen
  \bibfield  {author} {\bibinfo {author} {\bibfnamefont {W.}~\bibnamefont
  {Mei}}, \bibinfo {author} {\bibfnamefont {S.}~\bibnamefont {Mohagheghi}},
  \bibinfo {author} {\bibfnamefont {S.}~\bibnamefont {Zampieri}},\ and\
  \bibinfo {author} {\bibfnamefont {F.}~\bibnamefont {Bullo}},\ }\bibfield
  {title} {\bibinfo {title} {On the dynamics of deterministic epidemic
  propagation over networks},\ }\href@noop {} {\bibfield  {journal} {\bibinfo
  {journal} {Annual Reviews in Control}\ }\textbf {\bibinfo {volume} {44}},\
  \bibinfo {pages} {116} (\bibinfo {year} {2017})}\BibitemShut {NoStop}%
\bibitem [{\citenamefont {Holme}\ and\ \citenamefont
  {Saram{\"a}ki}(2012)}]{holme2012temporal}%
  \BibitemOpen
  \bibfield  {author} {\bibinfo {author} {\bibfnamefont {P.}~\bibnamefont
  {Holme}}\ and\ \bibinfo {author} {\bibfnamefont {J.}~\bibnamefont
  {Saram{\"a}ki}},\ }\bibfield  {title} {\bibinfo {title} {Temporal networks},\
  }\href@noop {} {\bibfield  {journal} {\bibinfo  {journal} {Physics reports}\
  }\textbf {\bibinfo {volume} {519}},\ \bibinfo {pages} {97} (\bibinfo {year}
  {2012})}\BibitemShut {NoStop}%
\bibitem [{\citenamefont {Holme}(2015)}]{holme2015modern}%
  \BibitemOpen
  \bibfield  {author} {\bibinfo {author} {\bibfnamefont {P.}~\bibnamefont
  {Holme}},\ }\bibfield  {title} {\bibinfo {title} {Modern temporal network
  theory: a colloquium},\ }\href@noop {} {\bibfield  {journal} {\bibinfo
  {journal} {The European Physical Journal B}\ }\textbf {\bibinfo {volume}
  {88}},\ \bibinfo {pages} {1} (\bibinfo {year} {2015})}\BibitemShut {NoStop}%
\bibitem [{\citenamefont {Kohar}\ and\ \citenamefont
  {Sinha}(2013)}]{kohar2013emergence}%
  \BibitemOpen
  \bibfield  {author} {\bibinfo {author} {\bibfnamefont {V.}~\bibnamefont
  {Kohar}}\ and\ \bibinfo {author} {\bibfnamefont {S.}~\bibnamefont {Sinha}},\
  }\bibfield  {title} {\bibinfo {title} {Emergence of epidemics in rapidly
  varying networks},\ }\href@noop {} {\bibfield  {journal} {\bibinfo  {journal}
  {Chaos, Solitons \& Fractals}\ }\textbf {\bibinfo {volume} {54}},\ \bibinfo
  {pages} {127} (\bibinfo {year} {2013})}\BibitemShut {NoStop}%
\bibitem [{\citenamefont {Valdano}\ \emph {et~al.}(2018)\citenamefont
  {Valdano}, \citenamefont {Fiorentin}, \citenamefont {Poletto},\ and\
  \citenamefont {Colizza}}]{valdano2018epidemic}%
  \BibitemOpen
  \bibfield  {author} {\bibinfo {author} {\bibfnamefont {E.}~\bibnamefont
  {Valdano}}, \bibinfo {author} {\bibfnamefont {M.~R.}\ \bibnamefont
  {Fiorentin}}, \bibinfo {author} {\bibfnamefont {C.}~\bibnamefont {Poletto}},\
  and\ \bibinfo {author} {\bibfnamefont {V.}~\bibnamefont {Colizza}},\
  }\bibfield  {title} {\bibinfo {title} {Epidemic threshold in continuous-time
  evolving networks},\ }\href@noop {} {\bibfield  {journal} {\bibinfo
  {journal} {Physical review letters}\ }\textbf {\bibinfo {volume} {120}},\
  \bibinfo {pages} {068302} (\bibinfo {year} {2018})}\BibitemShut {NoStop}%
\bibitem [{\citenamefont {Zhang}\ \emph {et~al.}(2017)\citenamefont {Zhang},
  \citenamefont {Li},\ and\ \citenamefont {Vasilakos}}]{zhang2017spectral}%
  \BibitemOpen
  \bibfield  {author} {\bibinfo {author} {\bibfnamefont {Y.-Q.}\ \bibnamefont
  {Zhang}}, \bibinfo {author} {\bibfnamefont {X.}~\bibnamefont {Li}},\ and\
  \bibinfo {author} {\bibfnamefont {A.~V.}\ \bibnamefont {Vasilakos}},\
  }\bibfield  {title} {\bibinfo {title} {Spectral analysis of epidemic
  thresholds of temporal networks},\ }\href@noop {} {\bibfield  {journal}
  {\bibinfo  {journal} {IEEE transactions on cybernetics}\ }\textbf {\bibinfo
  {volume} {50}},\ \bibinfo {pages} {1965} (\bibinfo {year}
  {2017})}\BibitemShut {NoStop}%
\bibitem [{\citenamefont {Pastor-Satorras}\ and\ \citenamefont
  {Vespignani}(2001{\natexlab{a}})}]{pastor2001epidemic}%
  \BibitemOpen
  \bibfield  {author} {\bibinfo {author} {\bibfnamefont {R.}~\bibnamefont
  {Pastor-Satorras}}\ and\ \bibinfo {author} {\bibfnamefont {A.}~\bibnamefont
  {Vespignani}},\ }\bibfield  {title} {\bibinfo {title} {Epidemic spreading in
  scale-free networks},\ }\href@noop {} {\bibfield  {journal} {\bibinfo
  {journal} {Physical review letters}\ }\textbf {\bibinfo {volume} {86}},\
  \bibinfo {pages} {3200} (\bibinfo {year} {2001}{\natexlab{a}})}\BibitemShut
  {NoStop}%
\bibitem [{\citenamefont {Pastor-Satorras}\ and\ \citenamefont
  {Vespignani}(2001{\natexlab{b}})}]{pastor2001epidemic2}%
  \BibitemOpen
  \bibfield  {author} {\bibinfo {author} {\bibfnamefont {R.}~\bibnamefont
  {Pastor-Satorras}}\ and\ \bibinfo {author} {\bibfnamefont {A.}~\bibnamefont
  {Vespignani}},\ }\bibfield  {title} {\bibinfo {title} {Epidemic dynamics and
  endemic states in complex networks},\ }\href@noop {} {\bibfield  {journal}
  {\bibinfo  {journal} {Physical Review E}\ }\textbf {\bibinfo {volume} {63}},\
  \bibinfo {pages} {066117} (\bibinfo {year} {2001}{\natexlab{b}})}\BibitemShut
  {NoStop}%
\bibitem [{\citenamefont {Schwarzkopf}\ \emph {et~al.}(2010)\citenamefont
  {Schwarzkopf}, \citenamefont {R{\'a}kos},\ and\ \citenamefont
  {Mukamel}}]{schwarzkopf2010epidemic}%
  \BibitemOpen
  \bibfield  {author} {\bibinfo {author} {\bibfnamefont {Y.}~\bibnamefont
  {Schwarzkopf}}, \bibinfo {author} {\bibfnamefont {A.}~\bibnamefont
  {R{\'a}kos}},\ and\ \bibinfo {author} {\bibfnamefont {D.}~\bibnamefont
  {Mukamel}},\ }\bibfield  {title} {\bibinfo {title} {Epidemic spreading in
  evolving networks},\ }\href@noop {} {\bibfield  {journal} {\bibinfo
  {journal} {Physical Review E}\ }\textbf {\bibinfo {volume} {82}},\ \bibinfo
  {pages} {036112} (\bibinfo {year} {2010})}\BibitemShut {NoStop}%
\bibitem [{\citenamefont {Li}\ \emph {et~al.}(2012)\citenamefont {Li},
  \citenamefont {van~de Bovenkamp},\ and\ \citenamefont
  {Van~Mieghem}}]{li2012susceptible}%
  \BibitemOpen
  \bibfield  {author} {\bibinfo {author} {\bibfnamefont {C.}~\bibnamefont
  {Li}}, \bibinfo {author} {\bibfnamefont {R.}~\bibnamefont {van~de
  Bovenkamp}},\ and\ \bibinfo {author} {\bibfnamefont {P.}~\bibnamefont
  {Van~Mieghem}},\ }\bibfield  {title} {\bibinfo {title}
  {{Susceptible-infected-susceptible model: A comparison of N-intertwined and
  heterogeneous mean-field approximations}},\ }\href@noop {} {\bibfield
  {journal} {\bibinfo  {journal} {Physical Review E}\ }\textbf {\bibinfo
  {volume} {86}},\ \bibinfo {pages} {026116} (\bibinfo {year}
  {2012})}\BibitemShut {NoStop}%
\bibitem [{\citenamefont {Devriendt}\ and\ \citenamefont
  {Van~Mieghem}(2017)}]{devriendt2017unified}%
  \BibitemOpen
  \bibfield  {author} {\bibinfo {author} {\bibfnamefont {K.}~\bibnamefont
  {Devriendt}}\ and\ \bibinfo {author} {\bibfnamefont {P.}~\bibnamefont
  {Van~Mieghem}},\ }\bibfield  {title} {\bibinfo {title} {Unified mean-field
  framework for susceptible-infected-susceptible epidemics on networks, based
  on graph partitioning and the isoperimetric inequality},\ }\href@noop {}
  {\bibfield  {journal} {\bibinfo  {journal} {Physical Review E}\ }\textbf
  {\bibinfo {volume} {96}},\ \bibinfo {pages} {052314} (\bibinfo {year}
  {2017})}\BibitemShut {NoStop}%
\bibitem [{\citenamefont {Holme}\ and\ \citenamefont
  {Liljeros}(2014)}]{holme2014birth}%
  \BibitemOpen
  \bibfield  {author} {\bibinfo {author} {\bibfnamefont {P.}~\bibnamefont
  {Holme}}\ and\ \bibinfo {author} {\bibfnamefont {F.}~\bibnamefont
  {Liljeros}},\ }\bibfield  {title} {\bibinfo {title} {Birth and death of links
  control disease spreading in empirical contact networks},\ }\href@noop {}
  {\bibfield  {journal} {\bibinfo  {journal} {Scientific reports}\ }\textbf
  {\bibinfo {volume} {4}},\ \bibinfo {pages} {4999} (\bibinfo {year}
  {2014})}\BibitemShut {NoStop}%
\bibitem [{\citenamefont {Leitch}\ \emph {et~al.}(2019)\citenamefont {Leitch},
  \citenamefont {Alexander},\ and\ \citenamefont
  {Sengupta}}]{leitch2019toward}%
  \BibitemOpen
  \bibfield  {author} {\bibinfo {author} {\bibfnamefont {J.}~\bibnamefont
  {Leitch}}, \bibinfo {author} {\bibfnamefont {K.~A.}\ \bibnamefont
  {Alexander}},\ and\ \bibinfo {author} {\bibfnamefont {S.}~\bibnamefont
  {Sengupta}},\ }\bibfield  {title} {\bibinfo {title} {Toward epidemic
  thresholds on temporal networks: a review and open questions},\ }\href@noop
  {} {\bibfield  {journal} {\bibinfo  {journal} {Applied Network Science}\
  }\textbf {\bibinfo {volume} {4}},\ \bibinfo {pages} {1} (\bibinfo {year}
  {2019})}\BibitemShut {NoStop}%
\bibitem [{\citenamefont {Perra}\ \emph {et~al.}(2012)\citenamefont {Perra},
  \citenamefont {Gon{\c{c}}alves}, \citenamefont {Pastor-Satorras},\ and\
  \citenamefont {Vespignani}}]{perra2012activity}%
  \BibitemOpen
  \bibfield  {author} {\bibinfo {author} {\bibfnamefont {N.}~\bibnamefont
  {Perra}}, \bibinfo {author} {\bibfnamefont {B.}~\bibnamefont
  {Gon{\c{c}}alves}}, \bibinfo {author} {\bibfnamefont {R.}~\bibnamefont
  {Pastor-Satorras}},\ and\ \bibinfo {author} {\bibfnamefont {A.}~\bibnamefont
  {Vespignani}},\ }\bibfield  {title} {\bibinfo {title} {Activity driven
  modeling of time varying networks},\ }\href@noop {} {\bibfield  {journal}
  {\bibinfo  {journal} {Scientific reports}\ }\textbf {\bibinfo {volume} {2}},\
  \bibinfo {pages} {469} (\bibinfo {year} {2012})}\BibitemShut {NoStop}%
\bibitem [{\citenamefont {Stehl{\'e}}\ \emph {et~al.}(2011)\citenamefont
  {Stehl{\'e}}, \citenamefont {Voirin}, \citenamefont {Barrat}, \citenamefont
  {Cattuto}, \citenamefont {Colizza}, \citenamefont {Isella}, \citenamefont
  {R{\'e}gis}, \citenamefont {Pinton}, \citenamefont {Khanafer}, \citenamefont
  {Van~den Broeck} \emph {et~al.}}]{stehle2011simulation}%
  \BibitemOpen
  \bibfield  {author} {\bibinfo {author} {\bibfnamefont {J.}~\bibnamefont
  {Stehl{\'e}}}, \bibinfo {author} {\bibfnamefont {N.}~\bibnamefont {Voirin}},
  \bibinfo {author} {\bibfnamefont {A.}~\bibnamefont {Barrat}}, \bibinfo
  {author} {\bibfnamefont {C.}~\bibnamefont {Cattuto}}, \bibinfo {author}
  {\bibfnamefont {V.}~\bibnamefont {Colizza}}, \bibinfo {author} {\bibfnamefont
  {L.}~\bibnamefont {Isella}}, \bibinfo {author} {\bibfnamefont
  {C.}~\bibnamefont {R{\'e}gis}}, \bibinfo {author} {\bibfnamefont {J.-F.}\
  \bibnamefont {Pinton}}, \bibinfo {author} {\bibfnamefont {N.}~\bibnamefont
  {Khanafer}}, \bibinfo {author} {\bibfnamefont {W.}~\bibnamefont {Van~den
  Broeck}}, \emph {et~al.},\ }\bibfield  {title} {\bibinfo {title} {{Simulation
  of an SEIR infectious disease model on the dynamic contact network of
  conference attendees}},\ }\href@noop {} {\bibfield  {journal} {\bibinfo
  {journal} {BMC medicine}\ }\textbf {\bibinfo {volume} {9}},\ \bibinfo {pages}
  {1} (\bibinfo {year} {2011})}\BibitemShut {NoStop}%
\bibitem [{\citenamefont {Par{\'e}}\ \emph {et~al.}(2015)\citenamefont
  {Par{\'e}}, \citenamefont {Beck},\ and\ \citenamefont
  {Nedi{\'c}}}]{pare2015stability}%
  \BibitemOpen
  \bibfield  {author} {\bibinfo {author} {\bibfnamefont {P.~E.}\ \bibnamefont
  {Par{\'e}}}, \bibinfo {author} {\bibfnamefont {C.~L.}\ \bibnamefont {Beck}},\
  and\ \bibinfo {author} {\bibfnamefont {A.}~\bibnamefont {Nedi{\'c}}},\
  }\bibfield  {title} {\bibinfo {title} {Stability analysis and control of
  virus spread over time-varying networks},\ }in\ \href@noop {} {\emph
  {\bibinfo {booktitle} {2015 54th IEEE Conference on Decision and Control
  (CDC)}}}\ (\bibinfo {organization} {IEEE},\ \bibinfo {year} {2015})\ pp.\
  \bibinfo {pages} {3554--3559}\BibitemShut {NoStop}%
\bibitem [{\citenamefont {Ogura}\ and\ \citenamefont
  {Preciado}(2016)}]{ogura2016stability}%
  \BibitemOpen
  \bibfield  {author} {\bibinfo {author} {\bibfnamefont {M.}~\bibnamefont
  {Ogura}}\ and\ \bibinfo {author} {\bibfnamefont {V.~M.}\ \bibnamefont
  {Preciado}},\ }\bibfield  {title} {\bibinfo {title} {Stability of spreading
  processes over time-varying large-scale networks},\ }\href@noop {} {\bibfield
   {journal} {\bibinfo  {journal} {IEEE Transactions on Network Science and
  Engineering}\ }\textbf {\bibinfo {volume} {3}},\ \bibinfo {pages} {44}
  (\bibinfo {year} {2016})}\BibitemShut {NoStop}%
\bibitem [{\citenamefont {Par{\'e}}\ \emph
  {et~al.}(2017{\natexlab{a}})\citenamefont {Par{\'e}}, \citenamefont {Liu},
  \citenamefont {Beck}, \citenamefont {Nedi{\'c}},\ and\ \citenamefont
  {Ba{\c{s}}ar}}]{pare2017multi}%
  \BibitemOpen
  \bibfield  {author} {\bibinfo {author} {\bibfnamefont {P.~E.}\ \bibnamefont
  {Par{\'e}}}, \bibinfo {author} {\bibfnamefont {J.}~\bibnamefont {Liu}},
  \bibinfo {author} {\bibfnamefont {C.~L.}\ \bibnamefont {Beck}}, \bibinfo
  {author} {\bibfnamefont {A.}~\bibnamefont {Nedi{\'c}}},\ and\ \bibinfo
  {author} {\bibfnamefont {T.}~\bibnamefont {Ba{\c{s}}ar}},\ }\bibfield
  {title} {\bibinfo {title} {Multi-competitive viruses over static and
  time-varying networks},\ }in\ \href@noop {} {\emph {\bibinfo {booktitle}
  {2017 American Control Conference (ACC)}}}\ (\bibinfo {organization} {IEEE},\
  \bibinfo {year} {2017})\ pp.\ \bibinfo {pages} {1685--1690}\BibitemShut
  {NoStop}%
\bibitem [{\citenamefont {Par{\'e}}\ \emph
  {et~al.}(2017{\natexlab{b}})\citenamefont {Par{\'e}}, \citenamefont {Beck},\
  and\ \citenamefont {Nedi{\'c}}}]{pare2017epidemic}%
  \BibitemOpen
  \bibfield  {author} {\bibinfo {author} {\bibfnamefont {P.~E.}\ \bibnamefont
  {Par{\'e}}}, \bibinfo {author} {\bibfnamefont {C.~L.}\ \bibnamefont {Beck}},\
  and\ \bibinfo {author} {\bibfnamefont {A.}~\bibnamefont {Nedi{\'c}}},\
  }\bibfield  {title} {\bibinfo {title} {Epidemic processes over time-varying
  networks},\ }\href@noop {} {\bibfield  {journal} {\bibinfo  {journal} {IEEE
  Transactions on Control of Network Systems}\ }\textbf {\bibinfo {volume}
  {5}},\ \bibinfo {pages} {1322} (\bibinfo {year}
  {2017}{\natexlab{b}})}\BibitemShut {NoStop}%
\bibitem [{\citenamefont {Lieberman}\ \emph {et~al.}(2005)\citenamefont
  {Lieberman}, \citenamefont {Hauert},\ and\ \citenamefont
  {Nowak}}]{lieberman2005evolutionary}%
  \BibitemOpen
  \bibfield  {author} {\bibinfo {author} {\bibfnamefont {E.}~\bibnamefont
  {Lieberman}}, \bibinfo {author} {\bibfnamefont {C.}~\bibnamefont {Hauert}},\
  and\ \bibinfo {author} {\bibfnamefont {M.~A.}\ \bibnamefont {Nowak}},\
  }\bibfield  {title} {\bibinfo {title} {Evolutionary dynamics on graphs},\
  }\href@noop {} {\bibfield  {journal} {\bibinfo  {journal} {Nature}\ }\textbf
  {\bibinfo {volume} {433}},\ \bibinfo {pages} {312} (\bibinfo {year}
  {2005})}\BibitemShut {NoStop}%
\bibitem [{\citenamefont {Prakash}\ \emph {et~al.}(2010)\citenamefont
  {Prakash}, \citenamefont {Tong}, \citenamefont {Valler}, \citenamefont
  {Faloutsos},\ and\ \citenamefont {Faloutsos}}]{prakash2010virus}%
  \BibitemOpen
  \bibfield  {author} {\bibinfo {author} {\bibfnamefont {B.~A.}\ \bibnamefont
  {Prakash}}, \bibinfo {author} {\bibfnamefont {H.}~\bibnamefont {Tong}},
  \bibinfo {author} {\bibfnamefont {N.}~\bibnamefont {Valler}}, \bibinfo
  {author} {\bibfnamefont {M.}~\bibnamefont {Faloutsos}},\ and\ \bibinfo
  {author} {\bibfnamefont {C.}~\bibnamefont {Faloutsos}},\ }\bibfield  {title}
  {\bibinfo {title} {Virus propagation on time-varying networks: {T}heory and
  immunization algorithms},\ }in\ \href@noop {} {\emph {\bibinfo {booktitle}
  {Machine Learning and Knowledge Discovery in Databases: European Conference,
  ECML PKDD 2010, Barcelona, Spain, September 20-24, 2010, Proceedings, Part
  III 21}}}\ (\bibinfo {organization} {Springer},\ \bibinfo {year} {2010})\
  pp.\ \bibinfo {pages} {99--114}\BibitemShut {NoStop}%
\bibitem [{\citenamefont {Kotnis}\ and\ \citenamefont
  {Kuri}(2013)}]{kotnis2013stochastic}%
  \BibitemOpen
  \bibfield  {author} {\bibinfo {author} {\bibfnamefont {B.}~\bibnamefont
  {Kotnis}}\ and\ \bibinfo {author} {\bibfnamefont {J.}~\bibnamefont {Kuri}},\
  }\bibfield  {title} {\bibinfo {title} {Stochastic analysis of epidemics on
  adaptive time varying networks},\ }\href@noop {} {\bibfield  {journal}
  {\bibinfo  {journal} {Physical Review E}\ }\textbf {\bibinfo {volume} {87}},\
  \bibinfo {pages} {062810} (\bibinfo {year} {2013})}\BibitemShut {NoStop}%
\bibitem [{\citenamefont {Ren}\ and\ \citenamefont
  {Wang}(2014)}]{ren2014epidemic}%
  \BibitemOpen
  \bibfield  {author} {\bibinfo {author} {\bibfnamefont {G.}~\bibnamefont
  {Ren}}\ and\ \bibinfo {author} {\bibfnamefont {X.}~\bibnamefont {Wang}},\
  }\bibfield  {title} {\bibinfo {title} {Epidemic spreading in time-varying
  community networks},\ }\href@noop {} {\bibfield  {journal} {\bibinfo
  {journal} {Chaos: An Interdisciplinary Journal of Nonlinear Science}\
  }\textbf {\bibinfo {volume} {24}},\ \bibinfo {pages} {023116} (\bibinfo
  {year} {2014})}\BibitemShut {NoStop}%
\bibitem [{\citenamefont {Vestergaard}\ and\ \citenamefont
  {G{\'e}nois}(2015)}]{vestergaard2015temporal}%
  \BibitemOpen
  \bibfield  {author} {\bibinfo {author} {\bibfnamefont {C.~L.}\ \bibnamefont
  {Vestergaard}}\ and\ \bibinfo {author} {\bibfnamefont {M.}~\bibnamefont
  {G{\'e}nois}},\ }\bibfield  {title} {\bibinfo {title} {Temporal {G}illespie
  algorithm: fast simulation of contagion processes on time-varying networks},\
  }\href@noop {} {\bibfield  {journal} {\bibinfo  {journal} {PLoS computational
  biology}\ }\textbf {\bibinfo {volume} {11}},\ \bibinfo {pages} {e1004579}
  (\bibinfo {year} {2015})}\BibitemShut {NoStop}%
\bibitem [{\citenamefont {Nadini}\ \emph {et~al.}(2018)\citenamefont {Nadini},
  \citenamefont {Sun}, \citenamefont {Ubaldi}, \citenamefont {Starnini},
  \citenamefont {Rizzo},\ and\ \citenamefont {Perra}}]{nadini2018epidemic}%
  \BibitemOpen
  \bibfield  {author} {\bibinfo {author} {\bibfnamefont {M.}~\bibnamefont
  {Nadini}}, \bibinfo {author} {\bibfnamefont {K.}~\bibnamefont {Sun}},
  \bibinfo {author} {\bibfnamefont {E.}~\bibnamefont {Ubaldi}}, \bibinfo
  {author} {\bibfnamefont {M.}~\bibnamefont {Starnini}}, \bibinfo {author}
  {\bibfnamefont {A.}~\bibnamefont {Rizzo}},\ and\ \bibinfo {author}
  {\bibfnamefont {N.}~\bibnamefont {Perra}},\ }\bibfield  {title} {\bibinfo
  {title} {Epidemic spreading in modular time-varying networks},\ }\href@noop
  {} {\bibfield  {journal} {\bibinfo  {journal} {Scientific reports}\ }\textbf
  {\bibinfo {volume} {8}},\ \bibinfo {pages} {2352} (\bibinfo {year}
  {2018})}\BibitemShut {NoStop}%
\bibitem [{\citenamefont {Guo}\ \emph {et~al.}(2021)\citenamefont {Guo},
  \citenamefont {Yin}, \citenamefont {Xia},\ and\ \citenamefont
  {Dehmer}}]{guo2021impact}%
  \BibitemOpen
  \bibfield  {author} {\bibinfo {author} {\bibfnamefont {H.}~\bibnamefont
  {Guo}}, \bibinfo {author} {\bibfnamefont {Q.}~\bibnamefont {Yin}}, \bibinfo
  {author} {\bibfnamefont {C.}~\bibnamefont {Xia}},\ and\ \bibinfo {author}
  {\bibfnamefont {M.}~\bibnamefont {Dehmer}},\ }\bibfield  {title} {\bibinfo
  {title} {Impact of information diffusion on epidemic spreading in partially
  mapping two-layered time-varying networks},\ }\href@noop {} {\bibfield
  {journal} {\bibinfo  {journal} {Nonlinear Dynamics}\ }\textbf {\bibinfo
  {volume} {105}},\ \bibinfo {pages} {3819} (\bibinfo {year}
  {2021})}\BibitemShut {NoStop}%
\bibitem [{\citenamefont {Han}\ \emph {et~al.}(2023)\citenamefont {Han},
  \citenamefont {Lin}, \citenamefont {Tang}, \citenamefont {Liu},\ and\
  \citenamefont {Guan}}]{han2023impact}%
  \BibitemOpen
  \bibfield  {author} {\bibinfo {author} {\bibfnamefont {L.}~\bibnamefont
  {Han}}, \bibinfo {author} {\bibfnamefont {Z.}~\bibnamefont {Lin}}, \bibinfo
  {author} {\bibfnamefont {M.}~\bibnamefont {Tang}}, \bibinfo {author}
  {\bibfnamefont {Y.}~\bibnamefont {Liu}},\ and\ \bibinfo {author}
  {\bibfnamefont {S.}~\bibnamefont {Guan}},\ }\bibfield  {title} {\bibinfo
  {title} {Impact of human contact patterns on epidemic spreading in
  time-varying networks},\ }\href@noop {} {\bibfield  {journal} {\bibinfo
  {journal} {Physical Review E}\ }\textbf {\bibinfo {volume} {107}},\ \bibinfo
  {pages} {024312} (\bibinfo {year} {2023})}\BibitemShut {NoStop}%
\bibitem [{\citenamefont {Van~Mieghem}\ \emph {et~al.}(2009)\citenamefont
  {Van~Mieghem}, \citenamefont {Omic},\ and\ \citenamefont
  {Kooij}}]{VanMieghem2008Virusspreadinnetworks}%
  \BibitemOpen
  \bibfield  {author} {\bibinfo {author} {\bibfnamefont {P.}~\bibnamefont
  {Van~Mieghem}}, \bibinfo {author} {\bibfnamefont {J.}~\bibnamefont {Omic}},\
  and\ \bibinfo {author} {\bibfnamefont {R.}~\bibnamefont {Kooij}},\ }\bibfield
   {title} {\bibinfo {title} {Virus spread in networks},\ }\href@noop {}
  {\bibfield  {journal} {\bibinfo  {journal} {IEEE/ACM Transactions On
  Networking}\ }\textbf {\bibinfo {volume} {17}},\ \bibinfo {pages} {1}
  (\bibinfo {year} {2009})}\BibitemShut {NoStop}%
\bibitem [{\citenamefont {Van~Mieghem}(2023)}]{VanMieghem2010graphspectra}%
  \BibitemOpen
  \bibfield  {author} {\bibinfo {author} {\bibfnamefont {P.}~\bibnamefont
  {Van~Mieghem}},\ }\href@noop {} {\emph {\bibinfo {title} {Graph spectra for
  complex networks}}}\ (\bibinfo  {publisher} {Cambridge University Press},\
  \bibinfo {address} {Cambridge, U.K.},\ \bibinfo {year} {2023})\BibitemShut
  {NoStop}%
\bibitem [{\citenamefont {Ganesh}\ \emph {et~al.}(2005)\citenamefont {Ganesh},
  \citenamefont {Massouli{\'e}},\ and\ \citenamefont
  {Towsley}}]{ganesh2005effect}%
  \BibitemOpen
  \bibfield  {author} {\bibinfo {author} {\bibfnamefont {A.}~\bibnamefont
  {Ganesh}}, \bibinfo {author} {\bibfnamefont {L.}~\bibnamefont
  {Massouli{\'e}}},\ and\ \bibinfo {author} {\bibfnamefont {D.}~\bibnamefont
  {Towsley}},\ }\bibfield  {title} {\bibinfo {title} {The effect of network
  topology on the spread of epidemics},\ }in\ \href@noop {} {\emph {\bibinfo
  {booktitle} {Proceedings IEEE 24th Annual Joint Conference of the IEEE
  Computer and Communications Societies.}}},\ Vol.~\bibinfo {volume} {2}\
  (\bibinfo {organization} {IEEE},\ \bibinfo {year} {2005})\ pp.\ \bibinfo
  {pages} {1455--1466}\BibitemShut {NoStop}%
\bibitem [{Note1()}]{Note1}%
  \BibitemOpen
  \bibinfo {note} {Technically, there is a non-zero probability that the
  epidemic dies out fast for $\tau > \tau _c$ in the Markovian SIS model,
  because the probability of curing before infecting anyone is non-zero for any
  $\tau $.}\BibitemShut {Stop}%
\bibitem [{\citenamefont
  {Van~Mieghem}(2013)}]{VanMieghem2013decaytowardstheoverallhealthystate}%
  \BibitemOpen
  \bibfield  {author} {\bibinfo {author} {\bibfnamefont {P.}~\bibnamefont
  {Van~Mieghem}},\ }\bibfield  {title} {\bibinfo {title} {{Decay towards the
  overall-healthy state in SIS epidemics on networks}},\ }\href@noop {}
  {\bibfield  {journal} {\bibinfo  {journal} {arXiv preprint arXiv:1310.3980}\
  } (\bibinfo {year} {2013})}\BibitemShut {NoStop}%
\bibitem [{\citenamefont {Van~Mieghem}(2020)}]{vanmieghem2020explosive}%
  \BibitemOpen
  \bibfield  {author} {\bibinfo {author} {\bibfnamefont {P.}~\bibnamefont
  {Van~Mieghem}},\ }\bibfield  {title} {\bibinfo {title} {Explosive phase
  transition in susceptible-infected-susceptible epidemics with arbitrary small
  but nonzero self-infection rate},\ }\href@noop {} {\bibfield  {journal}
  {\bibinfo  {journal} {Physical Review E}\ }\textbf {\bibinfo {volume}
  {101}},\ \bibinfo {pages} {032303} (\bibinfo {year} {2020})}\BibitemShut
  {NoStop}%
\bibitem [{\citenamefont {Cator}\ and\ \citenamefont
  {Van~Mieghem}(2012)}]{cator2012second}%
  \BibitemOpen
  \bibfield  {author} {\bibinfo {author} {\bibfnamefont {E.}~\bibnamefont
  {Cator}}\ and\ \bibinfo {author} {\bibfnamefont {P.}~\bibnamefont
  {Van~Mieghem}},\ }\bibfield  {title} {\bibinfo {title} {Second-order
  mean-field susceptible-infected-susceptible epidemic threshold},\ }\href@noop
  {} {\bibfield  {journal} {\bibinfo  {journal} {Physical review E}\ }\textbf
  {\bibinfo {volume} {85}},\ \bibinfo {pages} {056111} (\bibinfo {year}
  {2012})}\BibitemShut {NoStop}%
\bibitem [{\citenamefont {Van~Mieghem}(2014)}]{van2014performance}%
  \BibitemOpen
  \bibfield  {author} {\bibinfo {author} {\bibfnamefont {P.}~\bibnamefont
  {Van~Mieghem}},\ }\href@noop {} {\emph {\bibinfo {title} {Performance
  analysis of complex networks and systems}}}\ (\bibinfo  {publisher}
  {Cambridge University Press},\ \bibinfo {year} {2014})\BibitemShut {NoStop}%
\bibitem [{\citenamefont {Van~Mieghem}(2011)}]{VanMieghem2011Nintertwined}%
  \BibitemOpen
  \bibfield  {author} {\bibinfo {author} {\bibfnamefont {P.}~\bibnamefont
  {Van~Mieghem}},\ }\bibfield  {title} {\bibinfo {title} {{The N-intertwined
  SIS epidemic network model}},\ }\href@noop {} {\bibfield  {journal} {\bibinfo
   {journal} {Computing}\ }\textbf {\bibinfo {volume} {93}},\ \bibinfo {pages}
  {147} (\bibinfo {year} {2011})}\BibitemShut {NoStop}%
\bibitem [{\citenamefont {Lajmanovich}\ and\ \citenamefont
  {Yorke}(1976)}]{lajmanovich1976deterministic}%
  \BibitemOpen
  \bibfield  {author} {\bibinfo {author} {\bibfnamefont {A.}~\bibnamefont
  {Lajmanovich}}\ and\ \bibinfo {author} {\bibfnamefont {J.~A.}\ \bibnamefont
  {Yorke}},\ }\bibfield  {title} {\bibinfo {title} {A deterministic model for
  gonorrhea in a nonhomogeneous population},\ }\href@noop {} {\bibfield
  {journal} {\bibinfo  {journal} {Mathematical Biosciences}\ }\textbf {\bibinfo
  {volume} {28}},\ \bibinfo {pages} {221} (\bibinfo {year} {1976})}\BibitemShut
  {NoStop}%
\bibitem [{Note2()}]{Note2}%
  \BibitemOpen
  \bibinfo {note} {NIMFA upper bounds the infection probability in the
  Markovian SIS process \cite {VanMieghem2008Virusspreadinnetworks}.
  Specifically, when $\tau _c > \tau > \tau _{c}^{(1)}$ NIMFA will not die out,
  but a Markovian SIS process will die out fast.}\BibitemShut {Stop}%
\bibitem [{\citenamefont {Khanafer}\ \emph {et~al.}(2016)\citenamefont
  {Khanafer}, \citenamefont {Ba{\c{s}}ar},\ and\ \citenamefont
  {Gharesifard}}]{khanafer2016stability}%
  \BibitemOpen
  \bibfield  {author} {\bibinfo {author} {\bibfnamefont {A.}~\bibnamefont
  {Khanafer}}, \bibinfo {author} {\bibfnamefont {T.}~\bibnamefont
  {Ba{\c{s}}ar}},\ and\ \bibinfo {author} {\bibfnamefont {B.}~\bibnamefont
  {Gharesifard}},\ }\bibfield  {title} {\bibinfo {title} {{Stability of
  epidemic models over directed graphs: A positive systems approach}},\
  }\href@noop {} {\bibfield  {journal} {\bibinfo  {journal} {Automatica}\
  }\textbf {\bibinfo {volume} {74}},\ \bibinfo {pages} {126} (\bibinfo {year}
  {2016})}\BibitemShut {NoStop}%
\bibitem [{Note3()}]{Note3}%
  \BibitemOpen
  \bibinfo {note} {An Erd\H {o}s-R\'enyi random graph (ER graph) $G_p(N)$ is
  characterized by the link between each pair of the $N$ nodes existing with
  probability $p$, independent of any other link (see, e.g. \cite
  {van2014performance}).}\BibitemShut {Stop}%
\bibitem [{Note4()}]{Note4}%
  \BibitemOpen
  \bibinfo {note} {The Markovian SIS process average is conditioned on
  non-extinction}\BibitemShut {NoStop}%
\bibitem [{\citenamefont {Mossong}\ \emph {et~al.}(2008)\citenamefont
  {Mossong}, \citenamefont {Hens}, \citenamefont {Jit}, \citenamefont
  {Beutels}, \citenamefont {Auranen}, \citenamefont {Mikolajczyk},
  \citenamefont {Massari}, \citenamefont {Salmaso}, \citenamefont {Tomba},
  \citenamefont {Wallinga} \emph {et~al.}}]{mossong2008social}%
  \BibitemOpen
  \bibfield  {author} {\bibinfo {author} {\bibfnamefont {J.}~\bibnamefont
  {Mossong}}, \bibinfo {author} {\bibfnamefont {N.}~\bibnamefont {Hens}},
  \bibinfo {author} {\bibfnamefont {M.}~\bibnamefont {Jit}}, \bibinfo {author}
  {\bibfnamefont {P.}~\bibnamefont {Beutels}}, \bibinfo {author} {\bibfnamefont
  {K.}~\bibnamefont {Auranen}}, \bibinfo {author} {\bibfnamefont
  {R.}~\bibnamefont {Mikolajczyk}}, \bibinfo {author} {\bibfnamefont
  {M.}~\bibnamefont {Massari}}, \bibinfo {author} {\bibfnamefont
  {S.}~\bibnamefont {Salmaso}}, \bibinfo {author} {\bibfnamefont {G.~S.}\
  \bibnamefont {Tomba}}, \bibinfo {author} {\bibfnamefont {J.}~\bibnamefont
  {Wallinga}}, \emph {et~al.},\ }\bibfield  {title} {\bibinfo {title} {Social
  contacts and mixing patterns relevant to the spread of infectious diseases},\
  }\href@noop {} {\bibfield  {journal} {\bibinfo  {journal} {PLoS medicine}\
  }\textbf {\bibinfo {volume} {5}},\ \bibinfo {pages} {e74} (\bibinfo {year}
  {2008})}\BibitemShut {NoStop}%
\bibitem [{\citenamefont {Verelst}\ \emph {et~al.}(2021)\citenamefont
  {Verelst}, \citenamefont {Hermans}, \citenamefont {Vercruysse}, \citenamefont
  {Gimma}, \citenamefont {Coletti}, \citenamefont {Backer}, \citenamefont
  {Wong}, \citenamefont {Wambua}, \citenamefont {van Zandvoort}, \citenamefont
  {Willem} \emph {et~al.}}]{verelst2021socrates}%
  \BibitemOpen
  \bibfield  {author} {\bibinfo {author} {\bibfnamefont {F.}~\bibnamefont
  {Verelst}}, \bibinfo {author} {\bibfnamefont {L.}~\bibnamefont {Hermans}},
  \bibinfo {author} {\bibfnamefont {S.}~\bibnamefont {Vercruysse}}, \bibinfo
  {author} {\bibfnamefont {A.}~\bibnamefont {Gimma}}, \bibinfo {author}
  {\bibfnamefont {P.}~\bibnamefont {Coletti}}, \bibinfo {author} {\bibfnamefont
  {J.~A.}\ \bibnamefont {Backer}}, \bibinfo {author} {\bibfnamefont {K.~L.}\
  \bibnamefont {Wong}}, \bibinfo {author} {\bibfnamefont {J.}~\bibnamefont
  {Wambua}}, \bibinfo {author} {\bibfnamefont {K.}~\bibnamefont {van
  Zandvoort}}, \bibinfo {author} {\bibfnamefont {L.}~\bibnamefont {Willem}},
  \emph {et~al.},\ }\bibfield  {title} {\bibinfo {title} {Socrates-comix: a
  platform for timely and open-source contact mixing data during and in between
  covid-19 surges and interventions in over 20 european countries},\
  }\href@noop {} {\bibfield  {journal} {\bibinfo  {journal} {BMC medicine}\
  }\textbf {\bibinfo {volume} {19}},\ \bibinfo {pages} {1} (\bibinfo {year}
  {2021})}\BibitemShut {NoStop}%
\bibitem [{Note5()}]{Note5}%
  \BibitemOpen
  \bibinfo {note} {As an extension of our prediction method in this application
  setting, one could consider not only the most recent graph $G_{m-1}$ to be
  known, but the whole graph sequence $G_1, ..., G_{m-1}$.}\BibitemShut {Stop}%
\bibitem [{\citenamefont {Prasse}\ \emph {et~al.}(2021)\citenamefont {Prasse},
  \citenamefont {Devriendt},\ and\ \citenamefont
  {Van~Mieghem}}]{prasse2021clustering}%
  \BibitemOpen
  \bibfield  {author} {\bibinfo {author} {\bibfnamefont {B.}~\bibnamefont
  {Prasse}}, \bibinfo {author} {\bibfnamefont {K.}~\bibnamefont {Devriendt}},\
  and\ \bibinfo {author} {\bibfnamefont {P.}~\bibnamefont {Van~Mieghem}},\
  }\bibfield  {title} {\bibinfo {title} {{Clustering for epidemics on networks:
  A geometric approach}},\ }\href@noop {} {\bibfield  {journal} {\bibinfo
  {journal} {Chaos: an interdisciplinary journal of nonlinear science}\
  }\textbf {\bibinfo {volume} {31}},\ \bibinfo {pages} {063115} (\bibinfo
  {year} {2021})}\BibitemShut {NoStop}%
\bibitem [{\citenamefont {Van~Mieghem}\ and\ \citenamefont
  {Bovenkamp}(2013)}]{Van2013non}%
  \BibitemOpen
  \bibfield  {author} {\bibinfo {author} {\bibfnamefont {P.}~\bibnamefont
  {Van~Mieghem}}\ and\ \bibinfo {author} {\bibfnamefont {R.~v.~d.}\
  \bibnamefont {Bovenkamp}},\ }\bibfield  {title} {\bibinfo {title}
  {Non-{M}arkovian infection spread dramatically alters the
  susceptible-infected-susceptible epidemic threshold in networks},\
  }\href@noop {} {\bibfield  {journal} {\bibinfo  {journal} {Physical review
  letters}\ }\textbf {\bibinfo {volume} {110}},\ \bibinfo {pages} {108701}
  (\bibinfo {year} {2013})}\BibitemShut {NoStop}%
\bibitem [{\citenamefont {Prasse}\ and\ \citenamefont
  {Van~Mieghem}(2020)}]{prasse2020time}%
  \BibitemOpen
  \bibfield  {author} {\bibinfo {author} {\bibfnamefont {B.}~\bibnamefont
  {Prasse}}\ and\ \bibinfo {author} {\bibfnamefont {P.}~\bibnamefont
  {Van~Mieghem}},\ }\bibfield  {title} {\bibinfo {title} {{Time-dependent
  solution of the NIMFA equations around the epidemic threshold}},\ }\href@noop
  {} {\bibfield  {journal} {\bibinfo  {journal} {Journal of mathematical
  biology}\ }\textbf {\bibinfo {volume} {81}},\ \bibinfo {pages} {1299}
  (\bibinfo {year} {2020})}\BibitemShut {NoStop}%
\bibitem [{\citenamefont {Gr\"{o}nwall}(1919)}]{gronwall1919note}%
  \BibitemOpen
  \bibfield  {author} {\bibinfo {author} {\bibfnamefont {T.~H.}\ \bibnamefont
  {Gr\"{o}nwall}},\ }\bibfield  {title} {\bibinfo {title} {Note on the
  derivatives with respect to a parameter of the solutions of a system of
  differential equations},\ }\href@noop {} {\bibfield  {journal} {\bibinfo
  {journal} {Annals of Mathematics}\ ,\ \bibinfo {pages} {292}} (\bibinfo
  {year} {1919})}\BibitemShut {NoStop}%
\bibitem [{\citenamefont {Biggerstaff}\ \emph {et~al.}(2014)\citenamefont
  {Biggerstaff}, \citenamefont {Cauchemez}, \citenamefont {Reed}, \citenamefont
  {Gambhir},\ and\ \citenamefont {Finelli}}]{biggerstaff2014estimates}%
  \BibitemOpen
  \bibfield  {author} {\bibinfo {author} {\bibfnamefont {M.}~\bibnamefont
  {Biggerstaff}}, \bibinfo {author} {\bibfnamefont {S.}~\bibnamefont
  {Cauchemez}}, \bibinfo {author} {\bibfnamefont {C.}~\bibnamefont {Reed}},
  \bibinfo {author} {\bibfnamefont {M.}~\bibnamefont {Gambhir}},\ and\ \bibinfo
  {author} {\bibfnamefont {L.}~\bibnamefont {Finelli}},\ }\bibfield  {title}
  {\bibinfo {title} {Estimates of the reproduction number for seasonal,
  pandemic, and zoonotic influenza: a systematic review of the literature},\
  }\href@noop {} {\bibfield  {journal} {\bibinfo  {journal} {BMC infectious
  diseases}\ }\textbf {\bibinfo {volume} {14}},\ \bibinfo {pages} {1} (\bibinfo
  {year} {2014})}\BibitemShut {NoStop}%
\bibitem [{Note6()}]{Note6}%
  \BibitemOpen
  \bibinfo {note} {We denote Unif$(a,b)$ for the continuous uniform
  distribution on the interval $[a,b]$ and Unif$\{a,b\}$ for the discrete
  uniform distribution on the set $\{a, a+ 1,\protect \dots , b-1,
  b\}$}\BibitemShut {NoStop}%
\end{thebibliography}%

\end{document}